\documentclass[11pt, reqno, twoside, makeidx]{amsart}
\usepackage[margin=2.8cm, marginparsep=0.2cm, marginparwidth=2.4cm]{geometry}
\usepackage{graphicx}
\usepackage{amsmath}
\usepackage{amssymb}
\usepackage{amsthm}
\usepackage{mathtools}
\usepackage{mathrsfs}
\usepackage{bm}
\usepackage{slashed}
\usepackage{enumitem}
\usepackage{hyperref}
\usepackage{amscd}
\usepackage{cases}
\usepackage{tikz-cd}

\newtheorem{theorem}{Theorem}[section]
\newtheorem{lemma}[theorem]{Lemma}

\newtheorem{proposition}[theorem]{Proposition}
\newtheorem{question}[theorem]{Question}

\newtheorem{corollary}[theorem]{Corollary}
\newtheorem{claim}[theorem]{Claim}

\theoremstyle{definition}
\newtheorem{remark}[theorem]{Remark}
\newtheorem{definition}[theorem]{Definition}
\newtheorem{example}[theorem]{Example}
\newtheorem{construction}[theorem]{Construction}

\newtheorem{fact}[theorem]{Fact}

\newcommand{\C}{\mathbb{C}}
\newcommand{\R}{\mathbb{R}}
\newcommand{\Q}{\mathbb{Q}}
\newcommand{\Z}{\mathbb{Z}}
\newcommand{\SO}{\operatorname{SO}}
\newcommand{\Spin}{\operatorname{Spin}}

\newcommand{\Pin}{\operatorname{Pin}}
\newcommand{\KO}{\widetilde{KO}}
\newcommand{\K}{\widetilde{K}}
\newcommand{\pt}{\mathrm{pt}}
\newcommand{\s}{\mathfrak{s}}
\newcommand{\SWF}{\mathit{SWF}}
\newcommand{\Hty}{\mathcal{H}}
\newcommand{\ko}{\mathit{\kappa o}}

\title{ON INTERSECTION FORMS OF SPIN $4$-MANIFOLDS WITH AR HOMOLOGY SPHERE BOUNDARY}
\author{Eiichiro Hakamada}

\begin{document}

\begin{abstract}
In this paper, building on the work of Dai--Sasahira--Stoffregen, we derive constraints on the intersection forms of smooth, compact, indefinite, spin $4$-manifolds bounded by almost-rational plumbed homology spheres (AR homology spheres) using two methods. First, we use the constraints on the existence of $\Pin(2)$-equivariant maps between representation spheres, established by Hopkins--Lin--Shi--Xu in their proof of the $10/8 + 4$-inequality. Second, we compute Lin's $\ko$-invariant for AR homology spheres. Combined with Lin's version of the relative $10/8$-inequality, this computation yields constraints different from those based on Hopkins--Lin--Shi--Xu. As applications, we give an alternative proof of the non-sliceness of the $(2n, 1)$-cables of the figure-eight knot for even $n$---a result originally proved by Kang, Park, and Taniguchi via real Seiberg--Witten theory---by showing within ordinary Seiberg--Witten theory that certain cyclic branched covers of them are nontrivial in the homology cobordism group. We also obtain new lower bounds on the minimal genus of surfaces in various $4$-manifolds bounded by knots whose double branched covers are AR homology spheres. Moreover, we provide an obstruction to smooth $4$-manifold decompositions related to the $11/8$-conjecture.
\end{abstract}

\maketitle

\section{Introduction}
Let $X$ be a smooth, oriented, spin $4$-manifold. The question of which intersection forms can be realized by such manifolds has been studied for many years. If $X$ is closed and definite, then its intersection form must be trivial by Donaldson's diagonalizability theorem \cite{Don83, Don87}. In contrast to Freedman's theorem \cite{Fre82} in the topological category, this theorem implies that the smooth structure imposes constraints on the intersection form. Suppose that $X$ is closed and indefinite. By reversing the orientation of $X$ if necessary, we may assume that $\sigma(X) \le 0$. Since the intersection form of $X$ is even and unimodular, the algebraic classification of quadratic forms \cite{Ser73} implies that it can be expressed as
\[
p(-E_8) \oplus q \left( \begin{smallmatrix} 0 & 1 \\ 1 & 0 \end{smallmatrix} \right)
\]
for some integers $p \ge 0$ and $q > 0$. Note that $p$ is even by Rokhlin's theorem. In this setting, Furuta's $10/8$-theorem \cite{Fur01} provides the following inequality:
\[
q \ge p + 1.
\]
This inequality was obtained by extracting, via equivariant $K$-theory, a necessary condition for the existence of a suitable $\Pin(2)$-equivariant stable map between representation spheres
\[
S^{\frac{1}{2}p\mathbb{H}} \longrightarrow S^{q\tilde{\R}}
\]
arising from a finite-dimensional approximation of the Seiberg--Witten equations on $X$. For $p \ge 4$, Hopkins--Lin--Shi--Xu \cite{HLSX22} obtained a stronger constraint by giving a necessary and sufficient condition for the existence of such a map, which is stated as follows:
\[
q \ge  
\left\{
\begin{array}{l@{\quad}ll}
    p + 2 & \text{if } p \equiv 2, 4, 10, 12 & \pmod{16} \\
    p + 3 & \text{if } p \equiv 6, 8, 14     & \pmod{16} \\
    p + 4 & \text{if } p \equiv 16           & \pmod{16}.
\end{array}
\right.
\]

On the other hand, Manolescu \cite{Man14} generalized Furuta's $10/8$-theorem to manifolds with boundary by defining the $\kappa$-invariant. Suppose that $X$ is a smooth, spin cobordism from an integral homology $3$-sphere $Y_0$ to an integral homology $3$-sphere $Y_1$, with intersection form $p(-E_8) \oplus q \left( \begin{smallmatrix} 0 & 1 \\ 1 & 0 \end{smallmatrix} \right)$. Then the constraint derived from the $\kappa$-invariant is given as follows:
\[
\kappa(Y_1) + q \ge \kappa(Y_0) + p - 1.
\]
Furthermore, if $X$ is a smooth, compact, indefinite, spin $4$-manifold whose boundary $\partial X = Y$ is an integral homology $3$-sphere, Manolescu derived the following constraint:
\[
q \ge p + 1 - \kappa(Y).
\]
Subsequently, Lin \cite{Lin15} improved these constraints using equivariant $KO$-theory, introducing the $\ko$-invariant.

To derive such constraints, it is effective to analyze the Seiberg--Witten Floer stable homotopy types of the boundary manifolds. In a recent breakthrough, Dai--Sasahira--Stoffregen \cite{DSS26} showed that the Seiberg--Witten Floer stable homotopy type of an AR homology sphere can be calculated combinatorially from its lattice homology. Here, an AR homology sphere stands for an ``almost-rational plumbed homology sphere,'' a class of manifolds which includes all Seifert fibered rational homology spheres with base orbifold $S^2$.

By combining this calculation with the method of Hopkins--Lin--Shi--Xu \cite{HLSX22}, we derive new constraints on the intersection forms of smooth, compact, indefinite, spin $4$-manifolds bounded by AR homology spheres. In what follows, we denote the Neumann--Siebenmann invariant by $\bar{\mu}$. Our first main result is stated as follows:

\begin{theorem}\label{thm:main_inequality}
    Let $Y$ be an AR homology sphere. Suppose that $Y$ is an integral homology sphere. Then, for any smooth, compact, indefinite, spin $4$-manifold $X$ with boundary $\partial X = Y$ and intersection form $p (-E_8) \oplus q \left( \begin{smallmatrix} 0 & 1 \\ 1 & 0 \end{smallmatrix} \right)$ $(p \ge 0, q > 0)$ such that $p + \bar{\mu}(Y) \ge 4$, the following inequalities hold:
    \begin{enumerate}[label=(\roman*)]
        \item If $-\bar{\mu}(Y) = \delta(Y)$,
        \[
        \renewcommand{\arraystretch}{1.1}
        q \ge  
        \left\{
        \begin{array}{l@{\quad}ll}
        p + \bar{\mu}(Y) + 2 & \text{if } p + \bar{\mu}(Y) \equiv 2, 4, 10, 12 &\pmod{16} \\
        p + \bar{\mu}(Y) + 3 & \text{if } p + \bar{\mu}(Y) \equiv 6, 8, 14 &\pmod{16} \\
        p + \bar{\mu}(Y) + 4 & \text{if } p + \bar{\mu}(Y) \equiv 16 &\pmod{16}.
        \end{array}
        \right.
        \]
        \item If $-\bar{\mu}(Y) < \delta(Y)$,
        \[
        \renewcommand{\arraystretch}{1.1}
        q \ge  
        \left\{
        \begin{array}{l@{\quad}ll}
        p + \bar{\mu}(Y) + 1 & \text{if } p + \bar{\mu}(Y) \equiv 2, 4, 10, 12 &\pmod{16} \\
        p + \bar{\mu}(Y) + 2 & \text{if } p + \bar{\mu}(Y) \equiv 6, 8, 14 &\pmod{16} \\
        p + \bar{\mu}(Y) + 3 & \text{if } p + \bar{\mu}(Y) \equiv 16 &\pmod{16}.
        \end{array}
        \right.
        \]
    \end{enumerate}
\end{theorem}

This result can be generalized to connected sums as follows:

\begin{theorem}\label{thm:conn_sum}
    Let $Y_1, \dots, Y_n$ be AR homology spheres. Suppose that each $Y_i$ is an integral homology sphere. Let $Y = \mathop{\#}_{i=1}^n Y_i$. Then for any smooth, compact, indefinite, spin $4$-manifold $X$ with boundary $\partial X = Y$ and intersection form $p(-E_8) \oplus q \left(\begin{smallmatrix} 0 & 1 \\ 1 & 0 \end{smallmatrix}\right)$ $(p \ge 0, q > 0)$ such that $p + \sum_{i=1}^n \bar{\mu}(Y_i) \ge 4$, the following inequality holds:
    \[
    \renewcommand{\arraystretch}{1.2}
    q \ge 
    \left\{
    \begin{array}{l@{\quad}ll}
    p + \sum_{i=1}^n \bar{\mu}(Y_i) + 2 - m & \text{if } p + \sum_{i=1}^n \bar{\mu}(Y_i) \equiv 2, 4, 10, 12 &\pmod{16} \\
    p + \sum_{i=1}^n \bar{\mu}(Y_i) + 3 - m & \text{if } p + \sum_{i=1}^n \bar{\mu}(Y_i) \equiv 6, 8, 14 &\pmod{16} \\
    p + \sum_{i=1}^n \bar{\mu}(Y_i) + 4 - m & \text{if } p + \sum_{i=1}^n \bar{\mu}(Y_i) \equiv 16 &\pmod{16}
    \end{array}
    \right.
    \]
    where $m$ is the number of components $Y_i$ satisfying $-\bar{\mu}(Y_i) < \delta(Y_i)$.
\end{theorem}

\begin{remark}
    In general, for any AR homology sphere $Y$ with spin structure $\s$, the Neumann--Siebenmann invariant $\bar{\mu}(Y, \s) \in \Q$ is congruent to the Rokhlin invariant $\mu(Y, \s) \pmod{2\Z}$. Since the Rokhlin invariant is additive under connected sums, if a connected sum of AR homology spheres with spin structures $(Y, \s) = \mathop{\#}_{i=1}^n (Y_i, \s_i)$ bounds a spin $4$-manifold $X$, then $-\frac{1}{8}\sigma(X) + \sum_{i=1}^n \bar{\mu}(Y_i, \s_i)$ is an even integer.
\end{remark}

To derive these constraints, we use a morphism from the Seiberg--Witten Floer spectrum to a representation sphere (or in the opposite direction), which we call a \emph{spectrum cap}. By composing the relative Bauer--Furuta invariant with a spectrum cap, we reduce the problem to the existence of a suitable equivariant map between representation spheres, just as in the case of closed $4$-manifolds. In particular, we do not cap off the boundary to form an actual closed $4$-manifold; rather, we perform this capping operation at the level of Bauer--Furuta invariants. In fact, it remains unknown whether there exists a $4$-manifold whose relative Bauer--Furuta invariant realizes such a spectrum cap.

Another method for obtaining constraints on intersection forms involves using $\kappa$-invariants or $\ko$-invariants. In \cite{DSS26}, the $\kappa$-invariants of all AR homology spheres were calculated. Let $Y$ be an AR homology sphere equipped with a spin structure $\s$. Then we have:
\[
    \kappa(Y, \s) = 
    \begin{cases}
        -\bar{\mu}(Y, \s) & \text{if } -\bar{\mu}(Y, \s) = \delta(Y, \s) \\
        -\bar{\mu}(Y, \s) + 2 & \text{if } -\bar{\mu}(Y, \s) < \delta(Y, \s).
    \end{cases}
\]
Inspired by this result, we compute the $\ko$-invariants for all AR homology spheres, generalizing Lin's calculations for the Brieskorn spheres $\Sigma(2, 3, k)$ \cite[Theorem 1.9]{Lin15}.

\begin{theorem}\label{thm:ko_calculation}
    Let $Y$ be an AR homology sphere with a spin structure $\s$. Then the values of $\ko_i(Y, \s) + \frac{1}{2}\bar{\mu}(Y, \s)$ for $i \in \Z/8$ are given as the following bold numbers:
    \[
    \setlength{\arraycolsep}{8pt}
    \begin{array}{l|l|cccccccc}
        \multicolumn{2}{c|}{i \in \Z/8} & 0 & 1 & 2 & 3 & 4 & 5 & 6 & 7 \\ \hline
        \lfloor \frac{\bar{\mu}(Y, \s)}{2} \rfloor \text{ is even} & -\bar{\mu}(Y, \s) = \delta(Y, \s)     & \bm{0} & \bm{0} & \bm{0} &  \bm{0} &  \bm{0} &  \bm{0} & \bm{0} & \bm{0} \\
                                                                   & -\bar{\mu}(Y, \s) = \delta(Y, \s) - 1 & \bm{1} & \bm{1} & \bm{1} &  \bm{0} &  \bm{0} &  \bm{0} & \bm{0} & \bm{0} \\
                                                                   & -\bar{\mu}(Y, \s) \le \delta(Y, \s) - 2 & \bm{1} & \bm{1} & \bm{1} &  \bm{0} &  \bm{1} &  \bm{0} & \bm{0} & \bm{0} \\ \hline
        \lfloor \frac{\bar{\mu}(Y, \s)}{2} \rfloor \text{ is odd}  & -\bar{\mu}(Y, \s) = \delta(Y, \s)     & \bm{1} & \bm{1} & \bm{0} & \bm{-1} & \bm{-1} & \bm{-1} & \bm{0} & \bm{1} \\
                                                                   & -\bar{\mu}(Y, \s) = \delta(Y, \s) - 1 & \bm{1} & \bm{1} & \bm{0} & \bm{-1} &  \bm{0} &  \bm{0} & \bm{1} & \bm{1} \\
                                                                   & -\bar{\mu}(Y, \s) \le \delta(Y, \s) - 2 & \bm{2} & \bm{1} & \bm{0} & \bm{-1} &  \bm{0} &  \bm{0} & \bm{1} & \bm{1}
    \end{array}
    \]
\end{theorem}

Applying Lin's version of the $10/8$-type inequality \cite[Corollary 1.12]{Lin15}, we obtain the following constraint from Theorem~\ref{thm:ko_calculation}:

\begin{theorem}\label{thm:main_inequality_ko}
    Let $Y$ be an AR homology sphere. Suppose that $Y$ is an integral homology sphere. Then for any smooth, compact, indefinite, spin $4$-manifold $X$ with boundary $\partial X = Y$ and intersection form $p (-E_8) \oplus q \left( \begin{smallmatrix} 0 & 1 \\ 1 & 0 \end{smallmatrix} \right)$ $(p \ge 0, q > 0)$, the following inequalities derived from the $\ko$-invariants hold:
    \begin{enumerate}[label=(\roman*)]
        \item If $-\bar{\mu}(Y) = \delta(Y)$,
        \[
        \renewcommand{\arraystretch}{1.1}
        q \ge  
        \left\{
        \begin{array}{l@{\quad}ll}
        p + \bar{\mu}(Y) + 1 & \text{if } p + \bar{\mu}(Y) \equiv 0, 2 &\pmod{8} \\
        p + \bar{\mu}(Y) + 2 & \text{if } p + \bar{\mu}(Y) \equiv 4 &\pmod{8} \\
        p + \bar{\mu}(Y) + 3 & \text{if } p + \bar{\mu}(Y) \equiv 6 &\pmod{8}.
        \end{array}
        \right.
        \]
        \item If $-\bar{\mu}(Y) = \delta(Y) - 1$,
        \[
        \renewcommand{\arraystretch}{1.1}
        q \ge  
        \left\{
        \begin{array}{l@{\quad}ll}
        p + \bar{\mu}(Y) + 0 & \text{if } p + \bar{\mu}(Y) \equiv 2 &\pmod{8} \\
        p + \bar{\mu}(Y) + 1 & \text{if } p + \bar{\mu}(Y) \equiv 0, 4 &\pmod{8} \\
        p + \bar{\mu}(Y) + 2 & \text{if } p + \bar{\mu}(Y) \equiv 6 &\pmod{8}.
        \end{array}
        \right.
        \]
        \item If $-\bar{\mu}(Y) \le \delta(Y) - 2$,
        \[
        \renewcommand{\arraystretch}{1.1}
        q \ge  
        \left\{
        \begin{array}{l@{\quad}ll}
        p + \bar{\mu}(Y) + 0 & \text{if } p + \bar{\mu}(Y) \equiv 0, 2 &\pmod{8} \\
        p + \bar{\mu}(Y) + 1 & \text{if } p + \bar{\mu}(Y) \equiv 4 &\pmod{8} \\
        p + \bar{\mu}(Y) + 2 & \text{if } p + \bar{\mu}(Y) \equiv 6 &\pmod{8}.
        \end{array}
        \right.
        \]
    \end{enumerate}
\end{theorem}

\begin{remark}\label{rem:correction_term}
    We define the correction term $c(p, \bar{\mu}(Y), \delta(Y))$ to be the bold values given in the tables below.
    
    \smallskip\noindent
    (i) Values for Theorem~\ref{thm:main_inequality}:
    \[
    \setlength{\arraycolsep}{12pt}
    \renewcommand{\arraystretch}{1.3}
    \begin{array}{c||c|c|c|c|c}
        p + \bar{\mu}(Y) \pmod 8 & \multicolumn{2}{c|}{0} & 2 & 4 & 6 \\ \cline{2-3}
         & 0 \ (\mathrm{mod}\ 16) & 8 \ (\mathrm{mod}\ 16) & & & \\ \hline\hline
        -\bar{\mu}(Y) = \delta(Y) & \bm{4} & \bm{3} & \bm{2} & \bm{2} & \bm{3} \\ \hline
        -\bar{\mu}(Y) < \delta(Y) & \bm{3} & \bm{2} & \bm{1} & \bm{1} & \bm{2}
    \end{array}
    \]
    
    \smallskip\noindent
    (ii) Values for Theorem~\ref{thm:main_inequality_ko}:
    \[
    \setlength{\arraycolsep}{16pt}
    \renewcommand{\arraystretch}{1.3}
    \begin{array}{c||c|c|c|c}
        p + \bar{\mu}(Y) \pmod 8 & 0 & 2 & 4 & 6 \\ \hline\hline
        -\bar{\mu}(Y) = \delta(Y) & \bm{1} & \bm{1} & \bm{2} & \bm{3} \\ \hline
        -\bar{\mu}(Y) = \delta(Y) - 1 & \bm{1} & \bm{0} & \bm{1} & \bm{2} \\ \hline
        -\bar{\mu}(Y) \le \delta(Y) - 2 & \bm{0} & \bm{0} & \bm{1} & \bm{2}
    \end{array}
    \]

    Using this notation, the constraints on the intersection forms obtained from Theorems~\ref{thm:main_inequality} and \ref{thm:main_inequality_ko} can be expressed in a unified manner as
    \[
        q \ge p + \bar{\mu}(Y) + c(p, \bar{\mu}(Y), \delta(Y)),
    \]
    where $c$ is determined by Table (i) for Theorem~\ref{thm:main_inequality} (which requires $p + \bar{\mu}(Y) \ge 4$) and by Table (ii) for Theorem~\ref{thm:main_inequality_ko}. 
    As shown in the tables, under the condition $p + \bar{\mu}(Y) \ge 4$, the constraint derived from Theorem~\ref{thm:main_inequality} is stronger than or equal to that from Theorem~\ref{thm:main_inequality_ko}.
\end{remark}

Similarly, the constraints derived from Manolescu's $\kappa$-invariant can be expressed using the same unified notation as in Remark~\ref{rem:correction_term}. The corresponding values of the correction term $c(p, \bar{\mu}(Y), \delta(Y))$ are given in the following table:
\[
\setlength{\arraycolsep}{12pt}
\renewcommand{\arraystretch}{1.3}
\begin{array}{c||c|c|c|c}
    p + \bar{\mu}(Y) \pmod 8 & 0 & 2 & 4 & 6 \\ \hline\hline
    -\bar{\mu}(Y) = \delta(Y) & \bm{1} & \bm{1} & \bm{1} & \bm{1} \\ \hline
    -\bar{\mu}(Y) < \delta(Y) & \bm{-1} & \bm{-1} & \bm{-1} & \bm{-1}
\end{array}
\]

Additionally, it was shown by Lin \cite{Lin15} and Ue \cite{Ue22} via orbifold techniques that for any Seifert integral homology sphere $Y$ bounding a spin $4$-manifold $X$ with intersection form $p(-E_8) \oplus q \left(\begin{smallmatrix} 0 & 1 \\ 1 & 0 \end{smallmatrix}\right)$ ($p \ge 0$, $q > 0$), if $p + \bar{\mu}(Y) > 0$ and $p + \bar{\mu}(Y) \equiv 0 \pmod 8$, then $q \ge p + \bar{\mu}(Y) + 2$. As seen from the table above, the constraint provided by Theorem~\ref{thm:main_inequality} is sharper.

Moreover, the constraints from Lin \cite[Theorem 7]{Lin17} can be specialized to our setting by applying them to a cobordism from $Y_0 = S^3$ to $Y_1 = Y$, an AR homology sphere. Substituting $b^+(X) = q$ and $b^-(X) = 8p + q$, and using the relation $\beta(Y, \s) = -\bar{\mu}(Y, \s)$ (see \cite[Section 7.1]{DSS26}), Lin's theorem translates into the following constraints:
\begin{enumerate}[label=(\roman*)]
    \item If $q = 1$, then $p + \bar{\mu}(Y, \s) \le 0$.
    \item If $q = 2$, then $p - \alpha(Y, \s) \le 0$.
\end{enumerate}

Setting $q = 1$ in Theorem~\ref{thm:main_inequality_ko}, we obtain $p + \bar{\mu}(Y) \le 1$. Combining this with the fact that $p + \bar{\mu}(Y)$ is an even integer, we deduce that $p + \bar{\mu}(Y) \le 0$. 
Similarly, setting $q = 2$ in Theorem~\ref{thm:main_inequality_ko} yields $p + \bar{\mu}(Y) \le 1$ when $-\bar{\mu}(Y) = \delta(Y)$, and $p + \bar{\mu}(Y) \le 2$ when $-\bar{\mu}(Y) < \delta(Y)$. 
In the latter case, by \cite[Corollary 1.2]{Sto20} and the fact that the modulo $2$ reduction of $\alpha(Y)$ coincides with the Rokhlin invariant, we have $\alpha(Y) \ge -\bar{\mu}(Y) + 2$. This implies that our constraint $p + \bar{\mu}(Y) \le 2$ is sharper than the condition $p - \alpha(Y) \le 0$. Therefore, under our setting, the constraints derived from Theorem~\ref{thm:main_inequality_ko} are stronger than these previous results.

\begin{remark}\label{rem:rational}
    In all the above constraints, the assumption that $Y$ is an integral homology sphere is used only to ensure that the intersection form of the spin $4$-manifold $X$ bounded by $Y$ is unimodular. When generalizing to an arbitrary AR homology sphere $Y$, one must fix a spin structure $\mathfrak{s}$ on $Y$ and require that the spin structure on $X$ restricts to $\mathfrak{s}$ on the boundary. Furthermore, the assumption $\sigma(X) \le 0$ can be dropped, as it was imposed only to express the intersection form as $p(-E_8) \oplus q \left(\begin{smallmatrix} 0 & 1 \\ 1 & 0 \end{smallmatrix}\right)$. Under this setup, these constraints remain valid by replacing $p$ and $q$ with $-\frac{1}{8}\sigma(X)$ and $b^+(X)$, respectively.
\end{remark}

Using a capping-off argument similar to the proof of Theorem~\ref{thm:main_inequality}, we provide an alternative proof of the non-sliceness of the $(4n, 1)$-cables $E_{4n, 1}$ of the figure-eight knot $E$, which was originally proved by Kang, Park, and Taniguchi \cite{KPT25} for the broader case of $(2n, 1)$-cables via real Seiberg--Witten theory. More precisely, for any integer $k \ge 1$ and any odd integer $m \ge 1$, they proved that the $2^k$-fold branched cover $\Sigma_{2^k}(E_{2^k m, 1})$ does not bound any smooth $\Z_2$-equivariant $\Z_2$-homology $4$-ball with respect to the natural involution induced by composing the deck transformation $2^{k-1}$ times. In this paper, for the case $k \ge 2$, we show the stronger statement that $\Sigma_{2^k}(E_{2^k m, 1})$ does not bound any smooth spin rational homology $4$-ball.

\begin{theorem}\label{thm:non_trivial}
    For any integer $k \ge 2$ and any odd integer $m \ge 1$, the $2^k$-fold branched cover $\Sigma_{2^k}(E_{2^k m, 1})$ of $E_{2^k m, 1}$ does not bound any smooth spin rational homology $4$-ball. In particular, $E_{4n, 1}$ is not smoothly slice for any positive integer $n$.
\end{theorem}

\begin{remark}
    By Proposition~\ref{prop:hom_sph}, $\Sigma_{2^k}(E_{2^k m, 1})$ is an integral homology $3$-sphere and thus defines an element in the $3$-dimensional homology cobordism group $\Theta^3_{\Z}$. Although it is shown in \cite{KPT25} that $E_{2n, 1}$ has infinite order in the smooth knot concordance group for any $n \ge 1$, our method cannot determine whether $\Sigma_{2^k}(E_{2^k m, 1})$ has infinite order in $\Theta^3_{\Z}$ for $k \ge 2$. Nevertheless, we can show that the two-fold connected sum $2\Sigma_{2^k}(E_{2^k m, 1})$ (and, for $k \ge 4$, the three-fold connected sum $3\Sigma_{2^k}(E_{2^k m, 1})$) is nontrivial in $\Theta^3_{\Z}$ (Theorem~\ref{thm:non_trivial_conn_sum}).

    In contrast, in the case $k = m = 1$, the double-branched cover of the $(2, 1)$-cable of the figure-eight knot $\Sigma_2(E_{2, 1})$ bounds a smooth contractible $4$-manifold \cite[Remark~2.2]{DKMPS24}.
\end{remark}

We also provide another application concerning the smooth slice genus of knots. By applying the double-branched cover construction of \cite{KMT25} to our inequalities, we obtain the following genus bound:

\begin{corollary}\label{cor:knot_inequality}
    Let $K$ be a knot in $S^3$ whose double branched cover $\Sigma(K)$ is an AR homology sphere. Let $X$ be a smooth, compact $4$-manifold with boundary $\partial X = S^3$ such that $H_1(X; \Z) = 0$, and let $S \subset X$ be a connected, oriented, compact, properly embedded surface bounded by $K$. Suppose that $[S]$ is divisible by $2$ in $H_2(X; \Z)$ and $\operatorname{PD}(w_2(X)) \equiv [S]/2 \pmod 2$, and that the double branched cover $\Sigma(S)$ along $S$ satisfies $\sigma(\Sigma(S)) \le 0$ and $b^+(\Sigma(S)) > 0$. Then the following inequality holds:
    \[
    g(S) \ge -2b^+(X) - \frac{1}{4}\sigma(X) + \frac{5}{16}[S]^2 - \frac{5}{8}\sigma(K) + \bar{\mu}(\Sigma(K)) + c\big(p, \bar{\mu}(\Sigma(K)), \delta(\Sigma(K))\big)
    \]
    where $p = -\frac{1}{8}\sigma(\Sigma(S))$ and $c\big(p, \bar{\mu}(\Sigma(K)), \delta(\Sigma(K))\big)$ is the correction term given in Remark~\ref{rem:correction_term}.
\end{corollary}

As a special case of this corollary, we obtain the following result, which provides examples where equality is achieved in the constraints of Theorem~\ref{thm:main_inequality} and Theorem~\ref{thm:main_inequality_ko}.

\begin{example}\label{ex:CP^2_genus}
    Let $K \in \{T(3, 5), T(3, 7), T(3, 11)\}$ be a knot in $S^3$. Let $X$ be a smooth, compact $4$-manifold with boundary $\partial X = S^3$. Then the minimal genus among all connected, oriented, compact, properly embedded surfaces $S \subset X$ bounded by $K$ and representing the specified homology class is determined as follows:
    \begin{enumerate}
        \item If $X = \C P^2 \setminus \operatorname{int} D^4$ and $[S] = 6 \in H_2(X; \Z) \cong \Z$, then the minimal genus is equal to $g_4(K) + 10$.
        \item If $X = (\C P^2 \# \C P^2) \setminus \operatorname{int} D^4$ and $[S] = (6, 6) \in H_2(X; \Z) \cong \Z \oplus \Z$, then the minimal genus is equal to $g_4(K) + 20$.
    \end{enumerate}
\end{example}

\begin{remark}
    We note that the calculation of the $\ko$-invariants required for Example~\ref{ex:CP^2_genus} is covered by Lin's results on the Brieskorn spheres $\Sigma(2, 3, k)$ \cite[Theorem 1.9]{Lin15}.
\end{remark}

\begin{remark}\label{rem:closed_surface}
    The sharp lower bound necessary for Example~\ref{ex:CP^2_genus} can also be obtained by capping the knot using \cite{EG22} to reduce the problem to a closed surface, and then applying the bounds from \cite{HLSX22} and \cite{KMT25}. The details of this proof are provided in Section~\ref{subsec:apps}.
    Additionally, as stated in Proposition~\ref{prop:count_ex}, there exist torus knots that do not satisfy conditions (1) and (2) of Example~\ref{ex:CP^2_genus}.
\end{remark}

As a further application, we consider Bauer's strategy for proving the $11/8$-conjecture (see Manolescu \cite{Man14}). Suppose there exists a counterexample to the conjecture, namely a closed, smooth, spin $4$-manifold $X$ with intersection form $2r (-E_8) \oplus q \left( \begin{smallmatrix} 0 & 1 \\ 1 & 0 \end{smallmatrix} \right)$ such that $q < 3r$. By taking connected sums with $S^2 \times S^2$, we may assume that $q = 3r - 1$. If $\pi_1(X) = 1$, then by a theorem of Freedman and Taylor \cite{FT77} and the framework in \cite{Man14}, $X$ admits a decomposition
\begin{equation}\label{eq:X_decomp}
    X = X_1 \cup_{Y_1} X_2 \cup_{Y_2} \cdots \cup_{Y_{r - 1}} X_r
\end{equation}
such that:
\begin{itemize}
    \item $Y_i$ is an integral homology $3$-sphere for all $i$.
    \item For $1 \le i \le r - 1$, the manifold $X_i$ has intersection form $2 (-E_8) \oplus 3 \left( \begin{smallmatrix} 0 & 1 \\ 1 & 0 \end{smallmatrix} \right)$.
    \item The manifold $X_r$ has intersection form $2 (-E_8) \oplus 2 \left( \begin{smallmatrix} 0 & 1 \\ 1 & 0 \end{smallmatrix} \right)$.
\end{itemize}

Manolescu proved that such a decomposition cannot exist if all the homology spheres $Y_i$ are Floer $K_G$-split \cite[Theorem 1.6]{Man14}. Using our constraints on intersection forms, we establish an analogous non-existence result for AR homology spheres.

\begin{theorem}\label{thm:4-manifold_split}
    There exists no closed, smooth, spin $4$-manifold $X$ admitting a decomposition of the form~\eqref{eq:X_decomp} in which all the homology spheres $Y_i$ are AR homology spheres.
\end{theorem}

Since there exist infinitely many AR homology spheres that are not Floer $K_G$-split (Proposition~\ref{prop:AR_K_G}), Theorem~\ref{thm:4-manifold_split} is not a special case of \cite[Theorem 1.6]{Man14}.

\subsection*{Acknowledgments}
The author would like to express his deepest gratitude to his advisor, Hokuto Konno, for his detailed feedback and valuable suggestions on the manuscript, for many meaningful discussions, and for his continuous support and encouragement throughout this work. 
The author is also deeply grateful to Hirofumi Sasahira for his helpful comments and suggestions, especially regarding the application to splitting $4$-manifolds in Theorem~\ref{thm:4-manifold_split} and the construction of the $\Pin(2)$-lattice spectrum in Construction~\ref{con:lattice_spectrum}. Furthermore, the author wishes to thank Jianfeng Lin for fruitful conversations about $\ko$-invariants, particularly concerning Corollary~\ref{cor:sequence}; Motoo Tange for insightful discussions on knot applications; and Yoshihiro Fukumoto for explanations concerning orbifolds. Finally, the author is grateful to Masaki Taniguchi for invaluable discussions and advice regarding Theorem~\ref{thm:non_trivial}, for perceptive comments that led to Remark~\ref{rem:closed_surface} and Proposition~\ref{prop:count_ex}, and for suggesting the geometric perspective and question formulated in Remark~\ref{rem:philosophy} and Question~\ref{q:knot}.

\subsection*{AI use}
The author used generative AI tools (Gemini, ChatGPT, and Claude) to assist with searching the literature, checking computations, refining arguments, improving the writing, and creating figures. All mathematical ideas and the final content remain the author's own, and the author assumes full responsibility for this work.

\section{Seiberg--Witten Floer Spectrum and Relative Bauer--Furuta Invariant}

In this section, we will review the Seiberg--Witten Floer spectrum and the relative Bauer--Furuta invariant in the case of a rational homology $3$-sphere with a spin structure. See \cite{Man03} for the details.

\subsection{The $\Pin(2)$-equivariant stable homotopy category}

We introduce the $\Pin(2)$-equivariant stable homotopy category $\mathfrak{C}_{\Pin(2)}$ to define the $\Pin(2)$-equivariant Seiberg--Witten Floer spectrum. Let $G = \Pin(2)$.
\begin{definition}[{\cite[Section 3.1]{DSS26}}]\label{def:Pin(2)-category}
    Let $\mathfrak{C}_G$ be the category defined as follows.
\begin{itemize}
    \item The objects of $\mathfrak{C}_G$ are triples $(W, m, n)$. Here, $W$ is a pointed $G$-CW complex, $m \in \Z$, and $n \in \Q$.
    \item For objects $(W_0, m_0, n_0)$ and $(W_1, m_1, n_1)$ in $\mathfrak{C}_G$, the set of morphisms is defined as follows: when $n_0 - n_1 \in \Z$,
    \[
    \begin{split}
    &\mathrm{Mor}_{\mathfrak{C}_G}((W_0, m_0, n_0), (W_1, m_1, n_1)) \\
    &\quad = \lim_{p,q \to \infty} [\Sigma^{\tilde{\R}^p \oplus \mathbb{H}^q} W_0, \Sigma^{\tilde{\R}^{p+m_0-m_1} \oplus \mathbb{H}^{q+n_0-n_1}} W_1]^0_G
    \end{split}
    \]
    and when $n_0 - n_1 \notin \Z$, it is defined as
    \[
    \mathrm{Mor}_{\mathfrak{C}_G}((W_0, m_0, n_0), (W_1, m_1, n_1)) = \emptyset.
    \]
\end{itemize}
We call this category the \emph{$\Pin(2)$-equivariant stable homotopy category}.
\end{definition}
%\medskip

For a $G$-representation $F$ which is isomorphic to $\tilde{\R}^{a} \oplus \mathbb{H}^{b}$, we define the formal suspension:
\[
    \begin{split}
    &\Sigma^{F} (W, m, n) := (\Sigma^F W, m, n)\\
    &\Sigma^{-F} (W, m, n) := (\Sigma^{F^{S^1}} W, m + 2a, n+b)
    \end{split}
\]
where $F^{S^1}$ is the $S^1$-fixed subspace of $F$. For $q \in \Q$, we denote $(S^0, 0, q)$ as $S^{-q\mathbb{H}}$. Also, $(W, 0, 0)$ may sometimes be written simply as $W$.

\begin{remark}
    For such a $G$-representation $F$, $\Sigma^F$ and $\Sigma^{-F}$ define functors on $\mathfrak{C}_G$ that are mutually inverse up to natural isomorphism.
\end{remark}

\subsection{Seiberg--Witten Floer spectrum}

\begin{definition}
    Let $Z$ be a pointed $G$-CW complex and $l \in \Z_{\ge 0}$. We say that $Z$ is of type SWF at level $l$ if the following conditions are satisfied:
    \begin{itemize}
        \item The action of $G$ is free on the complement $Z - Z^{S^1}$.
        \item The $S^1$-fixed point set $Z^{S^1}$ is $G$-homotopy equivalent to the sphere $S^{l\tilde{\R}}$.
    \end{itemize}
\end{definition}

Let $Y$ be a rational homology $3$-sphere with a spin structure $\s$ and a metric $g$. Write $S$ for the associated spinor bundle on $Y$ and fix a flat spin connection $A_0$. We will define the Seiberg--Witten Floer spectrum for $(Y, \s)$ as an object of the category $\mathfrak{C}_G$. 

Now we consider the global Coulomb slice:
\[
V \coloneqq \ker (d^* :i\Omega^1(Y) \to i\Omega^0(Y)) \oplus \Gamma(S).
\]
The group $G$ acts on $\ker d^*$ as scalar multiplication through $G = S^1 \sqcup jS^1 \twoheadrightarrow \Z / 2 = \{\pm1\} \text{ (where } S^1 \mapsto 1 \text{ and } jS^1 \mapsto -1)$ and on $\Gamma(S)$ as quaternionic scalar multiplication. Let $l : V \to V$ be the self-adjoint first order elliptic operator defined by $l(a, \phi) = (*da, \slashed{D}_{A_0}(\phi))$ where $\slashed{D}_{A_0}$ is the Dirac operator associated to the connection $A_0$. For any $\lambda, \mu \in \R \ (\lambda < \mu)$, we denote by $V^\mu_\lambda$ the subspace of $V$ spanned by all eigenvectors whose eigenvalues are in $[\lambda, \mu)$. Since $l$ is an elliptic operator, $V^\mu_\lambda$ is a finite dimensional subspace and there exist $p, q \in \Z_{\ge 0}$ such that $V^\mu _\lambda \cong p \tilde{\R} \oplus q \mathbb{H}$, where $\tilde{\R}$ corresponds to the differential form component and $\mathbb{H}$ to the spinor component.

Following Khandhawit, Lin, and Sasahira \cite[Section 2]{KLS18}, we consider the gradient flow of the Chern--Simons--Dirac functional on the Coulomb slice $V$ equipped with an appropriate metric. By taking a finite-dimensional approximation of this vector field on $V^\mu_\lambda$ and multiplying it by a suitable bump function, we obtain a well-defined flow on a sufficiently large bounded region. Choosing $\lambda \ll 0$ and $\mu \gg 0$, we take the $G$-equivariant Conley index of this approximated flow on $V^\mu_\lambda$ to obtain a $G$-space $I^\mu_\lambda$ of type SWF at level $\dim_\R V^0_\lambda(\tilde{\R})$, where $V^0_\lambda(\tilde{\R}) = V^0_\lambda \cap \ker d^*$. Let $n(Y, \s, g)$ be the rational number defined by:
\[
n(Y, \s, g) \coloneqq \operatorname{ind}_{\C} \slashed{D}_{\hat{A}_0} + \frac{\sigma(X)}{8}.
\]
Here $X$ is a smooth, compact, spin Riemannian $4$-manifold with boundary $Y$. We assume that $X$ is equipped with a metric $\hat{g}$, a spin structure $\mathfrak{t}$, and a spin connection $\hat{A}_0$ such that they are cylindrical on a collar neighborhood of the boundary, and their restrictions to $Y$ coincide with $g$, $\s$, and $A_0$, respectively. Furthermore, $\slashed{D}_{\hat{A}_0}$ is the Dirac operator associated with the connection $\hat{A}_0$.

We define the $G$-equivariant Seiberg--Witten Floer spectrum $\SWF(Y, \s, g)$ to be the object $\Sigma^{-V^0_\lambda}(I^\mu_\lambda, 0, \frac{1}{2} n(Y, \s, g))$ of the category $\mathfrak{C}_G$, for fixed real numbers $\lambda, \mu$ with $\lambda \ll 0, \mu \gg 0$.

\begin{remark}
    The number $n(Y, \s, g)$ is independent of the choice of $X$ (including the metric and the spin structure on $X$) and the choice of the connection $\hat{A}_0$. Moreover, $\SWF(Y, \s, g)$ is independent of the metric $g$, the parameters $\lambda, \mu$, and the choice of $I^\mu_\lambda$ up to a canonical isomorphism in the category $\mathfrak{C}_G$. Therefore, $\SWF(Y, \s, g)$ may sometimes be written simply as $\SWF(Y, \s)$.
\end{remark}

\subsection{Relative Bauer--Furuta invariant}

Now we will define the relative Bauer--Furuta invariant as a morphism in the category $\mathfrak{C}_G$. Suppose $X$ is a spin cobordism between rational homology $3$-spheres $Y_0$ and $Y_1$ with $b_1(X) = 0$. Further, we assume $X$ is equipped with a metric $\hat{g}$ and a spin structure $\mathfrak{t}$, such that they are cylindrical on a collar neighborhood of the boundary, and their restrictions to $Y_i$ coincide with $g_i$ and $\s_i$, where $g_i$ and $\s_i$ are the metric and spin structure on $Y_i$, respectively.

\begin{definition}[{\cite[Definition 4.2]{Lin15}}]
    Let $(W_0, m_0, n_0)$ and $(W_1, m_1, n_1)$ be two objects in $\mathfrak{C}_G$, where $W_0$ and $W_1$ are spaces of type SWF at levels $l_0$ and $l_1$, respectively. A morphism $[\phi] \in \mathrm{Mor}_{\mathfrak{C}_G}((W_0, m_0, n_0), (W_1, m_1, n_1))$ is called \emph{admissible} if there exist sufficiently large integers $p, q$ and a representative $G$-map
    \[
    \phi_{p,q} : \Sigma^{\tilde{\R}^p \oplus \mathbb{H}^q} W_0 \to \Sigma^{\tilde{\R}^{p+m_0-m_1} \oplus \mathbb{H}^{q+n_0-n_1}} W_1
    \]
    such that one of the following two conditions is satisfied:
    \begin{itemize}
        \item $l_0 - m_0 < l_1 - m_1$ and the induced map on the $G$-fixed point set 
        \[
        (\phi_{p,q})^G : W_0^G \to W_1^G
        \] 
        is a homotopy equivalence.
        
        \item $l_0 - m_0 = l_1 - m_1$ and the induced map on the $S^1$-fixed point set 
        \[
        (\phi_{p,q})^{S^1} : \Sigma^{\tilde{\R}^p} W_0^{S^1} \to \Sigma^{\tilde{\R}^{p+m_0-m_1}} W_1^{S^1}
        \] 
        is a $G$-homotopy equivalence.
    \end{itemize}
\end{definition}

\begin{remark}
    The composition of two admissible morphisms is again admissible.
\end{remark}

The formulation of the relative Bauer--Furuta invariant by Manolescu \cite{Man03} can be summarized as follows.

\begin{theorem}\label{thm:rel_BF}
    By taking a finite-dimensional approximation of the Seiberg--Witten equations on $X$, we obtain an admissible morphism:
    \[
    \Phi_X(\mathfrak{t}, \hat{g}): \Sigma^{-\frac{\sigma(X)}{16} \mathbb{H}} \SWF(Y_0, \s_0, g_0) \to \Sigma^{b^+(X)\tilde{\R}}\SWF(Y_1, \s_1, g_1).
    \]
\end{theorem}

\begin{remark}
    The morphism $\Phi_X(\mathfrak{t}, \hat{g})$ is independent of the metric $\hat{g}$ up to canonical isomorphisms of the Seiberg--Witten Floer spectra of $Y_0$ and $ Y_1$. We sometimes denote this as:
    \[
    \Phi_X(\mathfrak{t}): \Sigma^{-\frac{\sigma(X)}{16} \mathbb{H}} \SWF(Y_0, \s_0) \to \Sigma^{b^+(X)\tilde{\R}}\SWF(Y_1, \s_1).
    \]
\end{remark}

\section{Seiberg--Witten Floer Stable Homotopy Type of AR Homology Spheres}

\subsection{Graded Roots and the Lattice Spectrum}

This section reviews the explicit combinatorial description of the SWF spectrum for almost-rational plumbed homology spheres, as recently established in \cite{DSS26}. This class of manifolds includes all Seifert fibered rational homology spheres with base orbifold $S^2$. We begin by recalling the formal definition of this class.

\begin{definition}
    A rational homology sphere $Y$ is called an \emph{almost-rational plumbed homology sphere} if it can be realized as the boundary of a $4$-manifold $X$ obtained by plumbing along an almost-rational graph (see \cite[Definition 8.1]{Nem05}). For brevity, we will often refer to such a manifold simply as an AR homology sphere.
\end{definition}

We utilize the framework of graded roots.

\begin{definition}[{\cite[Definition 3.2]{Nem05}}, {\cite[Section 2.3]{DM19}}]
    Let $R$ be an infinite tree with vertices $\mathcal{V}(R)$ and edges $\mathcal{E}(R)$. We say that $R$ is a \emph{graded root} with a grading function $\chi: \mathcal{V}(R) \to \Q$ taking values in a coset of $2\Z$ in $\Q$ if:
    \begin{itemize}
        \item $|\chi(u) - \chi(v)| = 2$ for any edge $(u, v) \in \mathcal{E}(R)$,
        \item $\chi(u) < \max\{\chi(v), \chi(w)\}$ for any edges $(u,v), (u,w) \in \mathcal{E}(R)$ with $v \neq w$,
        \item $\chi$ is bounded above, 
        \item $\chi^{-1}(k)$ is finite for any $k \in \Q$, and
        \item $\#\chi^{-1}(k) = 1$ for $k \ll 0$.
    \end{itemize}
\end{definition}

\begin{definition}[{\cite[Definition 2.11]{DM19}}]
    A \emph{symmetric graded root} is a graded root $R$ together with an involution $J : \mathcal{V}(R) \to \mathcal{V}(R)$ such that
    \begin{itemize}
        \item $ \chi(v) = \chi(Jv)$ for any vertex $v$,
        \item $(v, w)$ is an edge in $R$ if and only if $(Jv, Jw)$ is an edge in $R$,
        \item for every $k \in \Q$, there is at most one $J$-invariant vertex $v$ with $\chi(v) = k$.
    \end{itemize}
\end{definition}

Given a graded root, one can construct an associated $\Z[U]$-module that encodes its combinatorial structure into an algebraic object.

\begin{definition}[{\cite[Section 2.3]{DM19}}]
    For a graded root $(R,\chi)$, the \emph{associated $\Z[U]$-module} $\mathbb{H}^-(R)$ is defined by
    \[
    \mathbb{H}^-(R) = \bigoplus_{v \in \mathcal{V}(R)} \mathbb{Z}[U]v \Big/ \sim,
    \]
    where the congruence relation $\sim$ is generated by $U \cdot v = w$ for all $(v, w) \in \mathcal{E}(R)$ with $\chi(v) - \chi(w) = 2$. Also, we define the degree of $v \in \mathcal{V}(R)$ to be $\chi(v)$, so that multiplication by $U$ is a map of degree $-2$. When $R$ is symmetric, the symmetry of the module $\mathbb{H}^-(R)$ is induced by the action of $J$ on the generators $v \in \mathcal{V}(R)$.
\end{definition}

\begin{remark}
    Recall that a vertex $v$ of a graded root $R$ is called a \emph{leaf} if the number of edges incident to $v$ is exactly one. According to \cite[Section 3.3]{Nem05}, any graded root has only finitely many leaves. Since $R$ is a tree, for any two leaves $v_0$ and $v_1$, there exists a unique simple path connecting them. Consequently, in the associated module $\mathbb{H}^-(R)$, there exist integers $n_0, n_1 \ge 0$ such that $U^{n_0}v_0 = U^{n_1}v_1$.
\end{remark}

\begin{remark}
    In what follows, we will often identify a graded root with its associated $\Z[U]$-module and use these terms interchangeably when there is no danger of confusion.
\end{remark}

For an AR homology sphere $Y$ and a spin structure $\s$, the associated symmetric graded root $R(Y, \s)$ is combinatorially defined from its plumbing graph. The precise construction is given in \cite[Section 4]{Nem05}, \cite[Section 4]{DSS26}, and \cite[Section 2.3]{DM19}. Then, the module associated to $R(Y, \s)$ is isomorphic to the $S^1$-equivariant co-Borel homology of its SWF spectrum (cf.~\cite[Lemma 5.9]{DSS26}).

We extract two fundamental numerical invariants from a given symmetric graded root:

\begin{definition}
    Let $R$ be a symmetric graded root.
    \begin{enumerate}[label=(\arabic*)]
        \item We denote the grading of the uppermost vertex in $R$ by $2\delta(R)$.
        \item We denote the grading of the uppermost $J$-invariant vertex in $R$ by $2\beta(R)$.
    \end{enumerate}
\end{definition}

In the context of AR homology spheres, these invariants coincide with the well-known gauge-theoretic invariants $\delta$ and $\bar{\mu}$ (the Neumann--Siebenmann invariant):

\begin{theorem}[{\cite[Section 7.1]{DSS26}}]\label{thm:DSS7.1}
    Let $Y$ be an AR homology sphere with a spin structure $\s$ and let $R(Y, \s)$ be the symmetric graded root associated to $(Y, \s)$. Then the following equalities hold:
    \begin{enumerate}[label=(\arabic*)]
        \item $\delta(R(Y, \s)) = \delta(Y, \s)$.
        \item $\beta(R(Y, \s)) = \beta(Y, \s) = -\bar{\mu}(Y, \s)$,
    \end{enumerate}
    where $\beta(Y, \s)$ is the invariant introduced by Manolescu \cite{Man16a, Man16b}.
\end{theorem}

According to \cite[Section 7.1]{DSS26}, for any symmetric graded root $R$, there exists a $\Pin(2)$-equivariant stable homotopy type $\Hty(R)$ whose $S^1$-equivariant co-Borel homology is isomorphic to $\mathbb{H}^-(R)$. We recall the construction of the $\Pin(2)$-lattice spectrum below.

\begin{construction}\label{con:lattice_spectrum}
    The stable homotopy type $\Hty(R)$ is constructed as follows:
    \begin{enumerate}[label=(\roman*)]
        \item First, we truncate $R$ at a vertex of grading $h \ll 0$ on the infinite downward stem, choosing $h$ such that $2\beta(R) - h \equiv 0 \pmod 4$. To simplify notation, for any vertex $v$, we define its associated complex dimension as $c(v) := \frac{\chi(v) - h}{2}$, and we let $q := \frac{2\beta(R) - h}{4}$ denote the quaternionic dimension of the stem.
        \item We split the upper part of $R$ into two symmetric components by removing the $J$-invariant stem. Let $\mathcal{L} = \{x_0, \dots, x_k\}$ be the set of leaves on one side, so the set of non-$J$-invariant leaves of $R$ is exactly $\{x_0, \dots, x_k, Jx_0, \dots, Jx_k\}$. We introduce a quaternionic representation sphere $S^{q\mathbb{H}}$ for the stem, and for each symmetric pair of leaves, we introduce a $G$-space $S^{c(x_i)\C} \vee jS^{c(x_i)\C}$.
        \item We define a space $X$ by taking the disjoint union of these spaces and quotienting by an equivalence relation $\sim$ designed to realize the connectivity of the graded root:
        \[
        X = \left( S^{q\mathbb{H}} \sqcup \bigsqcup_{i = 0}^{k} \left( S^{c(x_i)\C} \vee jS^{c(x_i)\C} \right) \right) \Bigg/ \sim.
        \]
        The congruence relation $\sim$ is generated by two types of gluings. First, for any two leaves $x, y \in \mathcal{L}$, let $m, n \ge 0$ be the unique minimal integers such that $U^m x = U^n y$. Let $u$ be this intersection vertex, which has complex dimension $c(u) = \frac{\chi(x) - 2m - h}{2}$. We glue the corresponding branches together along the canonical inclusions of the subspace $S^{c(u)\C} \vee jS^{c(u)\C}$:
        \begin{align*}
        S^{c(u)\C} \vee jS^{c(u)\C} &\hookrightarrow S^{c(x)\C} \vee jS^{c(x)\C} \\
        S^{c(u)\C} \vee jS^{c(u)\C} &\hookrightarrow S^{c(y)\C} \vee jS^{c(y)\C}.
        \end{align*}
        Second, we glue each pair of symmetric branches to the $J$-invariant stem. For each leaf $x \in \mathcal{L}$, let $n \ge 0$ be the minimal integer such that $U^n x = U^n Jx$. The complex dimension of this intersection vertex is $c = \frac{\chi(x) - 2n - h}{2}$. To ensure the $G$-equivariance of the gluing when $c$ is odd (say, $c = 2a + 1$), we utilize the canonical orthogonal decomposition of the stem sphere $S^{q\mathbb{H}} \cong (\mathbb{H}^q)^+$:
        \[
        S^{q\mathbb{H}} \cong \left( \mathbb{H}^a \oplus (\C \oplus j\C) \oplus \mathbb{H}^{q - a - 1} \right)^+.
        \]
        (For even $c = 2a$, the $\C \oplus j\C$ term is simply omitted.) We then glue $S^{c(x)\C} \vee jS^{c(x)\C}$ to $S^{q\mathbb{H}}$ as follows: for the $S^{c(x)\C}$ component, we map the subspace $S^{c\C} \subset S^{c(x)\C}$ to the $c$-dimensional sub-representation sphere $(\mathbb{H}^a \oplus \C)^+ \subset S^{q\mathbb{H}}$. Similarly, for the $jS^{c(x)\C}$ component, we map the subspace $jS^{c\C} \subset jS^{c(x)\C}$ to the sub-representation sphere $(\mathbb{H}^a \oplus j\C)^+ \subset S^{q\mathbb{H}}$. Note that the intersection of these two images in $S^{q\mathbb{H}}$ is precisely the quaternionic sphere $S^{a\mathbb{H}} \cong (\mathbb{H}^a)^+$. This specific embedding guarantees that the $j$-action remains well-defined across the quotient, as the interchange of the symmetric branches naturally matches the quaternionic structure intertwining $\C$ and $j\C$ within the stem.
        \item Finally, we obtain the $\Pin(2)$-equivariant stable homotopy type as the spectrum
        \[
        \Hty(R) := \left(X, 0, -\frac{h}{4}\right).
        \]
    \end{enumerate}
\end{construction}

\begin{remark}
    Note that for any symmetric graded root $R$, the $S^1$-fixed point set of the $G$-space $X$ is isomorphic to $S^0$. In other words, $X$ is a $G$-space of type SWF at level $0$.
\end{remark}

The following lemma follows immediately from Construction~\ref{con:lattice_spectrum}.

\begin{lemma}\label{lem:can_inclusion}
    For any symmetric graded root $R$, there exists an admissible morphism
    \[
    S^{\frac{1}{2}\beta(R)\mathbb{H}} \to \Hty(R)
    \]
    represented by the inclusion $S^{q\mathbb{H}} \hookrightarrow X$, where $q$ and $X$ are as defined in Construction~\ref{con:lattice_spectrum}.
\end{lemma}

This morphism is a spectrum cap for $\Hty(R)$.

The topological significance of the space $\Hty(R)$ is established by the following theorem, which bridges the gap between the combinatorial root and the gauge-theoretic spectrum:

\begin{theorem}[{\cite[Theorem 1.2, Section 7.1]{DSS26}}]\label{thm:DSS1.2}
    Let $Y$ be an almost-rational plumbed homology sphere equipped with a spin structure $\s$, and let $R(Y, \s)$ be the symmetric graded root associated to $(Y, \s)$. Then there is a $\Pin(2)$-equivariant stable homotopy equivalence:
    \[
    \Hty(R(Y, \s)) \simeq \SWF(Y, \s).
    \]
\end{theorem}

\subsection{Local Equivalence and Monotone Subroots}

We require an appropriate notion of equivalence in the stable homotopy category:

\begin{definition}[{\cite[Definition 2.7]{Sto20}}]
    We say that two objects $\mathcal{X}_0 = (X_0, m_0, n_0), \mathcal{X}_1 = (X_1, m_1, n_1) \in \mathfrak{C}_G$ $(n_0 - n_1 \in \Z)$ are \emph{locally equivalent} if there exist morphisms $\phi: \mathcal{X}_0 \to \mathcal{X}_1$ and $\psi: \mathcal{X}_1 \to \mathcal{X}_0$ in $\mathfrak{C}_G$ such that their representative maps
    \[
     \tilde{\phi}: \Sigma^{a\tilde{\R}}\Sigma^{b \mathbb{H}} X_0 \to \Sigma^{(a + m_0 - m_1)\tilde{\R}}\Sigma^{(b + n_0 - n_1) \mathbb{H}} X_1
    \]
    and
    \[
     \tilde{\psi}: \Sigma^{(a + m_0 - m_1)\tilde{\R}}\Sigma^{(b + n_0 - n_1) \mathbb{H}} X_1 \to \Sigma^{a\tilde{\R}}\Sigma^{b \mathbb{H}} X_0
    \]
    induce $G$-homotopy equivalences on their $S^1$-fixed point sets. 
    Furthermore, a morphism satisfying this condition in one direction (such as $\phi$ or $\psi$ alone) is called a \emph{local map}.
\end{definition}

\begin{remark}
    Note that if both the domain
    and the target objects are of the form $(X, m, n)$ where $X$ is a space of type SWF, then any local map between them is admissible.
\end{remark}

This topological definition naturally descends to an algebraic notion on the level of graded roots:

\begin{definition}[{\cite[Definition 7.2]{DSS26}}]
    Let $R_0$ and $R_1$ be two graded roots. A \emph{graded root map} is a grading-preserving, $\Z[U]$-equivariant map from the $\Z[U]$-module $\mathbb{H}^-(R_0)$ to the $\Z[U]$-module $\mathbb{H}^-(R_1)$. If $R_0$ and $R_1$ are symmetric, then we require $F$ to intertwine these symmetries. We denote such a map by
    \[
    F \colon R_0 \rightarrow R_1.
    \]
    We say that a graded root map is \emph{local} if it induces an isomorphism after inverting $U$. We say that $R_0$ and $R_1$ are \emph{locally equivalent} if there exist local maps in both directions.
\end{definition}

The geometric and algebraic notions of local equivalence are deeply intertwined. According to \cite[Lemma 7.3]{DSS26}, under appropriate assumptions, a graded root map can be lifted to a morphism between the corresponding spectra. Furthermore, as explicitly noted in the second paragraph of the proof of \cite[Theorem 7.4]{DSS26}, if the original graded root map is a local map, then the lifted morphism of spectra is also a local map. Combining these facts yields the following theorem:

\begin{theorem}[{\cite[Lemma 7.3 and Theorem 7.4]{DSS26}}]\label{thm:local_map_equivalence}
    Let $R_0$ and $R_1$ be symmetric graded roots whose uppermost $J$-invariant vertices lie in the same grading modulo 4. For any graded root map $F: R_0 \to R_1$, there exists a $G$-equivariant morphism
    \[
    \mathcal{F}: \Hty(R_0) \to \Hty(R_1)
    \]
    that induces the map $F$ on ($S^1$-equivariant) co-Borel homology. Furthermore, $F$ is a local map of graded roots if and only if $\mathcal{F}$ is a local map of spectra.
\end{theorem}

To facilitate explicit computations of these local equivalence classes, it is often convenient to reduce a general symmetric graded root to a simplified standard form.
We now recall the construction of a monotone graded root (see \cite[Section 6]{DM19} for details). Fix a positive integer $n$. Let $h_1, h_2, \dots , h_n$ be a strictly decreasing sequence, and $r_1, r_2, \dots , r_n$ be a strictly increasing sequence of rational numbers taking values in a single coset of $2\Z$ in $\Q$, and $h_n \ge r_n$. A graded root $M = M(h_1, r_1; h_2, r_2; \dots; h_n, r_n)$ is constructed as follows:

\begin{enumerate}[label=(\roman*)]
    \item First, we form the stem of our graded root by drawing a single infinite downward $U$-tower whose uppermost vertex has degree $r_n$.
    \item Next, for each $1 \le i < n$, we introduce a symmetric pair of leaves $v_i$ and $Jv_i$ in degree $h_i$, and connect these to the stem by a pair of paths meeting the stem at degree $r_i$.
    \item If $h_n > r_n$, we similarly introduce a pair of vertices $v_n$ and $Jv_n$ in degree $h_n$ and connect these to the stem at degree $r_n$. In the degenerate case where $h_n = r_n$, we declare $v_n$ to be the $J$-invariant vertex of degree $r_n$ lying on the stem and take no further action.
\end{enumerate}

\begin{figure}[htbp]
    \centering
    \begin{tikzpicture}[xscale=0.55, yscale=0.4, thick]
        \begin{scope}[shift={(-7.5,0)}]
            \foreach \y in {5, 3, 1, -1, -3, -5} {
                \draw[dotted, gray, thick] (-5.5, \y) -- (5.5, \y);
            }
            
            \node[left] at (-5.5, 6) {Degree};
            \node[left] at (-5.5, 5) {$h_1$};
            \node[left] at (-5.5, 3) {$h_2$};
            \node[left] at (-5.5, 1) {$h_n$};
            \node[left] at (-5.5, -1) {$r_n$};
            \node[left] at (-5.5, -3) {$r_2$};
            \node[left] at (-5.5, -5) {$r_1$};
            
            \draw (0, -1) -- (0, -6);
            \node at (0, -6.5) {$\vdots$};
            
            \draw (-1, 1) node[above=1pt] {$v_n$} -- (0, -1);
            \draw (1, 1) node[above=1pt] {$Jv_n$} -- (0, -1);
            
            \draw (-3, 3) node[above=1pt] {$v_2$} -- (0, -3);
            \draw (3, 3) node[above=1pt] {$Jv_2$} -- (0, -3);
            
            \draw (-5, 5) node[above=1pt] {$v_1$} -- (0, -5);
            \draw (5, 5) node[above=1pt] {$Jv_1$} -- (0, -5);
            
            \node at (0, -8.5) {The case $h_n > r_n$};
        \end{scope}

        \begin{scope}[shift={(7.5,0)}]
            
            \foreach \y in {5, 3, 1, -1, -3} {
                \draw[dotted, gray, thick] (-5.5, \y) -- (5.5, \y);
            }
            
            \node[left] at (-5.5, 6) {Degree};
            \node[left] at (-5.5, 5) {$h_1$};
            \node[left] at (-5.5, 3) {$h_2$};
            \node[left] at (-5.5, 1) {$h_n=r_n$};
            \node[left] at (-5.5, -1) {$r_2$};
            \node[left] at (-5.5, -3) {$r_1$};
            
            \draw (0, 1) -- (0, -6);
            \node at (0, -6.5) {$\vdots$};
            
            \node[above=1pt] at (0, 1) {$v_n$};
             
            \draw (-2, 3) node[above=1pt] {$v_2$} -- (0, -1);
            \draw (2, 3) node[above=1pt] {$Jv_2$} -- (0, -1);
            
            \draw (-4, 5) node[above=1pt] {$v_1$} -- (0, -3);
            \draw (4, 5) node[above=1pt] {$Jv_1$} -- (0, -3);
            
            \node at (0, -8.5) {The case $h_n = r_n$};
        \end{scope}
        
    \end{tikzpicture}
    \caption{Monotone graded roots $M(h_1, r_1; \dots ; h_n, r_n)$ for $h_n > r_n$ and $h_n = r_n$ (illustrated for $n=3$).}
    \label{fig:monotone_root_comparison}
\end{figure}
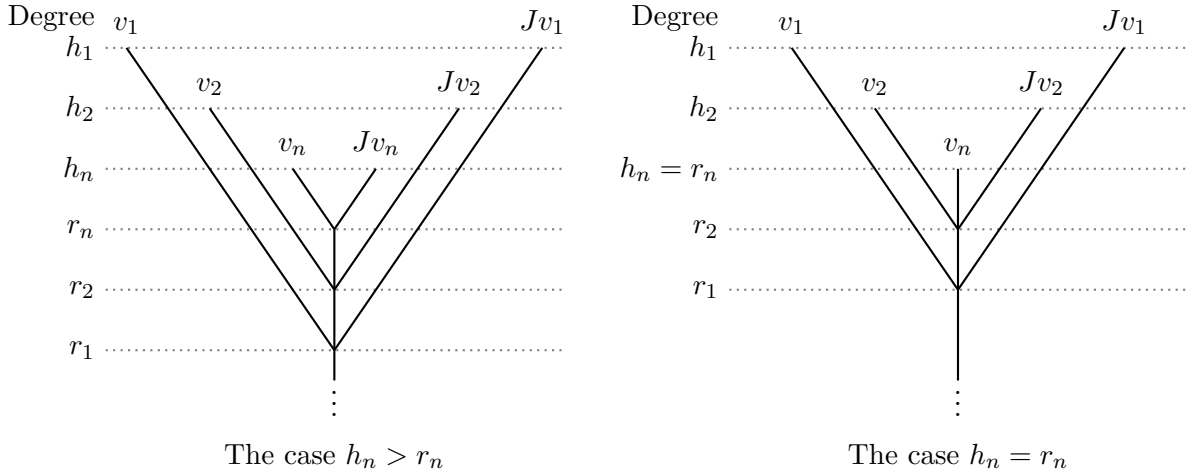

A symmetric graded root constructed in this way is called a \emph{monotone graded root}. As shown in \cite[Section 6]{DM19}, any symmetric graded root $R$ contains a monotone graded root as a subroot. In particular, the canonical one constructed via the greedy algorithm in \cite[Section 6]{DM19} is referred to as the \emph{monotone subroot} of $R$.

The utility of monotone graded roots stems from the fact that they completely capture the local equivalence class of the original root:

\begin{theorem}[{\cite[Theorem 6.1]{DM19}, \cite[Section 7.1]{DSS26}}]\label{thm:monotone_subroot}
    Any symmetric graded root $R$ is locally equivalent to its monotone subroot.
\end{theorem}

Motivated by this simplification, we define a family of elementary symmetric graded roots that serve as the fundamental building blocks:

\begin{definition}
    For rational numbers $m, n \in \Q$ with $n - m \in \Z_{\ge 0}$, let $X_{m, n}$ be the symmetric graded root defined by
    \[
    X_{m, n} = (\Z[U]v \oplus \Z[U]Jv) \big/ \Z[U](U^{n - m}v - U^{n - m} Jv)
    \]
    equipped with the grading $\chi(v) = \chi(Jv) = 2n$. By construction, $\delta(X_{m, n}) = n$ and $\beta(X_{m, n}) = m$. We define
    \[
    A_{m, n} := \Hty(X_{m, n}).
    \]
    Note that we allow the degenerate case $m = n$, in which case $A_{m, n}$ is stably $G$-homotopy equivalent to $\Sigma^{\frac{m}{2}\mathbb{H}}S^0$. 
\end{definition}

We construct a corresponding family of topological spaces that realize these elementary roots. Let $A_n$ be a $G$-subspace of the quaternionic representation sphere $S^{n\mathbb{H}}$. Using the canonical decomposition $S^{n\mathbb{H}} \cong (\C^n \oplus j\C^n)^+$, we explicitly define $A_n$ as the union of the two complex sub-representation spheres:
\[
A_n := (\C^n)^+ \cup (j\C^n)^+ \subset S^{n\mathbb{H}}.
\]
Geometrically, this space is obtained by taking two copies of the complex representation sphere $S^{n\C}$ and identifying them at the origin and at infinity. By construction, for any rational numbers $m, n \in \Q$ with $n - m \in \Z_{\ge 0}$, we have the following isomorphism in the stable homotopy category:
\[
A_{m, n} \simeq (A_{n - m}, 0, -\frac{m}{2}).
\]

\subsection{Admissible Morphisms to Representation Spheres}

We can construct explicit local maps of graded roots that give rise to spectrum caps. An analogous algebraic reduction is implicitly utilized in \cite[Section 6]{DM19} via the standard complex. Inspired by their approach, we provide here a direct and elementary proof in the context of graded roots.

\begin{lemma}\label{lem:local_graded_map}
    For any symmetric graded root $R$, there exists a local map of graded roots
    \[
    F : R \to X_{\beta(R), \delta(R)}.
    \]
\end{lemma}

\begin{proof}
    Since any symmetric graded root is locally equivalent to its monotone subroot, we may assume without loss of generality that $R$ is monotone. Let $R = M(h_1, r_1; \dots; h_n, r_n)$. By definition, the grading of the uppermost vertex is $h_1$, and the grading of the uppermost $J$-invariant vertex is $r_n$. Thus, we have $h_1 = 2\delta(R)$ and $r_n = 2\beta(R)$.

    Let $P$ be the free $\Z[U]$-module given by
    \[
    P = \bigoplus_{i=1}^n (\Z[U]v_i \oplus \Z[U]Jv_i),
    \]
    and let $N$ be the submodule of $P$ generated by the following elements:
    \[
    U^{\frac{h_i - r_i}{2}} v_i - U^{\frac{h_n - r_i}{2}} v_n, \quad U^{\frac{h_i - r_i}{2}} Jv_i - U^{\frac{h_n - r_i}{2}} Jv_n \quad (1 \le i \le n - 1),
    \]
    and
    \[
    U^{\frac{h_n - r_n}{2}} v_n - U^{\frac{h_n - r_n}{2}} Jv_n.
    \]
    Then the associated $\Z[U]$-module is given by the quotient $\mathbb{H}^-(R) = P / N$.

    We define a grading-preserving $\Z[U]$-module homomorphism $\tilde{F}$ from $P$ to the target module $\mathbb{H}^-(X_{\beta(R), \delta(R)})$ by mapping the basis elements as follows:
    \[
    v_i \mapsto U^{\frac{2\delta(R) - h_i}{2}}v, \quad Jv_i \mapsto U^{\frac{2\delta(R) - h_i}{2}}Jv.
    \]
    One can easily check that $\tilde{F}$ maps each generator of $N$ to $0$ in $\mathbb{H}^-(X_{\beta(R), \delta(R)})$. Therefore, $\tilde{F}$ descends to a well-defined map 
    \[
    F: \mathbb{H}^-(R) \to \mathbb{H}^-(X_{\beta(R), \delta(R)}).
    \]
    Moreover, we must show that $F$ is a local map, which means that it induces an isomorphism after inverting $U$. Recall that for any graded root with a single infinite downward stem, the localized module $U^{-1}\mathbb{H}^-$ is a free $\Z[U, U^{-1}]$-module of rank $1$, generated by any element on the infinite stem. In our setting, $U^{-1}\mathbb{H}^-(R)$ is generated by $v_n$, and $U^{-1}\mathbb{H}^-(X_{\beta(R), \delta(R)})$ is generated by $v$. Since $F(v_n) = U^{\frac{2\delta(R) - h_n}{2}}v$, and $U$ is invertible in the localized module, $F$ maps the generator of $U^{-1}\mathbb{H}^-(R)$ to a generator of $U^{-1}\mathbb{H}^-(X_{\beta(R), \delta(R)})$ up to a unit in $\Z[U, U^{-1}]$. Thus, the induced map 
    \[
    U^{-1}F: U^{-1}\mathbb{H}^-(R) \to U^{-1}\mathbb{H}^-(X_{\beta(R), \delta(R)})
    \]
    is an isomorphism. This concludes the proof.
\end{proof}

This algebraic mapping property allows us to deduce the existence of a corresponding local map at the level of stable homotopy types:

\begin{lemma}\label{lem:local_map_AR}
    Let $Y$ be an AR plumbed homology sphere equipped with a spin structure $\s$. Then there exists a local map of spectra
    \[
    \mathcal{F}: \SWF(Y, \s) \to A_{\beta(Y, \s), \delta(Y, \s)}.
    \]
\end{lemma}

\begin{proof}
    By Theorem~\ref{thm:DSS7.1} and Lemma~\ref{lem:local_graded_map}, there exists a local map of graded roots
    \[
    F: R(Y, \s) \to X_{\beta(Y, \s), \delta(Y, \s)}.
    \]
    Applying Theorem~\ref{thm:local_map_equivalence}, this map lifts to a local map of spectra 
    \[
    \mathcal{F}: \Hty(R(Y, \s)) \to A_{\beta(Y, \s), \delta(Y, \s)}
    \]
    which induces the map $F$ on co-Borel homology. Finally, recall from Theorem~\ref{thm:DSS1.2} that there is a $G$-equivariant stable homotopy equivalence $\SWF(Y, \s) \simeq \Hty(R(Y, \s))$. This completes the proof.
\end{proof}

\begin{remark}
    We note that the local map of spectra $\mathcal{F}$ constructed in Lemma~\ref{lem:local_map_AR} is explicitly utilized to compute the $\kappa$-invariant for AR homology spheres in \cite[Section 7.2]{DSS26}.
\end{remark}

We are now ready to establish the existence of admissible morphisms from the SWF spectrum of an AR homology sphere to certain representation spheres. The proof of the following theorem combines Lemma~\ref{lem:local_graded_map} and Lemma~\ref{lem:local_map_AR} with the explicit $G$-map $A_n \to S^{\tilde{\R}}$, which was utilized in the calculation of the $\kappa$-invariant in \cite{DSS26}. This result will serve as a crucial ingredient in the proof of our main theorem later in this paper.

\begin{theorem}\label{thm:admissible_map}
    Let $Y$ be an AR plumbed homology sphere equipped with a spin structure $\s$. If $\beta(Y, \s) = \delta(Y, \s)$, then there exists an admissible morphism
    \[
    f: \SWF(Y, \s) \to \Sigma^{\frac{\beta(Y, \s)}{2}\mathbb{H}} S^0.
    \]
    If $\beta(Y, \s) < \delta(Y, \s)$, then there exists an admissible morphism
    \[
    f: \SWF(Y, \s) \to \Sigma^{\frac{\beta(Y, \s)}{2}\mathbb{H}} S^{\tilde{\R}}.
    \]
\end{theorem}

\begin{proof}
    By Lemma~\ref{lem:local_map_AR}, there exists a local map of spectra $\mathcal{F}: \SWF(Y, \s) \to A_{\beta(Y, \s), \delta(Y, \s)}$.
    Let us write $\SWF(Y, \s) = (X, 0, -\frac{h}{4})$ and $A_{\beta(Y, \s), \delta(Y, \s)} = (A_n, 0, -\frac{\beta(Y, \s)}{2})$ in the stable homotopy category, where $n = \delta(Y, \s) - \beta(Y, \s) \ge 0$. By the definition of a local map, there exists a representative map
    \[
    \tilde{\mathcal{F}}: \Sigma^{a\tilde{\R}}\Sigma^{b_0\mathbb{H}} X \to \Sigma^{a\tilde{\R}}\Sigma^{b_1\mathbb{H}} A_n
    \]
    for some integers $a, b_1, b_0 \gg 0 \: (b_1 - b_0 = \frac{2\beta(Y, \s) - h}{4} \in \Z)$, whose restriction to the $S^1$-fixed point set induces a $G$-homotopy equivalence. By construction, the $S^1$-fixed point sets of both $X$ and $A_n$ are isomorphic to $S^0$, and their $G$-fixed point sets are also $S^0$. Therefore, the restriction to the $S^1$-fixed point set yields a map
    \[
    \tilde{\mathcal{F}}^{S^1}: \Sigma^{a\tilde{\R}}S^0 \xrightarrow{\simeq} \Sigma^{a\tilde{\R}}S^0
    \]
    which is a $G$-homotopy equivalence, and the further restriction to the $G$-fixed point set gives a homotopy equivalence $\tilde{\mathcal{F}}^G: S^0 \xrightarrow{\simeq} S^0$.
    
    To conclude the existence of the required admissible map $f$, we consider two cases:
    \begin{enumerate}[label=(\roman*)]
        \item If $\beta(Y, \s) = \delta(Y, \s)$, then $n = 0$, and thus $A_0 = S^0$. Therefore, the local map $\mathcal{F}$ itself serves as the desired admissible map.
        \item If $\beta(Y, \s) < \delta(Y, \s)$, then $n > 0$. We obtain a $G$-map $f: A_n \to S^{\tilde{\R}}$ by collapsing the two spheres (see Figure~\ref{fig:G-map}). This map restricts to an isomorphism on the $G$-fixed point set $S^0$. Consequently, the induced map on the $G$-fixed point sets for the composition $f \circ \tilde{\mathcal{F}}: \Sigma^{a\tilde{\R}}\Sigma^{b_0\mathbb{H}} X \to \Sigma^{a\tilde{\R}}\Sigma^{b_1\mathbb{H}} S^{\tilde{\R}}$ is a homotopy equivalence from $S^0$ to $S^0$. Therefore, regarding $f$ as a stable map, we obtain the desired admissible morphism $f \circ \mathcal{F}: \SWF(Y, \s) \to \Sigma^{\frac{\beta(Y, \s)}{2}\mathbb{H}} S^{\tilde{\R}}$.
    \end{enumerate}
    This completes the proof.
\end{proof}

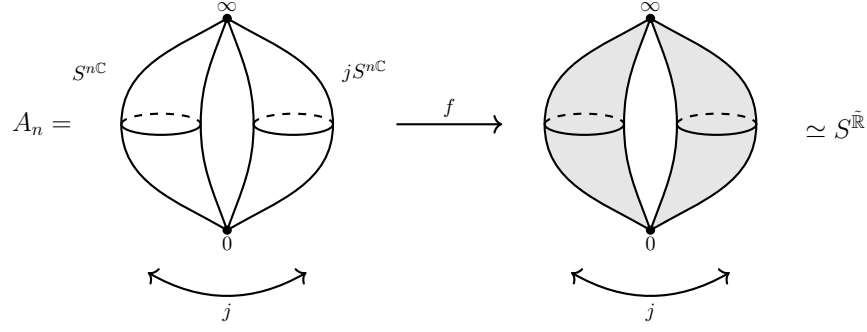
\begin{figure}[htbp]
    \centering
    \begin{tikzpicture}[scale=0.7, transform shape]
        \begin{scope}[shift={(0,0)}]
            \node at (-3.5, 0) {\Large $A_n = $};

            \draw[thick] (0,2) to[out=210, in=90] (-2,0) to[out=270, in=150] (0,-2);
            \draw[thick] (0,2) to[out=250, in=90] (-0.5,0) to[out=270, in=110] (0,-2);
            \draw[thick] (-2,0) arc (180:360:0.75 and 0.2);
            \draw[dashed, thick] (-2,0) arc (180:0:0.75 and 0.2);
            \node at (-2.6, 1) {$S^{n\C}$};

            \draw[thick] (0,2) to[out=330, in=90] (2,0) to[out=270, in=30] (0,-2);
            \draw[thick] (0,2) to[out=290, in=90] (0.5,0) to[out=270, in=70] (0,-2);
            \draw[thick] (0.5,0) arc (180:360:0.75 and 0.2);
            \draw[dashed, thick] (0.5,0) arc (180:0:0.75 and 0.2);
            \node at (2.6, 1) {$jS^{n\C}$};

            \fill (0,2) circle (2.5pt) node[above] {$\infty$};
            \fill (0,-2) circle (2.5pt) node[below] {$0$};

            \draw[<->, thick] (-1.5, -2.8) to[bend right=30] node[below] {$j$} (1.5, -2.8);
        \end{scope}

        \draw[->, thick] (3.2, 0) -- (5.2, 0) node[midway, above] {$f$};

        \begin{scope}[shift={(8,0)}]
            \fill[gray!20] (0,2) to[out=210, in=90] (-2,0) to[out=270, in=150] (0,-2) 
                           to[out=110, in=270] (-0.5,0) to[out=90, in=250] cycle;

            \fill[gray!20] (0,2) to[out=330, in=90] (2,0) to[out=270, in=30] (0,-2) 
                           to[out=70, in=270] (0.5,0) to[out=90, in=290] cycle;

            \draw[thick] (0,2) to[out=210, in=90] (-2,0) to[out=270, in=150] (0,-2);
            \draw[thick] (0,2) to[out=330, in=90] (2,0) to[out=270, in=30] (0,-2);
            \draw[thick] (0,2) to[out=250, in=90] (-0.5,0) to[out=270, in=110] (0,-2);
            \draw[thick] (0,2) to[out=290, in=90] (0.5,0) to[out=270, in=70] (0,-2);
            \draw[thick] (-2,0) arc (180:360:0.75 and 0.2);
            \draw[dashed, thick] (-2,0) arc (180:0:0.75 and 0.2);
            \draw[thick] (0.5,0) arc (180:360:0.75 and 0.2);
            \draw[dashed, thick] (0.5,0) arc (180:0:0.75 and 0.2);

            \fill (0,2) circle (2.5pt) node[above] {$\infty$};
            \fill (0,-2) circle (2.5pt) node[below] {$0$};

            \node at (3.5, 0) {\Large $\simeq S^{\tilde{\R}}$};

            \draw[<->, thick] (-1.5, -2.8) to[bend right=30] node[below] {$j$} (1.5, -2.8);
        \end{scope}

    \end{tikzpicture}
    \caption{The $G$-map $f: A_n \to S^{\tilde{\R}}$ obtained by collapsing the two spheres. The curved arrows indicate the involution $j \in \Pin(2)$ that interchanges the two symmetric complex representation spheres.}
    \label{fig:G-map}
\end{figure}

\section{A $\frac{10}{8}$-type inequality for AR homology spheres}

In this section, we will prove the main theorem by using the necessary condition of the $\frac{10}{8} + 4$ inequality established by Hopkins--Lin--Shi--Xu \cite{HLSX22}.

\subsection{Constraints on the intersection forms of $4$-manifolds bounded by AR homology spheres}

To connect our Seiberg--Witten theoretic setup with stable homotopy theory, we first recall the definition of a Furuta--Mahowald class. This concept, introduced in \cite{HLSX22}, plays a central role in extracting inequalities from equivariant stable maps.

\begin{definition}[{\cite[Definition 1.15]{HLSX22}}]
    For $p \ge 1$, \emph{a Furuta--Mahowald class} of level-$(p,q)$ is a stable map 
    \[
    \gamma \in \pi_{p\mathbb{H}- q\tilde{\R}}^{\Pin(2)}S^0
    \]
    that fits into the diagram:
    \[
    \begin{tikzcd}
    S^{p\mathbb{H}} \ar[rd, "\gamma"] &  \\ 
    S^0 \ar[u,"a_{\mathbb{H}}^p"] \ar[r, swap, "a^q_{\tilde{\R}}"] & S^{q\tilde{\R}}
    \end{tikzcd}
    \]
    where $a_\mathbb{H} \in \pi_{-\mathbb{H}}^{\Pin(2)} S^0$ and $a_{\tilde{\R}} \in \pi_{-\tilde{\R}}^{\Pin(2)} S^0$ are stable homotopy classes that represent the inclusions $S^0 \hookrightarrow S^\mathbb{H}$ and $S^0 \hookrightarrow S^{\tilde{\R}}$ of fixed points.
\end{definition}

In our framework, the topological conditions of a Furuta--Mahowald class can be rephrased using the language of admissible morphisms in the category $\mathfrak{C}_G$. The following proposition establishes the equivalence between these two notions. Its proof relies on an application of equivariant obstruction theory, following the approach of \cite[Lemma 5.4]{Lin15} and \cite[Theorem 1.14]{HLSX22}.

\begin{proposition}\label{prop:equivalence_of_morphism}
    For any integers $p \ge 1$ and $q \ge 1$, a morphism $\gamma : S^{p \mathbb{H}} \to S^{q \tilde{\R}}$ is a Furuta--Mahowald class of level-$(p, q)$ if and only if $\gamma$ is admissible.
\end{proposition}

\begin{proof}
     Let $\gamma: S^{p\mathbb{H}} \to S^{q\tilde{\R}}$ be an admissible morphism. We must verify that its representative map $\tilde{\gamma}: S^{a \tilde{\R} \oplus (p + b) \mathbb{H}} \to S^{(a + q) \tilde{\R} \oplus b \mathbb{H}}$ makes the following diagram commute up to $G$-homotopy:
    \[
    \begin{tikzcd}
        S^{a \tilde{\R} \oplus (p + b) \mathbb{H}} \ar[rd, "\tilde{\gamma}"] &  \\ 
        S^{a \tilde{\R} \oplus b \mathbb{H}} \ar[u,"\tilde{a}_{\mathbb{H}}^p"] \ar[r, swap, "\tilde{a}^q_{\tilde{\R}}"] & S^{(a + q) \tilde{\R} \oplus b \mathbb{H}}
    \end{tikzcd}
    \]
    That is, we need to show that $\tilde{\gamma} \circ \tilde{a}_{\mathbb{H}}^p \simeq_G \tilde{a}^q_{\tilde{\R}}$. By taking a sufficiently large suspension, we may assume $a > 0$ and $b > 0$. To prove this, we apply equivariant obstruction theory (see \cite[Proposition 8.3.1]{tD79}) in two steps.

    First, we prove that the restrictions of these maps to the $S^1$-fixed point set, $(\tilde{\gamma} \circ \tilde{a}_{\mathbb{H}}^p)^{S^1}, (\tilde{a}^q_{\tilde{\R}})^{S^1}: S^{a\tilde{\R}} \to S^{(a + q)\tilde{\R}}$, are $G$-homotopic (i.e., $WS^1 = \Pin(2)/S^1 \cong \Z/2$-homotopic). 
    The definition of admissibility implies that the $G$-fixed point maps are homotopy equivalences $S^0 \xrightarrow{\simeq} S^0$, which must be the identity. Since the $\Z/2$-action on $S^{a\tilde{\R}} \setminus S^0$ is free, and since $q \ge 1$ implies that the target sphere is at least $a$-connected, all the relevant obstruction groups vanish. Thus, \cite[Proposition 8.3.1]{tD79} guarantees that these two maps are $\Z/2$-homotopic on $S^{a\tilde{\R}}$.

    Second, we prove that the maps $\tilde{\gamma} \circ \tilde{a}_{\mathbb{H}}^p, \tilde{a}^q_{\tilde{\R}}: S^{a \tilde{\R} \oplus b \mathbb{H}} \to S^{(a + q) \tilde{\R} \oplus b \mathbb{H}}$ are $G$-homotopic on the entire space. The $G$-action on $S^{a \tilde{\R} \oplus b \mathbb{H}} \setminus S^{a \tilde{\R}}$ is free, and $q \ge 1$ implies that the target sphere is at least $(a+4b)$-connected. Hence, all the relevant obstruction groups vanish. Again, by \cite[Proposition 8.3.1]{tD79}, these two maps are $G$-homotopic.

    Conversely, suppose that $\gamma$ is a Furuta--Mahowald class of level-$(p,q)$. By definition,
    \[
    \gamma\circ a_{\mathbb H}^p\simeq_G a_{\tilde{\R}}^q.
    \]
    Taking $G$-fixed points shows that $\gamma^G:S^0\to S^0$ is homotopic to the identity. Hence $\gamma$ is admissible.
\end{proof}

With this equivalence in hand, we can translate the existence of an admissible morphism into an algebraic constraint. The following condition for the existence of a Furuta--Mahowald class is a key algebraic topology result from \cite{HLSX22}.

\begin{theorem}[{\cite[Theorem 1.21]{HLSX22}}]\label{thm:inequality}
    For $p\ge 2$, a level-$(p,q)$ Furuta--Mahowald class exists if and only if:
    \[ 
    q \ge  
    \begin{cases}
    2p+2 \quad p\equiv 1&\pmod 8 \\
    2p+2 \quad p\equiv 2&\pmod 8 \\
    2p+3 \quad p\equiv 3&\pmod 8 \\
    2p+3 \quad p\equiv 4&\pmod 8 \\
    2p+2 \quad p\equiv 5&\pmod 8 \\
    2p+2 \quad p\equiv 6&\pmod 8 \\
    2p+3 \quad p\equiv 7&\pmod 8 \\
    2p+4 \quad p\equiv 8&\pmod 8.
    \end{cases}
    \]
\end{theorem}

We now apply these stable homotopy tools to our geometric setting. Note that, in the following geometric arguments, we may always assume without loss of generality that our spin $4$-manifold $X$ satisfies $b_1(X) = 0$. Indeed, we can perform surgeries on loops generating $H_1(X;\Z)$ to achieve $b_1(X) = 0$ without changing the intersection form. With this reduction understood, the next lemma demonstrates how the intersection form of $X$ dictates the existence of a specific admissible morphism.

\begin{lemma}\label{lem:existence_of_admissible}
    Let $Y$ be an AR homology sphere, and let $X$ be a compact, smooth, indefinite, spin $4$-manifold with $\sigma(X) \le 0$ and boundary $\partial X = Y$. Suppose that $Y$ is an integral homology sphere and the intersection form of $X$ is isomorphic to $p (-E_8) \oplus q \left( \begin{smallmatrix} 0 & 1 \\ 1 & 0 \end{smallmatrix} \right)$ for some integers $p \ge 0$ and $q > 0$. Then we have the following:
    \begin{itemize}
        \item If $-\bar{\mu}(Y) = \delta(Y)$, then there exists an admissible morphism from $S^{\frac{1}{2}(p + \bar{\mu}(Y))\mathbb{H}}$ to $S^{q \tilde{\R}}$.
        \item If $-\bar{\mu}(Y) < \delta(Y)$, then there exists an admissible morphism from $S^{\frac{1}{2}(p + \bar{\mu}(Y))\mathbb{H}}$ to $S^{(q + 1) \tilde{\R}}$.
    \end{itemize}
\end{lemma}

\begin{remark}
    The number $\frac{1}{2}(p + \bar{\mu}(Y))$ is an integer because $p \equiv \bar{\mu}(Y) \equiv \mu(Y) \pmod 2$, where $\mu(Y)$ is the Rokhlin invariant of $Y$.
\end{remark}

\begin{proof}[Proof of Lemma \ref{lem:existence_of_admissible}]
    By Theorem~\ref{thm:rel_BF}, viewing $X$ as a spin cobordism from $S^3$ to $Y$ provides an admissible morphism
    \[
    \Phi_X: \Sigma^{\frac{1}{2} p \mathbb{H}} S^0 \to \Sigma^{q\tilde{\R}}\SWF(Y).
    \]
    
    Suppose $-\bar{\mu}(Y) = \delta(Y)$. Theorem~\ref{thm:admissible_map} gives another admissible morphism
    \[
    f: \SWF(Y) \to \Sigma^{\frac{\beta(Y)}{2}\mathbb{H}} S^0.
    \]
    Composing $f$ with $\Phi_X$ and applying the relation $-\bar{\mu}(Y) = \beta(Y)$, we obtain the following morphism:
    \[
    f \circ \Phi_X: S^{\frac{1}{2}p\mathbb{H}} \to \Sigma^{-\frac{\bar{\mu}(Y)}{2}\mathbb{H}} S^{q \tilde{\R}}.
    \]
    By Definition~\ref{def:Pin(2)-category}, this morphism is represented by a map of the form
    \[
    \widetilde{f \circ \Phi_X}: \Sigma^{a\mathbb{H} \oplus b\tilde{\R}}S^{\frac{1}{2}(p + \bar{\mu}(Y)) \mathbb{H}} \to \Sigma^{a\mathbb{H} \oplus b\tilde{\R}} S^{q \tilde{\R}}
    \]
    for some integers $a, b \ge 0$. This representative map naturally induces the desired admissible morphism from $S^{\frac{1}{2}(p + \bar{\mu}(Y))\mathbb{H}}$ to $S^{q \tilde{\R}}$.
    
    Now suppose $-\bar{\mu}(Y) < \delta(Y)$. In this case, Theorem~\ref{thm:admissible_map} yields an admissible morphism
    \[
    f: \SWF(Y) \to \Sigma^{\frac{\beta(Y)}{2}\mathbb{H}} S^{\tilde{\R}}.
    \]
    Similarly, composition with $\Phi_X$ and the relation $-\bar{\mu}(Y) = \beta(Y)$ lead to
    \[
    f \circ \Phi_X: S^{\frac{1}{2}(p + \bar{\mu}(Y))\mathbb{H}} \to S^{(q + 1) \tilde{\R}}.
    \]
    This gives the desired admissible morphism.
\end{proof}

Finally, by combining the geometric realization of admissible morphisms from Lemma \ref{lem:existence_of_admissible} with the algebraic constraints of Theorem \ref{thm:inequality}, we obtain a $\frac{10}{8}$-type inequality for spin $4$-manifolds bounded by AR homology spheres. We are now ready to prove Theorem \ref{thm:main_inequality}.

\begin{proof}[Proof of Theorem~\ref{thm:main_inequality}]
    By Lemma~\ref{lem:existence_of_admissible}, we obtain an admissible morphism. According to Proposition~\ref{prop:equivalence_of_morphism}, this morphism represents a Furuta--Mahowald class of level-$(\frac{1}{2}(p + \bar{\mu}(Y)), q)$ if $-\bar{\mu}(Y) = \delta(Y)$, and of level-$(\frac{1}{2}(p + \bar{\mu}(Y)), q + 1)$ if $-\bar{\mu}(Y) < \delta(Y)$. Applying the necessary conditions of Theorem~\ref{thm:inequality} to these classes immediately yields the desired inequalities.
\end{proof}

\subsection{Connected sum of AR homology spheres}

Next, we consider the constraints on the intersection form of $4$-manifolds bounded by connected sums of AR homology spheres.

\begin{theorem}[{\cite[Theorem 3.19]{DSS26}}]\label{thm:conn_sum_SWF}
    Let $Y_0, Y_1$ be rational homology three-spheres with spin structures $\s_0, \s_1$ and $(X_i, m_i, n_i) = \SWF(Y_i, \s_i)$ for $i = 0, 1$. Then there is a $\Pin(2)$-equivariant stable homotopy equivalence 
    \[
    \SWF(Y_0 \mathop{\#} Y_1, \s_0 \mathop{\#} \s_1) \simeq \SWF(Y_0, \s_0) \wedge \SWF(Y_1, \s_1).
    \]
\end{theorem}

\begin{theorem}
    Let $Y_1, \dots, Y_n$ be AR homology spheres with spin structures $\s_1, \dots, \s_n$, respectively, and let $(Y, \s) = \mathop{\#}_{i = 1}^n (Y_i, \s_i)$. Let $(X, \mathfrak{t})$ be a compact, smooth, indefinite, spin $4$-manifold with boundary $\partial (X, \mathfrak{t}) = (Y, \s)$. Suppose that the intersection form of $X$ is isomorphic to $p(-E_8) \oplus q\left(\begin{smallmatrix} 0 & 1 \\ 1 & 0 \end{smallmatrix}\right)$ $(p \ge 0, q > 0)$ such that $p + \sum_i\bar{\mu}(Y_i,\s_i) \ge 2$. Then there exists a Furuta--Mahowald class of level-$\left(\frac{1}{2}\left(p + \sum_{i = 1}^n \bar{\mu}(Y_i, \s_i)\right), q + m\right)$, where $m$ is the number of components $Y_i$ satisfying $-\bar{\mu}(Y_i, \s_i) < \delta(Y_i, \s_i)$.
\end{theorem}

\begin{proof}
    By Theorem~\ref{thm:conn_sum_SWF}, there exists a $\Pin(2)$-equivariant stable homotopy equivalence
    \[
    f: \SWF(Y, \s) \simeq \bigwedge_{i = 1}^n \SWF(Y_i, \s_i).
    \]
    By Theorem~\ref{thm:admissible_map}, for each $i$, we have an admissible morphism for $\SWF(Y_i, \s_i)$. Smashing these maps together yields an admissible morphism
    \[
    g: \bigwedge_{i = 1}^n \SWF(Y_i, \s_i) \to \bigwedge_{i = 1}^n \Sigma^{-\frac{1}{2}\bar{\mu}(Y_i, \s_i)\mathbb{H}}S^{\epsilon_i \tilde{\R}} = \Sigma^{-\frac{1}{2}\sum_{i = 1}^n \bar{\mu}(Y_i, \s_i)\mathbb{H}}S^{m\tilde{\R}},
    \]
    where $\epsilon_i = 0$ if $-\bar{\mu}(Y_i, \s_i) = \delta(Y_i, \s_i)$, and $\epsilon_i = 1$ if $-\bar{\mu}(Y_i, \s_i) < \delta(Y_i, \s_i)$.
    
    On the other hand, Theorem~\ref{thm:rel_BF} provides the relative Bauer--Furuta invariant
    \[
    \Phi_X(\mathfrak{t}): \Sigma^{\frac{1}{2} p \mathbb{H}} S^0 \to \Sigma^{q\tilde{\R}}\SWF(Y, \s).
    \]
    Composing $\Phi_X(\mathfrak{t})$ with $\Sigma^{q\tilde{\R}}f$ and $\Sigma^{q\tilde{\R}}g$, we obtain an admissible morphism
    \[
    \Sigma^{\frac{1}{2} p \mathbb{H}} S^0 \to \Sigma^{-\frac{1}{2}\sum_{i = 1}^n \bar{\mu}(Y_i, \s_i)\mathbb{H}}S^{(q + m)\tilde{\R}}.
    \]
    This naturally defines an admissible morphism from $S^{\frac{1}{2}(p + \sum_{i = 1}^n \bar{\mu}(Y_i, \s_i))\mathbb{H}}$ to $S^{(q + m)\tilde{\R}}$. Applying Proposition~\ref{prop:equivalence_of_morphism}, this morphism yields the desired Furuta--Mahowald class.
\end{proof}

By combining the Furuta--Mahowald class constructed for connected sums with the necessary conditions from Theorem~\ref{thm:inequality}, we obtain Theorem~\ref{thm:conn_sum}, which serves as a natural generalization of Theorem~\ref{thm:main_inequality}. The proof is essentially identical to that of Theorem~\ref{thm:main_inequality}.

\section{Review of $\Pin(2)$-equivariant $KO$-theory}

\subsection{Some facts about equivariant $KO$-theory}
In this section, following the expository style of \cite{Man14} and \cite{Lin15}, we briefly recall the basic definitions and collect several fundamental facts regarding equivariant $KO$-theory that will be used throughout this paper.

Let $G$ be a compact topological group and $X$ a compact $G$-space. The equivariant $KO$-theory of $X$, denoted by $KO_G(X)$, is the Grothendieck group associated with $G$-equivariant real vector bundles on $X$. Similarly, we denote by $K_G(X)$ and $KSp_G(X)$ the Grothendieck groups associated with $G$-equivariant complex and quaternionic vector bundles on $X$, respectively. When $G$ is trivial, we simply write $KO(X)$, $K(X)$, and $KSp(X)$. When $X$ is a point, $RO(G) = KO_G(\pt)$ is the real representation ring of $G$. In general, $KO_G(X)$ is an algebra over $RO(G)$.

\begin{fact}
    A continuous $G$-map $f: X \to X'$ induces a map $f^*: KO_G(X') \to KO_G(X)$.
\end{fact}

\begin{fact}
    If two continuous $G$-maps $f, g \colon X \to X'$ are $G$-homotopic, they induce the same homomorphism $f^* = g^* \colon KO_G(X') \to KO_G(X)$.
\end{fact}

\begin{fact}
    For every subgroup $H \subset G$, we have functorial restriction maps $KO_G(X) \to KO_H(X)$.
\end{fact}

\begin{fact}
    If $G$ acts freely on $X$, then the pull-back map $KO(X / G) \to KO_G(X)$ is a ring isomorphism.
\end{fact}

\begin{fact}\label{fact:trivial_action}
    Let $\Z \operatorname{Ir}_{\R}$, $\Z \operatorname{Ir}_{\C}$, and $\Z \operatorname{Ir}_{\mathbb{H}}$ be the free abelian groups generated by the isomorphism classes of irreducible $G$-representations of the respective types. If $G$ acts trivially on $X$, then there is an isomorphism of $\Z$-modules:
    \[
    KO_G(X) \cong ( KO(X) \otimes \Z \operatorname{Ir}_{\R} ) \oplus ( K(X) \otimes \Z \operatorname{Ir}_{\C} ) \oplus ( KSp(X) \otimes \Z \operatorname{Ir}_{\mathbb{H}} ).
    \]
\end{fact}

The following fact is stated for complex equivariant $K$-theory in \cite[Section 2, Example (iii)]{Seg68}. Since Segal's method establishes a bijection between the isomorphism classes of vector bundles independent of the choice of base field, it induces an isomorphism of the associated equivariant $KO$-groups.

\begin{fact}\label{fact:subgroup}
    Let $H$ be a closed subgroup of a compact group $G$, and let $X$ be a compact $H$-space. Then the natural inclusion $i \colon X \hookrightarrow G \times_H X$ induces an isomorphism
    \[
        KO_G(G \times_H X) \xrightarrow{\cong} KO_H(X).
    \]
    Specifically, this isomorphism is given by the composition $i^* \circ r^G_H$, where $r^G_H \colon KO_G(G \times_H X) \to KO_H(G \times_H X)$ is the restriction of the group action from $G$ to $H$.
\end{fact}

Now suppose that $X$ has a distinguished base point $x_0$ that is fixed by $G$. We define the reduced equivariant $KO$-theory of $X$, denoted by $\KO_G(X)$, as the kernel of the restriction map $KO_G(X) \to KO_G(x_0)$.

\begin{fact}
    If the action of $G$ on $X$ is free away from the base point, then the pull-back map $\KO(X / G) \to \KO_G(X)$ is a ring isomorphism.
\end{fact}

\begin{fact}
    There is a natural isomorphism $\KO_G(X \vee X') \cong \KO_G(X) \oplus \KO_G(X')$.
\end{fact}

\begin{fact}\label{fact:natural_product}
    There is a natural product map $\KO_G(X) \otimes \KO_G(X') \to \KO_G(X \wedge X')$.
\end{fact}

\begin{fact}\label{fact:Bott_isom}
    Let $V$ be a real representation space of $G$. Suppose that the dimension $n$ of $V$ is divisible by $8$ and $V$ is a spin representation, meaning that the group action $G \to \SO(V)$ lifts to $\Spin(V)$. Then we have the equivariant Bott isomorphism $\phi_V : \KO_G(X) \xrightarrow{\cong} \KO_G(\Sigma^V X)$, given by multiplication by the Bott class $b_V \in \KO_G(V^+)$ via the natural map $\KO_G(V^+) \otimes \KO_G(X) \to \KO_G(\Sigma^V X)$. This isomorphism is functorial with respect to pointed maps $f : X \to X'$.
\end{fact}

\begin{fact}\label{fact:Bott_res}
    Bott classes behave well under restriction maps; that is, for a restriction map $r$, we have $r(b_V) = b_{r(V)}$.
\end{fact}

\begin{fact}\label{fact:spin_rep}
    For any real $G$-representation $V$, the direct sum $2V$ is orientable and $4V$ is a spin representation.
\end{fact}

\begin{fact}
    Let $A \subset X$ be a closed $G$-subspace (containing the base point). Then there exists a long exact sequence:
    \[
    \cdots \to \KO_G^i(X \cup_A CA) \to \KO_G^i(X) \to \KO_G^i(A) \to \KO_G^{i + 1}(X \cup_A CA) \to \cdots
    \]
    where $CA$ denotes the cone on $A$.
\end{fact}

The following fact is a pointed version of Fact~\ref{fact:subgroup}.

\begin{fact}\label{fact:subgroup_pointed}
    Let $H$ be a closed subgroup of a compact group $G$, and let $X$ be a pointed compact $H$-space. Then the isomorphism of Fact~\ref{fact:subgroup} induces an isomorphism of reduced groups:
    \[
        \widetilde{KO}_G(G_+ \wedge_H X) \xrightarrow{\cong} \widetilde{KO}_H(X),
    \]
    where $G_+ \wedge_H X = (G \times_H X)/(G \times_H \{x_0\})$.
\end{fact}

\begin{proof}
    By the definition of reduced $KO$-theory, we have
    \begin{align*}
        \widetilde{KO}_G(G_+ \wedge_H X) &= \ker\left(i_G^* \colon KO_G(G \times_H X) \to KO_G(G \times_H \{x_0\})\right), \\
        \widetilde{KO}_H(X) &= \ker\left(i_H^* \colon KO_H(X) \to KO_H(\{x_0\})\right),
    \end{align*}
    where $x_0$ is the basepoint of $X$, and $i_G \colon G \times_H \{x_0\} \hookrightarrow G \times_H X$ and $i_H \colon \{x_0\} \hookrightarrow X$ are the natural inclusion maps. Since the underlying pullback maps of vector bundles and the restriction map commute, the following diagram is commutative:
    \[
    \begin{tikzcd}[row sep=large, column sep=huge]
        KO_G(G \times_H X) \arrow[r, "\cong"] \arrow[d, "i_G^*"']
            & KO_H(X) \arrow[d, "i_H^*"]
            \\
        KO_G(G \times_H \{x_0\}) \arrow[r, "\cong"]
            & KO_H(\{x_0\}).
    \end{tikzcd}
    \]
    Therefore, the horizontal isomorphism restricts to an isomorphism between the kernels of the vertical maps, yielding $\widetilde{KO}_G(G_+ \wedge_H X) \xrightarrow{\cong} \widetilde{KO}_H(X)$.
\end{proof}

\begin{fact}
    The complex and real representation rings of $S^1$ are given by
    \begin{align*}
        R(S^1) &\cong \Z[\theta, \theta^{-1}], \\
        RO(S^1) &\cong \Z[V],
    \end{align*}
    where $\theta$ is the standard one-dimensional complex representation of $S^1$, and $V$ is the standard two-dimensional real representation of $S^1$.
\end{fact}

\subsection{$\Pin(2)$-equivariant $KO$-theory}

Let $G = \Pin(2)$, and let $D$, $K$, and $H$ be $G$-representations defined as follows. We set $D = \tilde{\mathbb{R}}$ and $H = \mathbb{H}$. The representation space of $K$ is $\C$ (identified with $\mathbb{R}^2$), where the $G$-action is given by $S^1 \subset G$ acting by multiplication by $z^2$, and $j \in G$ acting as a reflection across the diagonal. We adopt these notations for the representations following \cite{Lin15} and \cite{Sch03}.

\begin{theorem}[{\cite[Proposition 5.1]{Sch03}}]
    There is a ring isomorphism
    \[
    RO(\Pin(2)) \cong \Z[D, K, H] / (D^2 - 1, DK - K, DH - H, H^2 - 4(1 + D + K)).
    \]
\end{theorem}

Following the notations of \cite{Lin15} and \cite{Sch03}, we denote $\KO_G((kD + lH)^+)$ by $KO_G(kD + lH)$. By Fact~\ref{fact:spin_rep}, and since $H$ factors through $\Spin(4)$, the representations $8D$, $4D + H$, and $2H$ are all spin representations. Therefore, by Fact~\ref{fact:Bott_isom}, we can choose Bott classes $b_{8D} \in KO_G(8D)$, $b_{4D+H} \in KO_G(4D + H)$, and $b_{2H} \in KO_G(2H)$. Since their underlying real dimensions are all even, these Bott classes commute with each other with respect to the multiplicative structure induced by the natural product map (Fact~\ref{fact:natural_product}). Furthermore, we can choose them to be compatible with each other, which means that $b_{8D} \cdot b_{2H} = b_{4D + H}^2$ \cite[Section 2.2]{Lin15}.

\begin{definition}
    For $k, l \in \Z$, we define $KO_G(kD+lH)$ to be $KO_G((k + 8a)D + (l + 2b)H)$ for any integers $a, b$ such that $k + 8a \ge 0$ and $l + 2b \ge 0$. Here, the identification between the groups for different choices of $a$ and $b$ is given by multiplication by the compatible Bott classes $b_{8D}$ and $b_{2H}$.
\end{definition}

\begin{definition}[{\cite[Corollary 5.8]{Sch03}}]
    We define the bi-graded ring
    \[
    KO_G(*, *) := \bigoplus_{k, l \in \Z} KO_G(kD + lH).
    \]
    For arbitrary indices $k_0, l_0, k_1, l_1 \in \Z$, we choose non-negative integers $a_0, b_0, a_1, b_1$ such that $k_i + 8a_i \ge 0$ and $l_i + 2b_i \ge 0$ for $i=0, 1$. The product of elements $x \in KO_G(k_0D + l_0 H)$ and $y \in KO_G(k_1D + l_1H)$ is then extended via the following composition of maps:
    \begin{align*}
        &KO_G(k_0D + l_0 H) \otimes KO_G(k_1D + l_1H) \\
        &\xrightarrow{\cong \text{ (Bott)}} KO_G((k_0 + 8a_0)D + (l_0 + 2b_0)H) \otimes KO_G((k_1 + 8a_1)D + (l_1 + 2b_1)H) \\
        &\xrightarrow{\text{product}} KO_G((k_0 + k_1 + 8(a_0 + a_1))D + (l_0 + l_1 + 2(b_0 + b_1))H) \\
        &\xrightarrow{\cong \text{ (Bott)}^{-1}} KO_G((k_0 + k_1)D + (l_0 + l_1)H).
    \end{align*}
    This extension is well-defined and independent of the choices of $a_i$ and $b_i$ because the Bott classes lie in the center of the $RO(G)$-algebra \cite[Section 2.2]{Lin15}.
\end{definition}

\begin{remark}
    As shown in \cite[Corollary 5.8]{Sch03}, there exists a surjective ring homomorphism from a commutative polynomial ring to $KO_G(*, *)$. Consequently, this bi-graded ring is commutative.
\end{remark}

Next, we introduce the elements $\gamma(D) \in KO_G(-D)$, $\eta(D) \in KO_G(D)$, and $\lambda(D), c \in KO_G(4D)$. The element $\gamma(D) \in KO_G(-D)$ is defined to be the unique element satisfying the equation $\gamma(D) b_{8D} = i^*b_{8D}$ in $KO_G(7D)$, where $i \colon 7D^+ \hookrightarrow 8D^+$ is the standard inclusion. It is worth noting that, by the equivariant Hopf theorem (see, for instance, \cite[Section 8.4]{tD79}), the $G$-homotopy class of such an inclusion map is unique, so the definition of $\gamma(D)$ is independent of the choice of the inclusion. Moreover, for the standard inclusion $i \colon kD + lH \hookrightarrow (k+1)D + lH$, the induced homomorphism $i^*$ acts as multiplication by $\gamma(D)$. This follows from the naturality of the product map and the equivariant Hopf theorem, which ensure the commutativity of the following diagram:
\[
\begin{tikzcd}[row sep=large, column sep=huge]
    KO_G((k + 9)D + lH) \arrow[r, "(i \oplus \mathrm{id}_{8D})^* = (\mathrm{id} \oplus i_{7D \hookrightarrow 8D})^*"]
    & KO_G((k + 8)D + lH) \\
    KO_G((k + 1)D + lH) \otimes KO_G(8D) \arrow[r, "\mathrm{id} \otimes (\cdot \gamma(D))"'] \arrow[u, "\text{product}"]
    & KO_G((k + 1)D + lH) \otimes KO_G(7D) \arrow[u, "\text{product}"']
\end{tikzcd}
\]
Since the left vertical map is surjective by equivariant Bott periodicity, $i^*$ is completely determined by this diagram, verifying $i^*(x) = x \cdot \gamma(D)$. Similarly, we define $\gamma(H)\in KO_G(-H)$.

The element $\eta(D) \in KO_G(D)$ is defined as the Hurewicz image of $\tilde{\eta}(D) \in \pi_G^0(D)$, which is represented by the Hopf map $\eta \colon H^+ \to 3D^+$ \cite[Proposition 4.3(2), Section 5]{Sch03}.

For the elements $\lambda(D), c \in KO_G(4D)$, since $4D \oplus \R^4$ and $8D$ are spin representations, Facts~\ref{fact:Bott_isom} and \ref{fact:trivial_action} yield the following isomorphisms:
\begin{align*}
    \KO_G((4D)^+) &\cong \KO_G((8D \oplus \R^4)^+) \cong \KO_G(S^4) \\
    &\cong \left( \KO(S^4) \otimes \Z \operatorname{Ir}_{\R} \right) \oplus \left( \widetilde{K}(S^4) \otimes \Z \operatorname{Ir}_{\C} \right) \oplus \left( \widetilde{KSp}(S^4) \otimes \Z \operatorname{Ir}_{\mathbb{H}} \right).
\end{align*}
We can choose suitable Bott classes so that under these isomorphisms, $\lambda(D)$ corresponds to $([V_{\mathbb{H}}] - 4_\R) \otimes 1_{\R} \in \KO(S^4) \otimes \Z \operatorname{Ir}_{\R}$, and $c$ corresponds to $([V_{\mathbb{H}}] - \mathbb{H}) \otimes H \in \widetilde{KSp}(S^4) \otimes \Z \operatorname{Ir}_{\mathbb{H}}$. Here, $V_{\mathbb{H}}$ is the quaternionic Hopf bundle over $S^4 \cong \mathbb{H}P^1$, and $\mathbb{H}$ and $\R$ denote the respective trivial vector bundles \cite[Section 2.2]{Lin15}.
The following theorems provide the explicit algebraic structure of the bi-graded ring $KO_G(*, *)$. These relations, originally calculated and proved in \cite{Sch03}, serve as the computational foundation for the $\ko$-invariant in subsequent sections.

\begin{theorem}[{\cite[Proposition 5.5]{Sch03}, \cite[Theorem 2.13]{Lin15}}]\label{thm:KO_G(lD)}
    As $\Z$-modules we have the following isomorphisms:
    \begin{enumerate}[label=(\arabic*)]
        \item $KO_G(\pt) \cong RO(\Pin(2)) \cong \Z[D,A,B]/(D^2 - 1,DA - A,DB - B,B^2 - 4 (A - 2B))$, where $A = K-(1 + D)$ and $B = H - 2(1 + D)$.
        \item $KO_G(-lD)\cong \Z \oplus \bigoplus_{n\ge 1} \Z/2$ for $l = 1,2$, generated by $\gamma(D)^{|l|}$ and $\gamma(D)^{|l|} A^n$.
        \item $KO_G(D) \cong \Z$, generated by $\eta(D)$.
        \item $KO_G(lD) \cong \Z \oplus \bigoplus_{m\ge 0} \Z/2$ for $l = 2, 3$. The generators are $\eta(D)^{2}$ and $\gamma(D)^2 A^m c$ for $l = 2$; $\gamma(D)\lambda(D)$ and $\gamma(D) A^m c$ for $l = 3$.
        \item $KO_G(4D)$ is freely generated by $\lambda(D), D\lambda(D), A^n \lambda(D)$ and $A^m c$ for $m \ge 0$ and $n \ge 1$.
        \item $KO_G(5D) \cong \Z$, generated by $\eta(D)\lambda(D)$.
    \end{enumerate}
\end{theorem}

\begin{theorem}[{\cite[Proposition 5.7]{Sch03}, \cite[Theorem 2.16]{Lin15}}]\label{thm:relations}
    The following relations hold for these elements:
    \begin{enumerate}[label=(\arabic*)]
        \item $H\lambda(D) = 4c$, $Hc = (A + 2 + 2D)\lambda(D)$, $Dc = c$.
        \item $(D + 1)\gamma(D) = 2A\gamma(D) = B\gamma(D) = 0$.
        \item $(D + 1)\eta(D) = A\eta(D) = B\eta(D) = 0$.
        \item $\gamma(D)\eta(D) = 1 - D$, $\gamma(D)\lambda(D) = \eta(D)^3$.
        \item $\gamma(D)^8 b_{8D} = 8(1 - D)$, $ \gamma(H)^2 b_{2H} = K - 2H + D + 5$.
        \item $\gamma(H + 4D)b_{H + 4D} = 4(1 - D)$.
        \item $\eta(D)\lambda(D) = \gamma(D)^3 b_{8D}$, $\eta(D)c = 0$.
        \item $\gamma(H)\lambda(H) = 4 - H$ and $\gamma(H)c(H) = H - 1 - D - K$.
    \end{enumerate}
\end{theorem}

\section{Calculation of $\ko$-invariant for AR homology spheres}

\subsection{Definition of the $\ko$-invariant}

In this section, we briefly review the definition of the $\ko$-invariant introduced in \cite{Lin15}.

\begin{definition}[{\cite[Definition 5.1]{Lin15}}]\label{def:varphi}
    For $l = -2, -1, \dots, 5$ we define the group homomorphism $\varphi_l: KO_G(lD) \to \Z$ as follows:
    \begin{enumerate}[label=(\arabic*)]
        \item For $l = 0$, $\varphi_0(D) = -1, \varphi_0(A) = \varphi_0(B) = 0$, then extend $\varphi_0$ by the multiplicative structure on $RO(G)$.
        \item For $l = -1, -2$, $\varphi_l(\gamma(D)^{|l|}) = 1, \varphi_l(\gamma(D)^{|l|}A^n) = 0$ for $n \ge 1$.
        \item For $l = 1$, $\varphi_1(\eta(D)) = 1$.
        \item For $l = 2$, $\varphi_2(\eta(D)^2) = 1, \varphi_2(\gamma(D)^2 A^m c) = 0$ for $m \ge 0$.
        \item For $l = 3$, $\varphi_3(\gamma(D)\lambda(D)) = 1, \varphi_3(\gamma(D) A^m c) = 0$ for $m \ge 0$.
        \item For $l = 4$, $\varphi_4(\lambda(D)) = 1, \varphi_4(D\lambda(D)) = -1, \varphi_4(A^n \lambda(D)) = 0, \varphi_4(A^m c) = 0$ for $n \ge 1, m\ge 0$.
        \item For $l = 5$, $\varphi_5(\eta(D)\lambda(D)) = 1$.
    \end{enumerate}
    For the other $l \in \Z$, we use the Bott isomorphism to identify $KO_G(lD)$ with $KO_G((l - 8k)D)$ for $-2 \le l - 8k \le 5$ and apply the above definition. Also, for any $a \in RO(G)$, $b \in KO_G(lD)$, we have $\varphi_0(a) \cdot \varphi_l(b) = \varphi_l(a \cdot b)$ \cite[Lemma 5.2]{Lin15}.
\end{definition}

Let $X$ be a space of type SWF at level $l$. A choice of $G$-homotopy equivalence $X^{S^1} \simeq (lD)^+$ composed with the natural inclusion $X^{S^1} \hookrightarrow X$ gives an inclusion map $i: (lD)^+ \to X$. This map induces a homomorphism $i^*: \KO_G(X) \to KO_G(lD)$, and we consider the composition $\varphi_l \circ i^*: \KO_G(X) \to \Z$. 

\begin{lemma}[{\cite[Proposition 5.6]{Lin15}}]
    The submodule $\operatorname{Im}(i^*)$ and the map $\varphi_l\circ i^*$ are both independent of the choice of the $G$-homotopy equivalence $X^{S^1} \simeq (lD)^+$. Moreover, we have $\operatorname{Im}(\varphi_l\circ i^*)=(2^k)$ for some $k\in \Z_{\ge 0}$.    
\end{lemma}

\begin{definition}[{\cite[Definition 5.7]{Lin15}}]\label{def:kappao}
    For a $G$-space $X$ of type SWF at level $l$, we define $\ko(X)$ to be the integer $k$ such that $\operatorname{Im}(\varphi_l\circ i^*)=(2^k)$.
\end{definition}

To formulate the invariant, we prepare the following numerical sequences.
For $j, k \in \Z$ with $j \ge 1$, we define $\beta^j_k := \sum_{i = 0}^{j - 1} \alpha_{k - i}$, where 
\[
\alpha_i = 
\begin{cases} 
1 & \text{for } i \equiv 1, 2, 3, 5 \pmod{8}, \\ 
0 & \text{for } i \equiv 0, 4, 6, 7 \pmod{8}. 
\end{cases}
\]

We recall the behavior of the $\ko$-invariant under suspensions.

\begin{proposition}[{\cite[Proposition 5.9]{Lin15}}]
    Let $X$ be a space of type SWF at level $l$. Then we have the following:
    \begin{enumerate}[label=(\arabic*)]
        \item $\ko(\Sigma^{8D}X) = \ko(X)$.
        \item $\ko(\Sigma^{2H}X) = \ko(X) + 2$.
        \item $\ko(\Sigma^{H + 4D}X) = \ko(X) + 3 - \beta^4_{l + 4}$.
    \end{enumerate}
\end{proposition}

These suspension properties allow us to extend the definition of the $\ko$-invariant from spaces to spectra.

\begin{definition}[{\cite[Definition 5.11]{Lin15}}]\label{def:kappao_spectrum}
    For a spectrum $\mathcal{X} = (X, m, n) \in \mathfrak{C}_G$, where $X$ is a space of type SWF, we define
    \[
    \ko(\mathcal{X}) := \ko(\Sigma^{(8M - m)D}\Sigma^{(2N - n')H}X) - 2N - s
    \]
    for $M, N, n' \in \Z$ and $s \in [0, 1)$ such that $8M - m \ge 0$, $2N - n' \ge 0$, and $n = n' + s$. The well-definedness of this invariant follows from \cite[Proposition 5.12]{Lin15}.
\end{definition}

As a consequence of this definition, the suspension formulas naturally generalize to the spectrum level.

\begin{proposition}[{\cite[Proposition 5.13]{Lin15}}]\label{prop:suspension_ko}
    For a spectrum $\mathcal{X} = (X, m, n) \in \mathfrak{C}_G$, where $X$ is a space of type SWF at level $l$, we have the following:
    \begin{enumerate}[label=(\arabic*)]
        \item $\ko(\Sigma^{8D}\mathcal{X}) = \ko(\mathcal{X})$.
        \item $\ko(\Sigma^{2H}\mathcal{X}) = \ko(\mathcal{X}) + 2$.
        \item $\ko(\Sigma^{H + 4D}\mathcal{X}) = \ko(\mathcal{X}) + 3 - \beta^4_{l - m + 4}$.
    \end{enumerate}
\end{proposition}

Finally, by evaluating this invariant on the Seiberg--Witten Floer spectrum, we define the $\ko$-invariants for a rational homology $3$-sphere.

\begin{definition}[{\cite[Definition 5.14]{Lin15}}]
    Let $Y$ be an oriented rational homology $3$-sphere with a spin structure $\s$. We define $\ko_i(Y, \s) := \ko(\Sigma^{iD}\SWF(Y, \s))$ for any $i \in \Z_{\ge 0}$. Then $\ko_{i + 8}(Y, \s) = \ko_i(Y, \s)$, which allows us to define $\ko_i(Y, \s)$ for $i \in \Z / 8$.
\end{definition}

\subsection{Calculation of the $\ko$-invariants for AR homology spheres}

In this subsection, we compute the $\ko$-invariants of AR homology spheres. Our strategy, inspired by the computation of the $\kappa$-invariants in \cite{DSS26}, is to utilize the local maps between the SWF spectra of AR homology spheres and the fundamental $G$-spectra. We first establish that a local map between two spectra yields an inequality between their $\ko$-invariants.

\begin{proposition}\label{prop:ko_local}
    Let $\mathcal{X}_0 = (X_0, m_0, n_0), \mathcal{X}_1 = (X_1, m_1, n_1) \in \mathfrak{C}_G$ be two spectra, where $X_0$ and $X_1$ are spaces of type SWF. If there exists a local map $F: \mathcal{X}_0 \to \mathcal{X}_1$, then $\ko_i(\mathcal{X}_0) \le \ko_i(\mathcal{X}_1)$ for any $i \in \Z/8$.
\end{proposition}

\begin{proof}
    We prove the case $i = 0$; the other cases can be proved similarly. By taking sufficiently many suspensions, we can choose a representative map $\tilde{F}$ of $F$ as follows:
    \[
        \tilde{F}: \Sigma^{(8a - m_0)D}\Sigma^{(2b - n_0')H} X_0 \to \Sigma^{(8a - m_1)D}\Sigma^{(2b - n_1') H} X_1
    \]
    where $n_0'$ and $n_1'$ are the greatest integers less than or equal to $n_0$ and $n_1$, respectively. Taking the $S^1$-fixed point sets, we obtain the following commutative diagram:
    \[
    \begin{tikzcd}
        \Sigma^{(8a - m_0)D}\Sigma^{(2b - n_0')H} X_0 \arrow[r, "\tilde{F}"]
            &\Sigma^{(8a - m_1)D}\Sigma^{(2b - n_1') H} X_1
            \\
        \Sigma^{(8a - m_0)D}X_0^{S^1} \arrow[u, hook, "i_0"] \arrow[r, "\simeq"', "\tilde{F}^{S^1}"]
            &\Sigma^{(8a - m_1)D} X_1^{S^1} \arrow[u, hook, "i_1"']
    \end{tikzcd}
    \]
    Applying $\KO_G$, we have:
    \[
    \begin{tikzcd}
        \KO_G(\Sigma^{(8a - m_0)D}\Sigma^{(2b - n_0')H} X_0) \arrow[d, "i_0^*"']
            &\KO_G(\Sigma^{(8a - m_1)D}\Sigma^{(2b - n_1') H} X_1) \arrow[l, "\tilde{F}^*"'] \arrow[d, "i_1^*"]
            \\
        \KO_G(\Sigma^{(8a - m_0)D}X_0^{S^1})
            &\KO_G(\Sigma^{(8a - m_1)D} X_1^{S^1}). \arrow[l, "\cong", "(\tilde{F}^{S^1})^*"']
    \end{tikzcd}
    \]
    By the commutativity of the diagram, we have $i_0^* \circ \tilde{F}^* = (\tilde{F}^{S^1})^* \circ i_1^*$. This implies that under the isomorphism $(\tilde{F}^{S^1})^*$, the image of $i_1^*$ is mapped into the image of $i_0^*$; that is, 
    \[
        (\tilde{F}^{S^1})^*(\operatorname{Im}(i_1^*)) \subset \operatorname{Im}(i_0^*).
    \]
    Therefore,
    \[
        \ko(\Sigma^{(8a - m_0)D}\Sigma^{(2b - n_0')H} X_0) \le \ko(\Sigma^{(8a - m_1)D}\Sigma^{(2b - n_1') H} X_1).
    \]
    By Definition~\ref{def:kappao_spectrum}, this yields $\ko_0(\mathcal{X}_0) \le \ko_0(\mathcal{X}_1)$.
\end{proof}

Following the approach used in \cite{DSS26}, we first determine the $\ko$-invariants of the fundamental $G$-spaces $A_n$.

\begin{proposition}\label{prop:ko_A_n}
    Let $A_n = S^{n\C} \cup_{S^0} jS^{n\C}$ be the $G$-space defined in Section 3. Then the values of $\ko_i(A_n)$ for $i \in \Z/8$ are given as follows:
    \[
    \begin{array}{c|cccccccc}
        i \pmod 8 & 0 & 1 & 2 & 3 & 4 & 5 & 6 & 7 \\ \hline
        \ko_i(A_0) & 0 & 0 & 0 & 0 & 0 & 0 & 0 & 0 \\
        \ko_i(A_1) & 1 & 1 & 1 & 0 & 0 & 0 & 0 & 0 \\
        \ko_i(A_n) \ (n \ge 2) & 1 & 1 & 1 & 0 & 1 & 0 & 0 & 0
    \end{array}
    \]
\end{proposition}

\begin{proof}
    For $n = 0$, the claim holds because $A_0 = S^0$. Therefore, we may assume $n \ge 1$. From the proof of \cite[Lemma 7.8]{DSS26}, we have the cofiber sequence:
    \[
    S^{n\C}\vee S^{n\C} \to A_n \to S^D \to \Sigma^D(S^{n\C}\vee S^{n\C})\to\cdots
    \]
    where $j$ interchanges the two wedge summands. Restricting to the $S^1$-fixed-point sets, we obtain the cofiber sequence
    \[
    S^0\vee S^0\to A_n^{S^1}\to (S^D)^{S^1}\to \cdots.
    \]
    Note that $A_n^{S^1} = S^0$, $(S^D)^{S^1} = S^D$. Taking the $iD$-th suspension (for $0 \le i < 8$), we obtain:
    \begin{equation}
        \begin{tikzcd}
        \Sigma^{iD}(S^{n\C} \vee S^{n\C}) \arrow[r]
            &\Sigma^{iD}(A_n) \arrow[r] 
            &\Sigma^{iD}(S^D)
            \\
        \Sigma^{iD}(S^0 \vee S^0) \arrow[r] \arrow[u, hook]
            &\Sigma^{iD}S^0 \arrow[r] \arrow[u, hook, "\iota"']
            &\Sigma^{iD}(S^D) \arrow[u, equal]
        \end{tikzcd}
    \end{equation}
    This sequence induces the following diagram:
    \begin{equation} \label{eq:KO_diagram1}
        \begin{tikzcd}
        \KO_G(\Sigma^{iD}(S^{n\C} \vee S^{n\C})) \arrow[d] 
            & \KO_G(\Sigma^{iD}(A_n)) \arrow[l] \arrow[d, "\iota^*"] 
            & \KO_G(S^{(i+1)D}) \arrow[l] \arrow[d, equal] 
            \\
         \KO_G(\Sigma^{iD}(S^0\vee S^0))
            & \KO_G(S^{iD}) \arrow[l] 
            & \KO_G(S^{(i+1)D}) \arrow[l, "\cdot\:\gamma(D)"'] 
        \end{tikzcd}
    \end{equation}
    Here, there is a $G$-equivariant homeomorphism:
    \[
    \Sigma^{iD}(G_+ \wedge_{S^1} S^{n\C}) \cong G_+ \wedge_{S^1} \Sigma^{\R^i} S^{n\C}.
    \]
    Explicitly, this homeomorphism is given by $y \wedge [g, x] \mapsto [g, g^{-1}y \wedge x]$, with its inverse given by $[g, y \wedge x] \mapsto gy \wedge [g, x]$.
    Applying Fact~\ref{fact:subgroup_pointed} with $X = \Sigma^{\R^i}S^{n\C}$, $H = S^1$, and $G = \mathrm{Pin}(2)$, we obtain
    \[
    \KO_G(G_+ \wedge_{S^1} \Sigma^{\R^i} S^{n\C}) \cong \KO_{S^1}(\Sigma^{\R^i} S^{n\C}).
    \]
    Thus, we have an isomorphism:
    \[
    \KO_G(\Sigma^{iD}(S^{n\C} \vee S^{n\C})) \cong \KO_{S^1}(\Sigma^{\R^i} S^{n\C}).
    \]
    Note that this isomorphism holds for $n = 0$ as well. Furthermore, the following two diagrams commute:
    \[
    \begin{tikzcd}
        \KO_G(\Sigma^{iD}(S^0\vee S^0)) \arrow[dr, "\cong"] 
            \\
        \KO_G(\Sigma^{iD}S^0) \arrow[r, "r^G_{S^1}"] \arrow[u, "p^*"]
            & \KO_{S^1}(\Sigma^{\R^i}S^0)
    \end{tikzcd}
    \]
    \[
    \begin{tikzcd}
        \KO_G(\Sigma^{iD}(S^{n\C} \vee S^{n\C})) \arrow[d, "i^*"] \arrow[r, "\cong"]
            &\KO_{S^1}(\Sigma^{\R^i}S^{n\C}) \arrow[d, "i^*"]
            \\
        \KO_G(\Sigma^{iD}(S^0 \vee S^0)) \arrow[r, "\cong"]
            &\KO_{S^1}(\Sigma^{\R^i}S^0)
    \end{tikzcd}
    \]
    where the maps denoted by $i$ are the respective inclusions, $p: S^0 \vee S^0 \to S^0$ is the projection, and $r^G_{S^1}$ is the restriction map from $G$ to $S^1$. The commutativity of these diagrams follows from the fact that the isomorphisms above are given by the composition of the inclusion and the restriction maps, and that the restriction map commutes with the induced maps.

    Using these isomorphisms and diagrams, we can rewrite the diagram \eqref{eq:KO_diagram1} as follows:
    \begin{equation} \label{eq:KO_diagram2}
        \begin{tikzcd}
            \KO_{S^1}(\Sigma^{\R^i}S^{n\C}) \arrow[d, "i^*"] 
                & \KO_G(\Sigma^{iD}(A_n)) \arrow[l] \arrow[d, "\iota^*"] 
                & \KO_G(S^{(i+1)D}) \arrow[l] \arrow[d, equal] 
                \\
            \KO_{S^1}(S^i)
                & \KO_G(S^{iD}) \arrow[l, "r^G_{S^1}"] 
                & \KO_G(S^{(i+1)D}). \arrow[l, "\cdot\:\gamma(D)"'] 
        \end{tikzcd}
    \end{equation}
    If we can determine the image of $\varphi_i \circ \iota^*$, we can find the invariant (Definition~\ref{def:kappao}). Since the rightmost vertical map is the identity, we have $\operatorname{Im}(\cdot \gamma(D)) \subset \operatorname{Im}(\iota^*)$. Thus, the following claim holds.
    
    \begin{claim}\label{cl:above_bound}
        $\ko_i(A_n)$ is bounded above by $1, 1, 1, 0, 1, 0, 0, 0$ for $i \equiv 0, 1, 2, 3, 4, 5, 6, 7 \pmod 8$, respectively. In particular, this shows that $\ko_i(A_n) = 0$ for $i \equiv 3, 5, 6, 7 \pmod 8$.
    \end{claim}
        
    \begin{proof}[Proof of Claim]
        We use Theorems~\ref{thm:KO_G(lD)} and \ref{thm:relations} together with Definition~\ref{def:varphi}.
        
        \begin{description}
            \item[$i = 0$] 
            Since $\gamma(D) \eta(D) = 1 - D$, we have $1 - D \in \operatorname{Im}(\cdot \gamma(D)) \subset \operatorname{Im}(\iota^*)$. Thus, $\varphi_0(1 - D) = 2 \in \operatorname{Im}(\varphi_0 \circ \iota^*)$, which implies $\ko_0(A_n) \le 1$.
            
            \item[$i = 1$] 
            Since $\gamma(D) \eta(D)^2 = (1 - D)\eta(D)$, we have $(1 - D)\eta(D) \in \operatorname{Im}(\iota^*)$. Thus, $\varphi_1((1 - D)\eta(D)) = 2 \in \operatorname{Im}(\varphi_1 \circ \iota^*)$, which implies $\ko_1(A_n) \le 1$.
            
            \item[$i = 2$] 
            Since $\gamma(D) \cdot \gamma(D)\lambda(D) = \gamma(D) \eta(D)^3 = (1 - D)\eta(D)^2$, we have $(1 - D)\eta(D)^2 \in \operatorname{Im}(\iota^*)$. Thus, $\varphi_2((1 - D)\eta(D)^2) = 2 \in \operatorname{Im}(\varphi_2 \circ \iota^*)$, which implies $\ko_2(A_n) \le 1$.
            
            \item[$i = 3$] 
            Since $\gamma(D)\lambda(D) \in \operatorname{Im}(\iota^*)$ and $\varphi_3(\gamma(D)\lambda(D)) = 1 \in \operatorname{Im}(\varphi_3 \circ \iota^*)$, we obtain $\ko_3(A_n) \le 0$.
            
            \item[$i = 4$] 
            Since $\gamma(D) \eta(D) \lambda(D) = (1 - D)\lambda(D)$, we have $(1 - D)\lambda(D) \in \operatorname{Im}(\iota^*)$. Thus, $\varphi_4((1 - D)\lambda(D)) = 2 \in \operatorname{Im}(\varphi_4 \circ \iota^*)$, which implies $\ko_4(A_n) \le 1$.
            
            \item[$i = 5$] 
            Since $\gamma(D) \cdot \gamma(D)^2 b_{8D} = \gamma(D)^3 b_{8D} = \eta(D)\lambda(D) \in \operatorname{Im}(\iota^*)$ and $\varphi_5(\eta(D)\lambda(D)) = 1 \in \operatorname{Im}(\varphi_5 \circ \iota^*)$, we obtain $\ko_5(A_n) \le 0$.
            
            \item[$i = 6$] 
            Since $\gamma(D) \cdot \gamma(D) b_{8D} = \gamma(D)^2 b_{8D} \in \operatorname{Im}(\iota^*)$ and $\varphi_6(\gamma(D)^2 b_{8D}) = 1 \in \operatorname{Im}(\varphi_6 \circ \iota^*)$ (by the Bott isomorphism), we obtain $\ko_6(A_n) \le 0$.
            
            \item[$i = 7$] 
            Since $\gamma(D) b_{8D} \in \operatorname{Im}(\iota^*)$ and $\varphi_7(\gamma(D)b_{8D}) = 1 \in \operatorname{Im}(\varphi_7 \circ \iota^*)$, we obtain $\ko_7(A_n) \le 0$.
        \end{description}
    \end{proof}

    While we used the right square of the diagram~\eqref{eq:KO_diagram2} to find an upper bound for $\ko_i(A_n)$, we now use the left square to find a lower bound.

    \begin{claim}
        $\ko_i(A_n)$ is bounded below by $1, 1, 1, 0, 0, 0, 0, 0$ for $i \equiv 0, 1, 2, 3, 4, 5, 6, 7 \pmod 8$, respectively. Combining this with Claim~\ref{cl:above_bound}, we can determine $\ko_i(A_n)$ except for $i \equiv 4 \pmod 8$.
    \end{claim}

    \begin{proof}[Proof of Claim]
        By extending the left square of the diagram~\eqref{eq:KO_diagram2} with the forgetful map from $S^1$-equivariant to non-equivariant $\KO$-theory, we obtain the following commutative diagram:
        \begin{equation} \label{eq:KO_diagram3}
            \begin{tikzcd}
                \KO(S^{i + 2n}) \arrow[d, "i^*"']
                    & \KO_{S^1}(\Sigma^{\R^i}S^{n\C}) \arrow[d, "i^*"] \arrow[l, "r^{S^1}"']
                    & \KO_G(\Sigma^{iD}(A_n)) \arrow[l] \arrow[d, "\iota^*"]  
                    \\
                \KO(S^i)
                    & \KO_{S^1}(S^i) \arrow[l, "r^{S^1}"]
                    & \KO_G(S^{iD}). \arrow[l, "r^G_{S^1}"] 
            \end{tikzcd}
        \end{equation}
        Since $n \ge 1$, the leftmost inclusion $i: S^i \hookrightarrow S^{i + 2n}$ is nullhomotopic. Thus, the leftmost vertical map $i^*$ is the zero map. On the bottom row, $r^{S^1} \circ r^G_{S^1} = r^G$, which is the forgetful map from $G$-equivariant to non-equivariant $\KO$-theory. By the commutativity of the diagram, $r^G \circ \iota^* = 0$. This implies $\operatorname{Im}(\iota^*) \subset \ker(r^G)$.

        By the argument in the proof of \cite[Proposition 8.1]{Lin15}, the structure of $\ker(r^G: \KO_G(S^{iD}) \to \KO(S^i))$ as an $RO(G)$-submodule is given as follows:
        \begin{itemize}
            \item For $i = 0$, $\ker(r^G)$ is the submodule generated by $1 - D, A, B$.
            \item For $i = 1$, $\ker(r^G)$ is generated by $2\eta(D)$.
            \item For $i = 2$, $\ker(r^G)$ is generated by $2\eta(D)^2$ and $\gamma(D)^2c$.
            \item For $i = 4$, $\ker(r^G)$ is generated by $\lambda(D)-c, (1-D)\lambda(D), A\lambda(D)$, and $Ac$.
            \item For $i = 3 , 5, 6, 7$, $\ker(r^G) \cong \KO_G(S^{iD})$.
        \end{itemize}
        Since $\operatorname{Im}(\iota^*) \subset \ker(r^G)$, applying the homomorphism $\varphi_i$ to the generators of $\ker(r^G)$ listed above yields the minimum possible power of $2$ in $\operatorname{Im}(\varphi_i \circ \iota^*)$. By definition of the $\ko$-invariant, this provides the stated lower bounds.
    \end{proof}

    Next, we determine $\ko_4(A_n)$. If $n = 1$, we have $A_1 = \tilde{G}$, where $\tilde{G}$ is the $G$-space defined in \cite[Section 8.1]{Lin15}. By \cite[Proposition 8.1]{Lin15}, we obtain $\ko_4(A_1) = 0$. Thus, we may assume $n \ge 2$ in what follows.
    
    Suppose that $\ko_4(A_n) = 0$. Then there exists an element $x \in \operatorname{Im}(\iota^*) \subset KO_G(4D)$ such that $\varphi_4(x) = 1$. 
    Recall that $\operatorname{Im}(\iota^*) \subset \ker(r^G)$. Since $\ker(r^G)$ is a free $\Z$-module generated by $\lambda(D) - c, (1 - D)\lambda(D), A^k\lambda(D)$, and $A^l c$ for $k, l \ge 1$, we can express $x$ as
    \[
        x = a(\lambda(D) - c) + b((1 - D)\lambda(D)) + \sum_{k \ge 1} p_k A^k \lambda(D) + \sum_{l \ge 1} q_l A^l c
    \]
    for some integers $a, b, p_k, q_l \in \Z$. 
    Applying $\varphi_4$ to both sides yields $1 = a + 2b$. Thus, we have
    \begin{equation}\label{eq:a_odd}
        a = 1 - 2b,
    \end{equation}
    which implies that $a$ is odd.
    
    By extending the left square of the diagram~\eqref{eq:KO_diagram2} with the complexification map $c_\C$, we obtain the following commutative diagram:
    \begin{equation} \label{eq:KO_diagram4}
        \begin{tikzcd}
            \K_{S^1}(S^4 \wedge S^{n\C}) \arrow[d, "\cdot\:(1 - \theta)^n"']
                & \KO_{S^1}(S^4 \wedge S^{n\C}) \arrow[d, "i^*"] \arrow[l, "c_\C"']
                & \KO_G(\Sigma^{4D}(A_n)) \arrow[l] \arrow[d, "\iota^*"]  
                \\
            \K_{S^1}(S^4)
                & \KO_{S^1}(S^4) \arrow[l, "c_\C"]
                & \KO_G(S^{4D}). \arrow[l, "r^G_{S^1}"] 
        \end{tikzcd}
    \end{equation}
    The leftmost vertical map is given by the composition
    \[
    \K_{S^1}(S^4 \wedge S^{n\C}) \xrightarrow[\cong]{\text{Bott}} \K_{S^1}(S^4) \cong \Z[\theta, \theta^{-1}] \cdot b_{S^4} \xrightarrow{\cdot\:(1 - \theta)^n} \Z[\theta, \theta^{-1}] \cdot b_{S^4} \cong \K_{S^1}(S^4),
    \]
    where $b_{S^4} \in \K_{S^1}(S^4)$ is the Bott class represented by $[V_{\mathbb{H}}] - 2$, regarding the quaternionic Hopf bundle $V_{\mathbb{H}}$ as a complex vector bundle via its natural complex structure. 
    
    Next, we examine the composition along the bottom row of the diagram, $c_{\C} \circ r^G_{S^1}: \KO_G(S^{4D}) \to \K_{S^1}(S^4)$.
    \begin{claim}
        The homomorphism $c_{\C} \circ r^G_{S^1}$ acts on the generators as follows:
        \[
            c_{\C} \circ r^G_{S^1}(\lambda(D)) = 2b_{S^4}, \quad c_{\C} \circ r^G_{S^1}(c) = (\theta + \theta^{-1})b_{S^4}, \quad c_{\C} \circ r^G_{S^1}(A) = (\theta - \theta^{-1})^2.
        \]
    \end{claim}

    \begin{proof}[Proof of Claim]
        Since the Bott isomorphism is given by multiplication by the Bott class, Fact~\ref{fact:Bott_res} ensures that it commutes with the restriction map. Thus, we have the following commutative diagram:
        \[
        \begin{tikzcd}
            \KO_G(S^{4D}) \arrow[r, "r^G_{S^1}"] \arrow[d, "\cong"', "\text{Bott}"]
                & \KO_{S^1}(S^{4})
                \\
            \KO_G(S^4) \arrow[ur, "r^G_{S^1}"']
        \end{tikzcd}
        \]
        Recall from Fact~\ref{fact:trivial_action} that we have:
        \[
        \KO_{S^1}(S^4) \cong \bigl( \KO(S^4) \otimes \Z \cdot 1_{\R} \bigr) \oplus \bigoplus_{k \ge 1} \bigl( \K(S^4) \otimes \Z \cdot V_k \bigr).
        \]
        Here, $1_{\R}$ and $V_k$ denote the trivial real representation and the standard $2$-dimensional irreducible real representation of $S^1$ of weight $k$, respectively. Using the identification via the Bott isomorphism, we can compute the images of $\lambda(D)$ and $c$ under the restriction map $r^G_{S^1}$ as follows:
        \begin{align*}
            r^G_{S^1}(\lambda(D)) &= r^G_{S^1}\bigl( ([V_{\mathbb{H}}] - 4_{\R}) \otimes 1_{\R} \bigr) = ([V_{\mathbb{H}}] - 4_{\R}) \otimes 1_{\R} \in \KO(S^4) \otimes \Z \cdot 1_{\R}, \\
            r^G_{S^1}(c) &= r^G_{S^1}\bigl( ([V_{\mathbb{H}}] - H) \otimes H \bigr) = ([V_{\mathbb{H}}] - 2_{\C}) \otimes V_1 \in \K(S^4) \otimes \Z \cdot V_1.
        \end{align*}
        
        Since $\mathbb{H}$ is complex self-conjugate, $V_{\mathbb{H}}$ satisfies $\overline{V_{\mathbb{H}}} \cong V_{\mathbb{H}}$. Recall that for any complex vector bundle $E$ and complex representation $W$, their complexifications satisfy $c_{\C}(E) = E \oplus \overline{E}$ and $c_{\C}(W) = W \oplus \overline{W}$. Thus, we have:
        \begin{align*}
            c_{\C} \circ r^G_{S^1}(\lambda(D)) &= c_{\C}\bigl(([V_{\mathbb{H}}] - 4_{\R}) \otimes 1_{\R}\bigr) \\
            &= \bigl(([V_{\mathbb{H}}] - 2_{\C}) \oplus \overline{([V_{\mathbb{H}}] - 2_{\C})}\bigr) \otimes 1_{\C} \\
            &= 2b_{S^4},
        \end{align*}
        where $1_{\C}$ denotes the trivial complex representation. Similarly, we have:
        \begin{align*}
            c_{\C} \circ r^G_{S^1}(c) &= c_{\C}\bigl(([V_{\mathbb{H}}] - 2_{\C}) \otimes_{\C} V_1\bigr) \\
            &= \bigl(([V_{\mathbb{H}}] - 2_{\C}) \otimes_{\C} \theta\bigr) \oplus \overline{\bigl(([V_{\mathbb{H}}] - 2_{\C}) \otimes_{\C} \theta\bigr)} \\
            &= \bigl(([V_{\mathbb{H}}] - 2_{\C}) \otimes_{\C} \theta\bigr) \oplus \bigl(\overline{([V_{\mathbb{H}}] - 2_{\C})} \otimes_{\C} \overline{\theta}\bigr) \\
            &= ([V_{\mathbb{H}}] - 2_{\C}) \otimes_{\C} (\theta \oplus \theta^{-1}) \\
            &= (\theta + \theta^{-1})b_{S^4}.
        \end{align*}
        
        Finally, recall that $A = K - (1+D)$. The $S^1$-representation $r^G_{S^1}(K)$ is the realification of $\theta^2$, and $r^G_{S^1}(1+D)$ is the trivial $2$-dimensional real representation. Therefore, $r^G_{S^1}(A)$ is the realification of the complex representation $\theta^2 - 1_{\C}$, which yields:
        \begin{align*}
            c_{\C} \circ r^G_{S^1}(A) &= c_{\C}(\theta^2 - 1_{\C}) \\
            &= (\theta^2 - 1_{\C}) \oplus \overline{(\theta^2 - 1_{\C})} \\
            &= (\theta^2 - 1_{\C}) + (\theta^{-2} - 1_{\C}) \\
            &= (\theta - \theta^{-1})^2.
        \end{align*}
        This completes the proof.
    \end{proof}
    
    Since $r^G_{S^1}((1 - D)\lambda(D)) = 0$, we have:
    \begin{equation}\label{image_x_1}
        c_{\C} \circ r^G_{S^1}(x) = \biggl( a\bigl(2 - (\theta + \theta^{-1})\bigr) + \sum_{k \ge 1} 2p_k(\theta - \theta^{-1})^{2k} + \sum_{l \ge 1} q_l(\theta + \theta^{-1})(\theta - \theta^{-1})^{2l} \biggr) \cdot b_{S^4}.
    \end{equation}
    Note that $\theta - \theta^{-1} = -\theta^{-1}(1 - \theta)(1 + \theta)$ and $2 - (\theta + \theta^{-1}) = -\theta^{-1}(1 - \theta)^2$. Omitting $b_{S^4}$ for simplicity, we can express this using some $Y(\theta) \in \Z[\theta, \theta^{-1}]$ as follows:
    \begin{equation}\label{image_x_2}
        c_{\C} \circ r^G_{S^1}(x) = (1 - \theta)^2 \bigl( -a\theta^{-1} + \theta^{-2}(1 + \theta)^2\bigl(2p_1 + (\theta + \theta^{-1})q_1\bigr) \bigr) + (1 - \theta)^4 Y(\theta).
    \end{equation}
    Let $X(\theta) \in \Z[\theta, \theta^{-1}]$ be the coefficient of $(1 - \theta)^2$, i.e., 
    \[
        X(\theta) = -a\theta^{-1} + \theta^{-2}(1 + \theta)^2\bigl(2p_1 + (\theta + \theta^{-1})q_1\bigr).
    \]

    If $n \ge 3$, then by diagram~\eqref{eq:KO_diagram4}, $c_{\C} \circ r^G_{S^1}(x) \in \operatorname{Im}\bigl(\cdot (1 - \theta)^3\bigr)$. Thus, $1 - \theta$ must divide $X(\theta)$, which implies $X(1) = 0$. Substituting $\theta = 1$ yields:
    \[
        -a + 4(2p_1 + 2q_1) = 0.
    \]
    However, $a$ is odd by \eqref{eq:a_odd}, which is a contradiction. Therefore, if $n \ge 3$, then $\ko_4(A_n) \neq 0$.

    Next, we consider the case $n = 2$. By Bott periodicity $\KO_{S^1}(S^4 \wedge S^{2\C}) \cong RO(S^1)\{b_{S^4 \wedge S^{2\C}}\}$, we have the following commutative diagram:
    \[
    \begin{tikzcd}[row sep=large, column sep=huge]
        \K_{S^1}(S^4 \wedge S^{2\C}) \arrow[d, "\cong"', "\text{Bott}"]
        & \KO_{S^1}(S^4\wedge S^{2\C}) \arrow[l, "c_{\C}"'] \arrow[d, "\cong"', "\text{Bott}"]
        \\
        R(S^1)\{b_{S^4 \wedge S^{2\C}}\} \arrow[d, "\cdot (1 - \theta)^2"']
        & RO(S^1)\{b_{S^4 \wedge S^{2\C}}\} \arrow[l, "c_{\C}"] \arrow[d, "i^*"]
        \\
        R(S^1)\{b_{S^4}\}
        & \KO_{S^1}(S^4) \arrow[l, "c_{\C}"]
    \end{tikzcd}
    \]
    Since the complexification map $c_{\C} \colon RO(S^1) \cong \Z[V] \to \Z[\theta, \theta^{-1}] \cong R(S^1)$ is given by $V \mapsto \theta + \theta^{-1}$, any element in the image of $c_{\C}$ is a symmetric Laurent polynomial in $\theta$ and $\theta^{-1}$. By the commutativity of the diagram, any element in the image of $c_{\C} \circ i^*$ can be expressed as $(1 - \theta)^2 \cdot P(\theta)$, where $P(\theta)$ is symmetric with respect to $\theta$ and $\theta^{-1}$. By equation~\eqref{image_x_2}, we have $c_{\C} \circ r^G_{S^1}(x) = (1 - \theta)^2 \big(X(\theta) + (1 - \theta)^2 Y(\theta) \big)$. Thus, $X(\theta) + (1 - \theta)^2Y(\theta)$ must be symmetric, which implies that $\big(X(\theta) + (1 - \theta)^2Y(\theta)\big) - \big(X(\theta^{-1}) + (1 - \theta^{-1})^2Y(\theta^{-1})\big) = 0$. However, we have
    \begin{align*}
        &\big(X(\theta) + (1 - \theta)^2Y(\theta)\big) - \big(X(\theta^{-1}) + (1 - \theta^{-1})^2Y(\theta^{-1})\big) \\
        &\quad = X(\theta) - X(\theta^{-1}) + \mathcal{O}\big((1 - \theta)^2\big) \\
        &\quad = (1 - \theta) Z(\theta) + \mathcal{O}\big((1 - \theta)^2\big),
    \end{align*}
    where $Z(\theta)$ is the coefficient of $(1 - \theta)$, given by
    \[
        Z(\theta) = -\theta^{-1}(1 + \theta)\big(a - (\theta + \theta^{-1} + 2)(2p_1 + (\theta + \theta^{-1})q_1)\big).
    \]
    Substituting $\theta = 1$ into $Z(\theta)$ yields
    \[
        Z(1) = -2\big(a - 8(p_1 + q_1)\big).
    \]
    Since $a$ is odd, we have $Z(1) \neq 0$. This is a contradiction; therefore, $\ko_4(A_2) \neq 0$. 
    
    In summary, if $n \ge 2$, then $\ko_4(A_n) \neq 0$, and by Claim~\ref{cl:above_bound}, $\ko_4(A_n) = 1$. This completes the proof of Proposition~\ref{prop:ko_A_n}.
\end{proof}

To generalize this calculation, we first prepare the following lemma regarding the induced maps on $S^1$-equivariant $KO$-theory.

\begin{lemma}\label{lem:two_inclusion}
    Let $X = S^{4m\C} \cup_{S^{l\C}} S^{n\C}$ be an $S^1$-space for some integers $l, m, n \in \Z_{\ge 0}$ with $l \le 4m < n$. Let $i_k \colon S^{4m\C} \hookrightarrow X$ $(k = 0, 1)$ be two inclusions defined such that $i_0$ is the standard inclusion into the first factor $S^{4m\C} \subset X$, and $i_1$ is the composition of the standard inclusions $S^{4m\C} \hookrightarrow S^{n\C} \hookrightarrow X$. Then $i_0^* = i_1^* \colon \KO_{S^1}(X) \to \KO_{S^1}(S^{4m\C})$.
\end{lemma}

\begin{proof}
    By Bott periodicity, $\KO_{S^1}(S^{4m\C}) \cong \Z[V] \cdot b_{4m\C}$. Let $i \colon S^0 \hookrightarrow S^{4m\C}$ be the standard inclusion. Since $l \ge 0$, the subspace $S^0$ is contained in the gluing subspace $S^{l\C}$. Thus, $i_0 \circ i = i_1 \circ i \colon S^0 \hookrightarrow X$. This implies that $(i_0 \circ i)^* = (i_1 \circ i)^* \colon \KO_{S^1}(X) \to \KO_{S^1}(S^0)$. 
    
    Since the induced map $i^* \colon \KO_{S^1}(S^{4m\C}) \to \KO_{S^1}(S^0)$, which corresponds to $\Z[V] \cdot b_{4m\C} \to \Z[V]$, is given by multiplication by the Euler class $(2 - V)^{4m}$, the map $i^*$ is injective. From $i^* \circ i_0^* = i^* \circ i_1^*$, we conclude that $i_0^* = i_1^*$.
\end{proof}

From these calculations, we obtain the following result.

\begin{theorem}
    Let $R$ be a symmetric graded root with $\beta(R) = 0$. Then the values of the $\ko$-invariant $\ko_i(\Hty(R))$ of the spectrum $\Hty(R)$ defined in Section 3 are given as follows:
    \[
    \begin{array}{c|cccccccc}
        i \pmod 8 & 0 & 1 & 2 & 3 & 4 & 5 & 6 & 7 \\ \hline
        \delta(R) = 0 & 0 & 0 & 0 & 0 & 0 & 0 & 0 & 0 \\
        \delta(R) = 1 & 1 & 1 & 1 & 0 & 0 & 0 & 0 & 0 \\
        \delta(R) \ge 2 & 1 & 1 & 1 & 0 & 1 & 0 & 0 & 0
    \end{array}
    \]
    Note that $\ko_i(\Hty(R))$ is equal to $\ko_i(A_{\delta(R)})$.
\end{theorem}

\begin{proof}
    If $\delta(R) = 0$, then $\Hty(R)$ is locally equivalent to $S^0$, and hence $\ko_i(\Hty(R)) = 0$ for all $i \in \Z/8$. Thus, we may assume $\delta(R) \ge 1$.
    
    By Lemma~\ref{lem:local_graded_map}, there exists a local map of graded roots $R \to X_{0, \delta(R)}$. Applying Theorem~\ref{thm:local_map_equivalence}, we obtain a local map of spectra $\mathcal{F}: \Hty(R) \to \Hty(X_{0, \delta(R)}) \simeq (A_{\delta(R)}, 0, 0)$. By Proposition~\ref{prop:ko_local}, this yields the upper bound $\ko_i(\Hty(R)) \le \ko_i(A_{\delta(R)})$ for any $i \in \Z/8$.

    Next, we establish the lower bound. We prove the case $i = 0$; the other cases follow by a similar argument. By Construction~\ref{con:lattice_spectrum}, we can express $\Hty(R) = (X, 0, 2h)$ for some integer $h \in \Z_{\ge 0}$. Since $\beta(R) = 0$, there exists an inclusion $\iota: S^{2h\mathbb{H}} \hookrightarrow X$, which corresponds to the stem of $R$. Thus, the standard inclusion $S^0 \hookrightarrow X$ factors through $S^{2h\mathbb{H}}$. Recall that the $\ko$-invariant of the spectrum $\Hty(R)$ is defined as $\ko_0(X) - 2h$. Although the induced map associated with $S^0 \hookrightarrow S^{2h\mathbb{H}}$ increases the $\ko$-invariant by $2h$, this increment cancels out the $-2h$ term in the definition of the spectrum. Consequently, the value of $\ko_0(\Hty(R))$ can be determined directly from the image of the induced map $\iota^*: \KO_{G}(X) \to \KO_G(S^{2h\mathbb{H}}) \cong RO(G) \cdot b_{2h\mathbb{H}}$.
    
    By Construction~\ref{con:lattice_spectrum}, there exists an $S^1$-equivariant inclusion $f: S^{(4h + \delta(R))\C} \hookrightarrow X$. Combining this with the restriction maps, we obtain the following diagram:
    \[
    \begin{tikzcd}
        \KO_G(X) \arrow[r, "r^G_{S^1}"] \arrow[d, "\iota^*"']
        &\KO_{S^1}(X) \arrow[r, "f^*"] \arrow[d, "\iota^*"']
        &\KO_{S^1}(S^{(4h + \delta(R))\C}) \arrow[d, "i^*"]
        \\
        \KO_G(S^{2h\mathbb{H}}) \arrow[r, "r^G_{S^1}"]
        &\KO_{S^1}(S^{4h\C}) \arrow[r, equal]
        &\KO_{S^1}(S^{4h\C}),
    \end{tikzcd}
    \]
    where $i$ denotes the standard inclusion. The commutativity of the right square follows from Lemma~\ref{lem:two_inclusion}. Using Bott periodicity, we can rewrite the diagram as follows:
    \[
    \begin{tikzcd}
        \KO_G(X) \arrow[r] \arrow[d, "\iota^*"']
        &\KO_{S^1}(S^{\delta(R)\C}) \cdot b_{4h\C} \arrow[d, "i^*"]
        \\
        RO(G) \cdot b_{2h\mathbb{H}} \arrow[r, "r^G_{S^1}"]
        &RO(S^1) \cdot b_{4h\C}.
    \end{tikzcd}
    \]
    This diagram has the same structure as the left square of diagram~\eqref{eq:KO_diagram2}, which was used to determine the lower bound of $\ko_i(A_n)$. Therefore, we can obtain the lower bound of $\ko_i(\Hty(R))$ in exactly the same way as we did for $\ko_i(A_n)$.
\end{proof}

In general, if $\beta(R) \neq 0$, let $\tilde{R}$ be the grading shift of $R$ such that $\beta(\tilde{R}) = 0$. Then we have $\Hty(R) \simeq \Sigma^{\frac{\beta(R)}{2}\mathbb{H}}\Hty(\tilde{R})$. By Definition~\ref{def:kappao_spectrum}, we have:
\[
\ko_i(\Hty(R)) = 
\begin{cases}
    \ko_i(\Hty(\tilde{R})) + \frac{\beta(R)}{2} & \text{if } \lfloor -\frac{\beta(R)}{2} \rfloor \text{ is even}, \\
    \ko_i(\Sigma^{\mathbb{H}}\Hty(\tilde{R})) + \frac{\beta(R)}{2} - 1 & \text{if } \lfloor -\frac{\beta(R)}{2} \rfloor \text{ is odd}
\end{cases}
\]
where $\lfloor \alpha \rfloor$ denotes the greatest integer less than or equal to $\alpha$.

Applying Proposition~\ref{prop:suspension_ko}, we have:
\begin{align*}
    \ko_i(\Sigma^{\mathbb{H}} \Hty(\tilde{R})) 
    &= \ko(\Sigma^{\mathbb{H} \oplus (8 + i)\tilde{\R}}\Hty(\tilde{R})) \\
    &= \ko(\Sigma^{(4 + i)\tilde{\R}}\Hty(\tilde{R})) + 3 - \beta^4_i \\
    &= \ko_{4 + i}(\Hty(\tilde{R})) + 3 - \beta^4_i.
\end{align*}
This completes the proof of Theorem~\ref{thm:ko_calculation}.

We now discuss the constraints on the intersection forms of spin $4$-manifolds bounded by AR homology spheres that arise from the $\ko$-invariants. We recall the following inequalities established by Lin.

\begin{theorem}[{\cite[Corollary 1.12]{Lin15}}]\label{thm:cor_lin_ko}
    Let $(X, \mathfrak{t})$ be an oriented smooth spin $4$-manifold whose boundary is a homology sphere $(Y, \s)$. Suppose that the intersection form of $X$ is $p(-E_8) \oplus q \left( \begin{smallmatrix} 0 & 1 \\ 1 & 0 \end{smallmatrix} \right)$ with $p \ge 0$ and $q > 0$. Then we have the following inequalities:
    \begin{itemize}
        \item If $p = 4l$, then $2l < \ko_{4 + q}(Y) + \beta_{4 + q}^{q}$.
        \item If $p = 4l + 1$, then $2l + \frac{5}{2} < \ko_{q}(Y) + \beta_{q}^{4 + q}$.
        \item If $p = 4l + 2$, then $2l + 3 < \ko_{q}(Y) + \beta_{q}^{4 + q}$.
        \item If $p = 4l + 3$, then $2l + \frac{3}{2} < \ko_{4 + q}(Y) + \beta_{4 + q}^{q}$.
    \end{itemize}
\end{theorem}

\begin{proof}[Proof of Theorem~\ref{thm:main_inequality_ko}]
    By applying Theorem~\ref{thm:cor_lin_ko} to the $\ko$-invariants for AR homology spheres computed in Theorem~\ref{thm:ko_calculation}, we obtain the desired inequalities.
\end{proof}

\subsection{Lower and upper bounds on the $\ko$-invariants for connected sums of AR homology spheres}

In this subsection, we bound the $\ko$-invariants of connected sums of AR homology spheres from above and below.

First, we provide a lower bound for the connected sums of AR homology spheres.

\begin{proposition}
    Let $(Y_1, \s_1), \dots, (Y_n, \s_n)$ be spin AR homology spheres, and let $(Y, \s) = \mathop{\#}_{i=1}^n (Y_i, \s_i)$. Then, we have:
    \[
    \ko_i(Y, \s) \ge \ko_i(S^{-\frac{1}{2}\bar{\mu}(Y, \s)\mathbb{H}})\left( = \ko_i\big(S^{-\frac{1}{2}\left(\sum_{k = 1}^n\bar{\mu}(Y_k, \s_k)\right)\mathbb{H}}\big)\right)
    \]
    for all $i \in \Z / 8$.
\end{proposition}

\begin{proof}
    By Lemma~\ref{lem:can_inclusion}, there exist inclusion maps $i: S^{-\frac{1}{2}\bar{\mu}(Y_k, \s_k)\mathbb{H}} \to \SWF(Y_k, \s_k)$ for all $1 \le k \le n$. Combining the fact that these inclusions are local maps with Theorem~\ref{thm:conn_sum_SWF}, we obtain a local map $f: S^{-\frac{1}{2}\left(\sum_{k = 1}^n\bar{\mu}(Y_k, \s_k)\right)\mathbb{H}} \to \SWF(Y, \s)$. Applying Proposition~\ref{prop:ko_local} to this map completes the proof.
\end{proof}

Similarly to the argument in Section 6.2, Definition~\ref{def:kappao_spectrum} and Proposition~\ref{prop:suspension_ko} yield:
\begin{equation}\label{eq:ko_sphere}
\ko_i(S^{-\frac{1}{2}\bar{\mu}(Y, \s)\mathbb{H}}) = 
\begin{cases}
    -\frac{1}{2}\bar{\mu}(Y, \s) & \text{if } \lfloor \frac{1}{2}\bar{\mu}(Y, \s) \rfloor \text{ is even}, \\
    2 - \beta^4_i - \frac{1}{2}\bar{\mu}(Y, \s) & \text{if } \lfloor \frac{1}{2}\bar{\mu}(Y, \s) \rfloor \text{ is odd}
\end{cases}
\end{equation}
for all $i \in \Z / 8$. Since $2 - \beta^4_i \ge -1$ for any $i \in \Z / 8$, we immediately obtain the following corollary:

\begin{corollary}
    Let $(Y_1, \s_1), \dots, (Y_n, \s_n)$ be spin AR homology spheres, and let $(Y, \s) = \mathop{\#}_{i=1}^n (Y_i, \s_i)$. Then, we have:
    \[
    \ko_i(Y, \s) \ge -\frac{1}{2}\bar{\mu}(Y, \s) - 1
    \]
    for all $i \in \Z / 8$.
\end{corollary}

\begin{corollary}\label{cor:sequence}
    Let $(Y_n, \s_n)$ be a sequence of connected sums of spin AR homology spheres such that $\bar{\mu}(Y_n, \s_n) \to -\infty$ as $n \to \infty$. Then, $\ko_i(Y_n, \s_n) \to \infty$ as $n \to \infty$ for all $i \in \Z / 8$. 
    Examples of such sequences include $\mathop{\#}^n \Sigma(2, 3, 5)$ and $\mathop{\#}^n \big(\Sigma(2, 3, 5) \mathop{\#} \Sigma(2, 3, 11) \big)$.
\end{corollary}

Next, we provide an upper bound for the connected sums of AR homology spheres. We recall an important property of the $\ko$-invariant under admissible morphisms.

\begin{proposition}[{\cite[Proposition 6.3]{Lin15}}]\label{prop:ko_admissible}
    Let $\mathcal{X}_0 = (X_0, m_0, n_0), \mathcal{X}_1 = (X_1, m_1, n_1) \in \mathfrak{C}_G$ be two spectrum classes where $X_0$ and $X_1$ are spaces of type SWF at level $k_0$ and $k_1$, respectively. If there exists an admissible morphism $f: \mathcal{X}_0 \to \mathcal{X}_1$, then we have:
    \[
    \ko(\mathcal{X}_0) \le \ko(\mathcal{X}_1) + \beta^{k_1 - k_0}_{k_1}.
    \]
\end{proposition}

Applying this proposition, we obtain the following upper bound for the connected sum of AR homology spheres when the $\bar{\mu}$-invariant vanishes.

\begin{lemma}\label{lem:ko_above}
    Let $(Y_1, \s_1), \dots, (Y_n, \s_n)$ be spin AR homology spheres, and let $(Y, \s) = \mathop{\#}_{i=1}^n (Y_i, \s_i)$ with $\bar{\mu}(Y, \s) = \sum_{i = 1}^{n}\bar{\mu}(Y_i, \s_i) = 0$. Then, we have:
    \[
    \ko_i(Y, \s) \le \beta^{m}_{i + m},
    \]
    where $m$ is the number of components $Y_i$ satisfying $-\bar{\mu}(Y_i, \s_i) < \delta(Y_i, \s_i)$.
\end{lemma}

\begin{proof}
    By Theorem~\ref{thm:admissible_map} and Theorem~\ref{thm:conn_sum_SWF}, there exists an admissible morphism
    \[
        f: \SWF(Y, \s) \to S^{m\tilde{\R}}.
    \]
    Applying Proposition~\ref{prop:ko_admissible} and taking the suspension $\Sigma^{i\tilde{\R}}$, we obtain:
    \[
    \ko_i(Y, \s) \le \beta^{m}_{i + m} + \ko(S^{(i + m) \tilde{\R}}).
    \]
    Since $\ko(S^{(i + m) \tilde{\R}}) = 0$, this completes the proof. 
\end{proof}

In general, similarly to the argument in Section 6.2, we have $\ko_i(Y, \s) = \ko_i(\Sigma^{-\frac{1}{2}\bar{\mu}(Y, \s)\mathbb{H}} S)$, where $S = \Sigma^{\frac{1}{2}\bar{\mu}(Y, \s)\mathbb{H}}\SWF(Y, \s)$, and we obtain:
\begin{equation}\label{eq:ko_general}
\ko_i(Y, \s) = 
\begin{cases}
    \ko_i(S) - \frac{\bar{\mu}(Y, \s)}{2} & \text{if } \lfloor \frac{\bar{\mu}(Y, \s)}{2} \rfloor \text{ is even}, \\
    \ko_{4 + i}(S) + 2 - \beta^4_i - \frac{\bar{\mu}(Y, \s)}{2} & \text{if } \lfloor \frac{\bar{\mu}(Y, \s)}{2} \rfloor \text{ is odd}.
\end{cases}
\end{equation}

Combining \eqref{eq:ko_general} with Lemma~\ref{lem:ko_above}, we obtain the following upper bound:

\begin{theorem}\label{thm:ko_general_above}
    Let $(Y_1, \s_1), \dots, (Y_n, \s_n)$ be spin AR homology spheres, and let $(Y, \s) = \mathop{\#}_{i=1}^n (Y_i, \s_i)$. Then we have:
    \[
    \ko_i(Y, \s) \le
    \begin{cases}
        \beta^{m}_{i + m} - \frac{\bar{\mu}(Y, \s)}{2} & \text{if } \lfloor \frac{\bar{\mu}(Y, \s)}{2} \rfloor \text{ is even}, \\
        \beta^{4 + m}_{4 + i + m} - 2 - \frac{\bar{\mu}(Y, \s)}{2} & \text{if } \lfloor \frac{\bar{\mu}(Y, \s)}{2} \rfloor \text{ is odd}.
    \end{cases}
    \]
\end{theorem}

\begin{proof}
    As in the proof of Lemma~\ref{lem:ko_above}, we have $\ko_i(S) \le \beta^{m}_{i + m}$ and $\ko_{4 + i}(S) \le \beta^{m}_{4 + i + m}$. Since $\beta^{m}_{4 + i + m} - \beta^4_i = \beta^{4 + m}_{4 + i + m} - (\beta^4_{4 + i} + \beta^4_i) = \beta^{4 + m}_{4 + i + m} - \beta^8_{4 + i} = \beta^{4 + m}_{4 + i + m} - 4$, combining this with \eqref{eq:ko_general} yields the desired inequality. 
\end{proof}

Applying Theorem~\ref{thm:cor_lin_ko}, we obtain the constraint on intersection forms for connected sums of AR homology spheres.

\begin{corollary}\label{cor:conn_sum_ineq_ko}
    Let $Y_1, \dots, Y_n$ be AR homology spheres. Suppose that each $Y_i$ is an integral homology sphere. Let $Y = \mathop{\#}_{i=1}^n Y_i$. Then, for any smooth, compact, indefinite, spin $4$-manifold $X$ with boundary $\partial X = Y$ and intersection form $p(-E_8) \oplus q \left(\begin{smallmatrix} 0 & 1 \\ 1 & 0 \end{smallmatrix}\right)$ $(p \ge 0, q > 0)$, the following inequality holds:
    \[
    \renewcommand{\arraystretch}{1.2}
    q \ge 
    \left\{
    \begin{array}{l@{\quad}ll}
    p + \sum_{i=1}^n \bar{\mu}(Y_i) + 1 - m & \text{if } p + \sum_{i=1}^n \bar{\mu}(Y_i) \equiv 0, 2 &\pmod{8} \\
    p + \sum_{i=1}^n \bar{\mu}(Y_i) + 2 - m & \text{if } p + \sum_{i=1}^n \bar{\mu}(Y_i) \equiv 4 &\pmod{8} \\
    p + \sum_{i=1}^n \bar{\mu}(Y_i) + 3 - m & \text{if } p + \sum_{i=1}^n \bar{\mu}(Y_i) \equiv 6 &\pmod{8}
    \end{array}
    \right.
    \]
    where $m$ is the number of components $Y_i$ satisfying $-\bar{\mu}(Y_i) < \delta(Y_i)$.
\end{corollary}

\begin{proof}
    By Theorem~\ref{thm:ko_general_above} and \eqref{eq:ko_sphere}, we have:
    \[
        \ko_i(Y, \s) \le \ko_{i + m}(S^{-\frac{1}{2}\bar{\mu}(Y, \s)\mathbb{H}}) + \beta^{m}_{i + m}.
    \]
    Combining this with Theorem~\ref{thm:cor_lin_ko}, we obtain:
    \begin{itemize}
        \item If $p = 4l$, then $2l < \ko_{4 + q + m}(S^{-\frac{1}{2}\bar{\mu}(Y, \s)\mathbb{H}}) + \beta^{m}_{4 + q + m} + \beta_{4 + q}^{q}$.
        \item If $p = 4l + 1$, then $2l + \frac{5}{2} < \ko_{q + m}(S^{-\frac{1}{2}\bar{\mu}(Y, \s)\mathbb{H}}) + \beta^{m}_{q + m} + \beta_{q}^{4 + q}$.
        \item If $p = 4l + 2$, then $2l + 3 < \ko_{q + m}(S^{-\frac{1}{2}\bar{\mu}(Y, \s)\mathbb{H}}) + \beta^{m}_{q + m} + \beta_{q}^{4 + q}$.
        \item If $p = 4l + 3$, then $2l + \frac{3}{2} < \ko_{4 + q + m}(S^{-\frac{1}{2}\bar{\mu}(Y, \s)\mathbb{H}}) + \beta^{m}_{4 + q + m} + \beta_{4 + q}^{q}$.
    \end{itemize}
    Since $\beta^{m}_{4 + q + m} + \beta_{4 + q}^{q} = \beta^{q + m}_{4 + q + m}$ and $\beta^{m}_{q + m} + \beta_{q}^{4 + q} = \beta^{4 + q + m}_{q + m}$, we have:
    \begin{itemize}
        \item If $p = 4l$, then $2l < \ko_{4 + q + m}(S^{-\frac{1}{2}\bar{\mu}(Y, \s)\mathbb{H}}) + \beta^{q + m}_{4 + q + m}$.
        \item If $p = 4l + 1$, then $2l + \frac{5}{2} < \ko_{q + m}(S^{-\frac{1}{2}\bar{\mu}(Y, \s)\mathbb{H}}) + \beta^{4 + q + m}_{q + m}$.
        \item If $p = 4l + 2$, then $2l + 3 < \ko_{q + m}(S^{-\frac{1}{2}\bar{\mu}(Y, \s)\mathbb{H}}) + \beta^{4 + q + m}_{q + m}$.
        \item If $p = 4l + 3$, then $2l + \frac{3}{2} < \ko_{4 + q + m}(S^{-\frac{1}{2}\bar{\mu}(Y, \s)\mathbb{H}}) + \beta^{q + m}_{4 + q + m}$.
    \end{itemize}
    These inequalities coincide with the constraints for the case where the SWF spectrum of the boundary $3$-manifold is a sphere spectrum, with $q$ replaced by $q + m$. This corresponds to condition (i) in Theorem~\ref{thm:main_inequality_ko}. Thus, we obtain the desired inequality.
\end{proof}

\begin{remark}
    The constraint obtained from Corollary~\ref{cor:conn_sum_ineq_ko} coincides with the one obtained by applying \cite[Theorem 1.8]{Sch03} to the necessary condition for the existence of an admissible map $S^{\frac{1}{2}(p + \bar{\mu})\mathbb{H}} \to S^{(q + m)\tilde{\R}}$.
\end{remark}

\begin{remark}
    Under the condition that $p + \sum_{i = 1}^n\bar{\mu}(Y_i, \s_i) \ge 4$, the constraint obtained from Theorem~\ref{thm:conn_sum} is stronger than that of Corollary~\ref{cor:conn_sum_ineq_ko}.
\end{remark}

\section{Applications}

\subsection{Non-sliceness of the $(4n, 1)$-cables of the figure-eight knot}

In this subsection, using ordinary Seiberg--Witten theory, we provide an alternative proof of the non-sliceness of the $(4n, 1)$-cables $E_{4n, 1}$ of the figure-eight knot $E$, which was originally proved by Kang, Park, and Taniguchi \cite{KPT25} via real Seiberg--Witten theory. This means that we can show the non-sliceness of $E_{4n, 1}$ without relying on the natural involution on its $2^k$-fold branched cover $\Sigma_{2^k}(E_{4n, 1})$, where $4n = 2^k m$ with $k \ge 2$ and $m \ge 1$ odd. In particular, this implies that $\Sigma_{2^k}(E_{4n, 1})$ does not bound any smooth spin rational homology $4$-ball. Up to the construction of the $4$-manifold via the $2^k$-fold branched cover, we follow the setup of \cite{KPT25}.

Let $n, k$, and $m$ be positive integers satisfying $4n = 2^k m$ with $m$ odd. Let $X$ be a twice-punctured $2\C P^2$, given by
\[
X = 2\C P^2 \setminus (\operatorname{int} D^4 \sqcup \operatorname{int} D^4) \cong 2\C P^2 \# (S^3 \times [-1, 1]).
\]
A key topological input for our proof is the following construction of a smooth concordance in $X$ to a torus knot, originally observed in \cite{ACMPS26}.

\begin{proposition}[{\cite{ACMPS26}}]
    There is a smooth concordance $S_{4n}$ in $X$ from $E_{4n, 1}$ to the torus knot $T(4n, 1 - 40n)$ representing the homology class $(4n, 12n)$ in $H_2(X, \partial X; \Z)$.
\end{proposition}

Taking the $2^k$-fold branched cover of $X$ along $S_{4n}$, we obtain a smooth $4$-manifold $\Sigma_{2^k}(S_{4n})$ with boundary
\[
\partial \Sigma_{2^k}(S_{4n}) = -\Sigma_{2^k}(E_{2^k m, 1}) \sqcup -\Sigma(2^k, 2^k m, 2^k \cdot 10m - 1),
\]
where $\Sigma_{2^k}(E_{2^k m, 1})$ denotes the $2^k$-fold branched cover of $S^3$ along $E_{2^k m, 1}$. In what follows, we denote the Brieskorn sphere $\Sigma(2^k, 2^k m, 2^k \cdot 10m - 1)$ by $Y_{k, m}$. By the same computation as in \cite[Lemma 2.7]{KPT25}, the topological invariants of $\Sigma_{2^k}(S_{4n})$ are given as follows:

\begin{lemma}\label{lem:top_inv}
    We have
    \[
    \sigma(\Sigma_{2^k}(S_{4n})) = 2, \quad b^+(\Sigma_{2^k}(S_{4n})) = 2^k + 1, \quad b_1(\Sigma_{2^k}(S_{4n})) = 0.
    \]
\end{lemma}

\begin{proof}
    Since $E_{4n, 1}$ is algebraically slice for any integer $n$ \cite{Kaw80, Cha07, KW18}, we have $\sigma^{(2^k)}(E_{4n, 1}) = \sum_{\omega^{2^k} = 1} \sigma_{E_{4n, 1}}(\omega) = 0$. Moreover, since the Levine--Tristram signature is additive under disjoint union, the same computation as in \cite[Lemma 2.7]{KPT25} yields the desired equalities for $\sigma(\Sigma_{2^k}(S_{4n}))$ and $b^+(\Sigma_{2^k}(S_{4n}))$. Also, by the proof of \cite[Lemma 2.8]{KPT25}, we have $H^1(\Sigma_{2^k}(S_{4n}); \Z_2) = 0$, which implies $b_1(\Sigma_{2^k}(S_{4n})) = 0$.
\end{proof}

Using the $4$-manifold $\Sigma_{2^k}(S_{4n})$ together with the relative Bauer--Furuta invariant and the capping off argument, we obtain Theorem~\ref{thm:non_trivial}.

\begin{proof}[Proof of Theorem~\ref{thm:non_trivial}]
    Suppose that $\Sigma_{2^k}(E_{2^k m, 1})$ bounds a smooth spin rational homology $4$-ball. Since $\Sigma_{2^k}(S_{4n})$ is spin by \cite[Lemma 2.8]{KPT25} and $\Sigma_{2^k}(E_{2^k m, 1})$ carries a unique spin structure, capping off $\Sigma_{2^k}(S_{4n})$ with this spin rational homology $4$-ball yields a smooth spin $4$-manifold $W$ with $b_1(W) = 0$ and $\partial W = -Y_{k, m}$ whose intersection form is isomorphic to that of $\Sigma_{2^k}(S_{4n})$. Since $Y_{k, m}$ is a $\Z_2$-homology sphere, the spin structure on $Y_{k, m}$ is also unique.
    
    Viewing $W$ as a spin cobordism from $Y_{k, m}$ to $\emptyset$ and considering the relative Bauer--Furuta invariant of $W$ (Theorem~\ref{thm:rel_BF}), we obtain an admissible map
    \[
    \Sigma^{-\frac{\sigma(W)}{16}\mathbb{H}} \SWF(Y_{k, m}) \to S^{b^+(W)\tilde{\R}}.
    \]
    Since $\sigma(W) = \sigma(\Sigma_{2^k}(S_{4n})) = 2$ and $b^+(W) = b^+(\Sigma_{2^k}(S_{4n})) = 2^k + 1$ by Lemma~\ref{lem:top_inv}, we have an admissible map
    \[
    \Sigma^{-\frac{1}{8}\mathbb{H}} \SWF(Y_{k, m}) \to S^{(2^k + 1)\tilde{\R}}.
    \]
    
    By \cite[Lemma 2.1]{KPT25}, the Neumann--Siebenmann invariant of $Y_{k, m} = \Sigma(2^k, 2^k m, 2^k \cdot 10m - 1)$ is given by
    \[
    \bar\mu(Y_{k, m}) = -2^k - \frac{1}{4}.
    \]
    By Lemma~\ref{lem:can_inclusion}, we then have a spectrum cap
    \[
    S^{-\frac{1}{2}\bar\mu(Y_{k, m})\mathbb{H}} = S^{(2^{k - 1} + \frac{1}{8})\mathbb{H}} \to \SWF(Y_{k, m}).
    \]
    Composing the relative Bauer--Furuta invariant of $W$ with the spectrum cap, we obtain an admissible map
    \[
    S^{2^{k - 1}\mathbb{H}} \to S^{(2^k + 1)\tilde{\R}}.
    \]
    By Proposition~\ref{prop:equivalence_of_morphism}, this map gives a Furuta--Mahowald class of level-$(2^{k - 1}, 2^k + 1)$. Since $k \ge 2$, we have $2^{k - 1} \ge 2$. Applying Theorem~\ref{thm:inequality}, we obtain
    \[
    2^k + 1 \ge 2 \cdot 2^{k - 1} + 2 = 2^k + 2,
    \]
    which is a contradiction. Since the $2^k$-fold branched cover of $D^4$ along any smooth slice disk is a smooth $\Z_2$-homology $4$-ball, it follows that $E_{4n, 1}$ is not smoothly slice.
\end{proof}

The following theorem represents the limit of what our method can detect for connected sums.

\begin{theorem}\label{thm:non_trivial_conn_sum}
    For any integer $k \ge 2$ and any odd integer $m \ge 1$, $2\Sigma_{2^k}(E_{2^k m, 1}) = \Sigma_{2^k}(E_{2^k m, 1}) \# \Sigma_{2^k}(E_{2^k m, 1})$ does not bound any smooth spin rational homology $4$-ball. Similarly, for any integer $k \ge 4$ and any odd integer $m \ge 1$, $3\Sigma_{2^k}(E_{2^k m, 1})$ does not bound any smooth spin rational homology $4$-ball.
\end{theorem}

\begin{proof}
    Suppose that $2\Sigma_{2^k}(E_{2^k m, 1})$ bounds a smooth spin rational homology $4$-ball. Taking the boundary connected sum of two copies of $\Sigma_{2^k}(S_{4n})$ along $\Sigma_{2^k}(E_{2^k m, 1})$, we obtain a smooth spin $4$-manifold $\Sigma_{2^k}(S_{4n}) \natural \Sigma_{2^k}(S_{4n})$ with
    \[
    \partial \bigl(\Sigma_{2^k}(S_{4n}) \natural \Sigma_{2^k}(S_{4n})\bigr) = -\bigl(\Sigma_{2^k}(E_{2^k m, 1}) \# \Sigma_{2^k}(E_{2^k m, 1})\bigr) \sqcup (-Y_{k, m}) \sqcup (-Y_{k, m}).
    \]
    Capping off this $4$-manifold with the assumed spin rational homology $4$-ball and attaching the natural cobordism from $(-Y_{k, m}) \sqcup (-Y_{k, m})$ to $(-Y_{k, m}) \# (-Y_{k, m})$, we obtain a smooth spin $4$-manifold $W'$ with $\partial W' = 2(-Y_{k, m})$, $\sigma(W') = 4$, and $b^+(W') = 2(2^k + 1)$.

    By performing spin surgeries along loops generating the free part of $H_1(W'; \Z)$, we may assume that $b_1(W') = 0$ without changing the intersection form of $W'$.

    By Lemma~\ref{lem:can_inclusion} and Theorem~\ref{thm:conn_sum_SWF}, we have an admissible map
    \[
    S^{-\frac{1}{2}\bar\mu(Y_{k, m})\mathbb{H}} \wedge S^{-\frac{1}{2}\bar\mu(Y_{k, m})\mathbb{H}} = S^{-\bar\mu(Y_{k, m})\mathbb{H}} \to \SWF(2Y_{k, m}).
    \]
    Composing the relative Bauer--Furuta invariant of $W'$ with this map, we obtain an admissible map
    \[
    S^{2^k\mathbb{H}} \to S^{2(2^k + 1)\tilde{\R}}.
    \]
    By Proposition~\ref{prop:equivalence_of_morphism}, this map gives a Furuta--Mahowald class of level-$(2^k, 2^{k + 1} + 2)$. Since $k \ge 2$, $2^k$ is divisible by $4$. Applying Theorem~\ref{thm:inequality}, we obtain
    \[
    2^{k + 1} + 2 \ge 2 \cdot 2^k + 3 = 2^{k + 1} + 3,
    \]
    which is a contradiction.

    Next, suppose that $k \ge 4$ and $3\Sigma_{2^k}(E_{2^k m, 1})$ bounds a smooth spin rational homology $4$-ball. Then the same argument as above yields a smooth spin $4$-manifold $W''$ with $b_1(W'') = 0$, $\partial W'' = 3(-Y_{k, m})$, $\sigma(W'') = 6$, and $b^+(W'') = 3(2^k + 1)$. Composing the relative Bauer--Furuta invariant of $W''$ with the corresponding spectrum cap for $3Y_{k, m}$, we obtain a Furuta--Mahowald class of level-$(3 \cdot 2^{k - 1}, 3 \cdot 2^k + 3)$. Since $k \ge 4$, $3 \cdot 2^{k - 1}$ is divisible by $8$. Applying Theorem~\ref{thm:inequality}, we obtain
    \[
    3 \cdot 2^k + 3 \ge 2 \cdot 3 \cdot 2^{k - 1} + 4 = 3 \cdot 2^k + 4,
    \]
    which is again a contradiction.
\end{proof}

The following proposition implies that $\Sigma_{2^k}(E_{2^k m, 1})$ defines an element in the $3$-dimensional integral homology cobordism group $\Theta^3_{\Z}$. Thus, Theorem~\ref{thm:non_trivial} and Theorem~\ref{thm:non_trivial_conn_sum} imply that the connected sums $c\Sigma_{2^k}(E_{2^k m, 1})$ are nontrivial in $\Theta^3_{\Z}$ for $c \in \{1, 2\}$ when $k \ge 2$, and for $c \in \{1, 2, 3\}$ when $k \ge 4$.

\begin{proposition}\label{prop:hom_sph}
    For any knot $K \subset S^3$, any integer $k \ge 1$, and any odd integer $m \ge 1$, the $2^k$-fold branched cover $\Sigma_{2^k}(K_{2^k m, 1})$ of the $(2^k m, 1)$-cable of $K$ is an integral homology $3$-sphere.
\end{proposition}

\begin{proof}
    Since $U_{2^k m, 1} = U$, where $U$ is the unknot, the satellite formula for the Alexander polynomial yields $\Delta_{K_{2^k m, 1}}(t) = \Delta_K(t^{2^k m})$. By the formula in \cite{Fox56}, the order of $H_1(\Sigma_{2^k}(K_{2^k m, 1}); \Z)$ is given by
    \[
    \bigl| H_1(\Sigma_{2^k}(K_{2^k m, 1}); \Z) \bigr| = \prod_{j=1}^{2^k - 1} \bigl| \Delta_{K_{2^k m, 1}}(\zeta^j) \bigr| = \prod_{j=1}^{2^k - 1} \bigl| \Delta_K(\zeta^{2^k m j}) \bigr| = \prod_{j=1}^{2^k - 1} \bigl| \Delta_K(1) \bigr| = 1,
    \]
    where $\zeta = e^{2\pi i / 2^k}$ is a primitive $2^k$-th root of unity. Thus $H_1(\Sigma_{2^k}(K_{2^k m, 1}); \Z) = 0$, which means that $\Sigma_{2^k}(K_{2^k m, 1})$ is an integral homology $3$-sphere.
\end{proof}

\subsection{Applications to genus lower bounds}\label{subsec:apps}

In this subsection, we consider the gauge-theoretic applications to genus lower bounds. To evaluate the minimal genus of surfaces bounded by knots in specific $4$-manifolds, we first recall a key lemma that describes the topological invariants of double branched covers.
\begin{lemma}[{\cite[Lemma 4.2]{KMT25}}]\label{lem:KMT_double_branch}
    Let $(X, S)$ be a connected oriented knot cobordism from $(S^3, K)$ to $(S^3, K')$. Suppose $H_1(X; \Z) = 0$, $[S]$ is divisible by $2$, and $\operatorname{PD}(w_2(X)) \equiv [S]/2 \pmod 2$. Then the double branched cover $\Sigma(S)$ along $S$ has a unique spin structure. Moreover, the following equalities hold:
    \begin{equation}
        \sigma(\Sigma(S)) = 2\sigma(X) - \frac{1}{2}[S]^2 - \sigma(K) + \sigma(K').
    \end{equation}
    \begin{equation}
        b^+(\Sigma(S)) = 2b^+(X) + g(S) - \frac{1}{4}[S]^2 - \frac{1}{2}\sigma(K) + \frac{1}{2}\sigma(K').
    \end{equation}
    \begin{equation}
        b_1(\Sigma(S)) = 0.
    \end{equation}
\end{lemma}

\begin{proof}[Proof of Corollary~\ref{cor:knot_inequality}]
    By Remark~\ref{rem:rational}, we can apply Theorem~\ref{thm:main_inequality} and Theorem~\ref{thm:main_inequality_ko} to the spin $4$-manifold $\Sigma(S)$ by replacing $p$ and $q$ with $-\frac{1}{8}\sigma(\Sigma(S))$ and $b^+(\Sigma(S))$, respectively. We note that the spin structure of $\Sigma(K)$ is also unique since it is known that $H^1(\Sigma(K); \Z/2) = 0$. Applying Lemma~\ref{lem:KMT_double_branch} to the cobordism from $(S^3, U)$ to $(S^3, K)$ yields the desired inequality.
\end{proof}

\begin{proof}[Proof of Example~\ref{ex:CP^2_genus}]
    We divide the proof into two parts: an explicit geometric construction for the upper bound, and a gauge-theoretic obstruction for the lower bound.
    
    \textbf{Upper Bound.}
    First, consider the case $X = \C P^2 \setminus \operatorname{int} D^4$. Let $C \subset \C P^2$ be a non-singular algebraic curve of degree $6$. Its genus is $g(C) = (6-1)(6-2)/2 = 10$, and it represents the homology class $6 \in H_2(\C P^2; \Z)$. By removing a small open disk $\operatorname{int} D^2$ from $C$, we obtain an oriented, compact, properly embedded surface $C' \subset \C P^2 \setminus \operatorname{int} D^4$ of genus $10$, whose boundary is the unknot $U \subset S^3 = \partial(\C P^2 \setminus \operatorname{int} D^4)$.
    
    On the other hand, in the collar neighborhood $\partial(\C P^2 \setminus \operatorname{int} D^4) \times [0, 1] \cong S^3 \times [0, 1]$, there exists a knot cobordism $C_K$ from $U$ to $K$ whose genus realizes the slice genus $g_4(K)$. By gluing $C'$ and $C_K$ along their common boundary $U$, we obtain a properly embedded surface $S \subset \C P^2 \setminus \operatorname{int} D^4$ bounded by $K$. The genus of this surface is $g(S) = g(C') + g(C_K) = g_4(K) + 10$, and its relative homology class remains $6$. Thus, the minimal genus is bounded above by $g_4(K) + 10$.
    
    For the case $X = (\C P^2 \# \C P^2) \setminus \operatorname{int} D^4$, taking the connected sum of two such surfaces $C'$ yields a surface of genus $20$ bounded by $U$. Gluing this to $C_K$ gives a surface bounded by $K$ with genus $g_4(K) + 20$ and relative homology class $(6, 6)$.
    
    \textbf{Lower Bound.}
    We proceed by contradiction. \vspace{0.5em}\\
    \noindent\textbf{The Case $X = \C P^2 \setminus \operatorname{int} D^4$.}

    Suppose there exists a connected, oriented, compact, properly embedded surface $S \subset \C P^2 \setminus \operatorname{int} D^4$ bounded by $K$ with homology class $[S] = 6$ such that $g(S) \le g_4(K) + 9$. By adding local unknotted tubes to $S$ if necessary, we may assume that $g(S) = g_4(K) + 9$.
    
    We view the pair $(X, S)$ as a knot cobordism from $(S^3, U)$ to $(S^3, K)$. By Corollary~\ref{cor:knot_inequality}, the minimal genus satisfies:
    \begin{equation} \label{eq:genus_inequality}
        g(S) \ge -2b^+(X) - \frac{1}{4}\sigma(X) + \frac{5}{16}[S]^2 - \frac{5}{8}\sigma(K) + \bar{\mu}(\Sigma(K)) + c
    \end{equation}
    where $c$ is the correction term defined in Remark~\ref{rem:correction_term}.
    Substituting $b^+(X) = 1$, $\sigma(X) = 1$, $[S]^2 = 36$, and $g(S) = g_4(K) + 9$ into \eqref{eq:genus_inequality} yields:
    \[
        g_4(K) + 9 \ge -2 - \frac{1}{4} + \frac{45}{4} - \frac{5}{8}\sigma(K) + \bar{\mu}(\Sigma(K)) + c,
    \]
    which simplifies to:
    \begin{equation}\label{eq:genus_simplified_1}
        g_4(K) + \frac{5}{8}\sigma(K) - \bar{\mu}(\Sigma(K)) - c \ge 0.
    \end{equation}
    
    On the other hand, by Lemma~\ref{lem:KMT_double_branch}, the invariants of the double branched cover $\Sigma(S)$ are given by:
    \begin{align*}
        \sigma(\Sigma(S)) &= 2 - \frac{1}{2}\cdot 6^2 - 0 + \sigma(K) = -16 + \sigma(K), \\
        b^+(\Sigma(S)) &= 2 + (g_4(K) + 9) - \frac{1}{4}\cdot 6^2 - 0 + \frac{1}{2}\sigma(K) = g_4(K) + \frac{1}{2}\sigma(K) + 2.
    \end{align*}
    
    \textit{Case: $X = \C P^2 \setminus \operatorname{int} D^4$ and $K = T(3, 5)$.} \\
    For the torus knot $T(3, 5)$, we have $\sigma(K) = -8$ and $g_4(K) = 4$. From the above equations, we compute $\sigma(\Sigma(S)) = -24$ and $b^+(\Sigma(S)) = 2 > 0$. Thus, the branched cover $\Sigma(S)$ is an indefinite spin $4$-manifold. Its boundary $Y = \Sigma(2, 3, 5)$ has $\bar{\mu}(Y) = -1$ and $\delta(Y) = 1$.
    Since $-\bar{\mu}(Y) = 1 = \delta(Y)$ and $-\frac{1}{8}\sigma(\Sigma(S)) + \bar{\mu}(Y) = 3 - 1 = 2 \equiv 2 \pmod 8$, we have $c = 1$ by Theorem~\ref{thm:main_inequality_ko}. 
    By \eqref{eq:genus_simplified_1}, we obtain:
    \[
        4 + \frac{5}{8}(-8) - (-1) - 1 \ge 0 \implies -1 \ge 0.
    \]
    This is a contradiction.

    \textit{Case: $X = \C P^2 \setminus \operatorname{int} D^4$ and $K = T(3, 7)$.} \\
    For $K = T(3, 7)$, we have $\sigma(K) = -8$ and $g_4(K) = 6$. The boundary $Y = \Sigma(2, 3, 7)$ has $\bar{\mu}(Y) = 1$ and $\delta(Y) = 0$. We compute $\sigma(\Sigma(S)) = -16 + (-8) = -24$ and $b^+(\Sigma(S)) = 6 - 4 + 2 = 4 > 0$. Thus $\Sigma(S)$ is indefinite.
    Since $-\bar{\mu}(Y) = -1 = \delta(Y) - 1$ and $-\frac{1}{8}\sigma(\Sigma(S)) + \bar{\mu}(Y) = 3 + 1 = 4 \equiv 4 \pmod 8$, we have $c = 1$. By \eqref{eq:genus_simplified_1}, we obtain:
    \[
        6 + \frac{5}{8}(-8) - 1 - 1 \ge 0 \implies -1 \ge 0.
    \]
    This is a contradiction.

    \textit{Case: $X = \C P^2 \setminus \operatorname{int} D^4$ and $K = T(3, 11)$.} \\
    For $K = T(3, 11)$, we have $\sigma(K) = -16$ and $g_4(K) = 10$. The boundary $Y = \Sigma(2, 3, 11)$ has $\bar{\mu}(Y) = 0$ and $\delta(Y) = 1$. We compute $\sigma(\Sigma(S)) = -16 + (-16) = -32$ and $b^+(\Sigma(S)) = 10 - 8 + 2 = 4 > 0$. Thus $\Sigma(S)$ is indefinite.
    Since $-\bar{\mu}(Y) = 0 = \delta(Y) - 1$ and $-\frac{1}{8}\sigma(\Sigma(S)) + \bar{\mu}(Y) = 4 + 0 = 4 \equiv 4 \pmod 8$, we have $c = 1$. By \eqref{eq:genus_simplified_1}, we obtain:
    \[
        10 + \frac{5}{8}(-16) - 0 - 1 \ge 0 \implies -1 \ge 0.
    \]
    This is a contradiction.

    \vspace{0.5em}
    \noindent\textbf{The Case $X = (\C P^2 \# \C P^2) \setminus \operatorname{int} D^4$.}
    
    We now suppose there exists a connected, oriented, compact, properly embedded surface $S \subset (\C P^2 \# \C P^2) \setminus \operatorname{int} D^4$ bounded by $K$ with homology class $[S] = (6, 6)$ such that $g(S) = g_4(K) + 19$. 
    Here, $b^+(X) = 2$, $\sigma(X) = 2$, and $[S]^2 = 72$. Substituting these into~\eqref{eq:genus_inequality} yields:
    \[
        g_4(K) + 19 \ge -2 \cdot 2 - \frac{1}{4} \cdot 2 + \frac{5}{16} \cdot 72 - \frac{5}{8}\sigma(K) + \bar{\mu}(\Sigma(K)) + c,
    \]
    which simplifies to:
    \begin{equation}\label{eq:genus_simplified_2}
        g_4(K) + \frac{5}{8}\sigma(K) + 1 - \bar{\mu}(\Sigma(K)) - c \ge 0.
    \end{equation}
    For this manifold, the invariants of the double branched cover are given by $\sigma(\Sigma(S)) = 2 \cdot 2 - \frac{72}{2} + \sigma(K) = -32 + \sigma(K)$ and $b^+(\Sigma(S)) = 2 \cdot 2 + (g_4(K) + 19) - \frac{72}{4} + \frac{1}{2}\sigma(K) = g_4(K) + \frac{1}{2}\sigma(K) + 5$.

    \textit{Case: $X = (\C P^2 \# \C P^2) \setminus \operatorname{int} D^4$ and $K = T(3, 5)$.} \\
    For $K = T(3, 5)$, we compute $\sigma(\Sigma(S)) = -32 + (-8) = -40$ and $b^+(\Sigma(S)) = 4 - 4 + 5 = 5 > 0$. Thus $\Sigma(S)$ is indefinite.
    Since $-\bar{\mu}(Y) = 1 = \delta(Y)$ and $-\frac{1}{8}\sigma(\Sigma(S)) + \bar{\mu}(Y) = 5 - 1 = 4 \equiv 4 \pmod 8$, we have $c = 2$. By \eqref{eq:genus_simplified_2}, we obtain:
    \[
        4 + \frac{5}{8}(-8) + 1 - (-1) - 2 \ge 0 \implies -1 \ge 0.
    \]
    This is a contradiction.

    \textit{Case: $X = (\C P^2 \# \C P^2) \setminus \operatorname{int} D^4$ and $K = T(3, 7)$.} \\
    For $K = T(3, 7)$, we compute $\sigma(\Sigma(S)) = -32 + (-8) = -40$ and $b^+(\Sigma(S)) = 6 - 4 + 5 = 7 > 0$. Thus $\Sigma(S)$ is indefinite.
    Since $-\bar{\mu}(Y) = -1 = \delta(Y) - 1$ and $-\frac{1}{8}\sigma(\Sigma(S)) + \bar{\mu}(Y) = 5 + 1 = 6 \equiv 6 \pmod 8$, we have $c = 2$. By \eqref{eq:genus_simplified_2}, we obtain:
    \[
        6 + \frac{5}{8}(-8) + 1 - 1 - 2 \ge 0 \implies -1 \ge 0.
    \]
    This is a contradiction.

    \textit{Case: $X = (\C P^2 \# \C P^2) \setminus \operatorname{int} D^4$ and $K = T(3, 11)$.} \\
    For $K = T(3, 11)$, we compute $\sigma(\Sigma(S)) = -32 + (-16) = -48$ and $b^+(\Sigma(S)) = 10 - 8 + 5 = 7 > 0$. Thus $\Sigma(S)$ is indefinite.
    Since $-\bar{\mu}(Y) = 0 = \delta(Y) - 1$ and $-\frac{1}{8}\sigma(\Sigma(S)) + \bar{\mu}(Y) = 6 + 0 = 6 \equiv 6 \pmod 8$, we have $c = 2$. By \eqref{eq:genus_simplified_2}, we obtain:
    \[
        10 + \frac{5}{8}(-16) + 1 - 0 - 2 \ge 0 \implies -1 \ge 0.
    \]
    This is a contradiction.
    
    This completes the proof of Example~\ref{ex:CP^2_genus}.
\end{proof}

\begin{remark}
    Note that Theorem~\ref{thm:main_inequality} requires $p + \bar{\mu}(Y, \s) \ge 4$. Since the case $X = \C P^2 \setminus \operatorname{int} D^4$ and $K = T(3, 5)$ yields $p + \bar{\mu}(Y, \s) = 2$, Theorem~\ref{thm:main_inequality_ko} is more suitable for obtaining the desired bounds.
\end{remark}

As mentioned in Remark~\ref{rem:closed_surface}, we can derive the sharp lower bounds necessary for Example~\ref{ex:CP^2_genus} using \cite{EG22}, Lemma~\ref{lem:KMT_double_branch} (\cite[Lemma 4.2]{KMT25}), and \cite[Theorem 1.12]{HLSX22}. 
Let $X = \C P^2 \setminus \operatorname{int} D^4$ and fix a knot $K \in \{T(3, 5), T(3, 7), T(3, 11)\}$. Suppose that $S \subset X$ is a connected, oriented, compact, and properly embedded surface representing the homology class $[S] = 6 \in \Z = H_2(X; \Z)$ such that $\partial S = K$. 
According to the arguments in \cite{EG22} (specifically, the proof of Lemma 6.6 in Appendix A for $T(3, 5)$, the proof of Theorem 1.19 for $T(3, 7)$, and the proof of Lemma 6.2 for $T(3, 11)$), there exists a connected, oriented, compact, and properly embedded surface $S' \subset \C P^2 \setminus \operatorname{int} D^4$ representing the homology class $[S'] = 6 \in \Z = H_2(\C P^2 \setminus \operatorname{int} D^4; \Z)$ such that $\partial S' = -K$ and $g(S') = 10 - g_4(K)$.
By taking the boundary connected sum of $S$ and $S'$, and capping off $K \mathop{\#} (-K)$ with a slice disk in $D^4$, we obtain a connected, oriented, closed surface $F \subset \C P^2 \mathop{\#} \C P^2$ representing the homology class $[F] = (6, 6) \in \Z \oplus \Z = H_2(\C P^2 \mathop{\#} \C P^2; \Z)$ with genus $g(F) = g(S) + 10 - g_4(K)$. 
Applying Lemma~\ref{lem:KMT_double_branch} to the unknot, we obtain the topological invariants of the double branched cover $\Sigma(F)$ as $\sigma(\Sigma(F)) = -32$ and $b^+(\Sigma(F)) = g(F) - 14$. 
Then, applying \cite[Theorem 1.12]{HLSX22} in the case where $p \equiv 2 \pmod 8$, we obtain $g(F) - 14 \ge 4 + 2$, which implies $g(F) \ge 20$. 
Since $g(F) = g(S) + 10 - g_4(K)$, we conclude that $g(S) \ge g_4(K) + 10$.
The proof for case (2) of Example~\ref{ex:CP^2_genus}, where $X = (\C P^2 \mathop{\#} \C P^2) \setminus \operatorname{int}D^4$, follows a similar argument; in this case, we obtain $g(S) + 10 - g_4(K) \ge 30$, which yields $g(S) \ge g_4(K) + 20$.

\begin{remark}
    When proving case (1) of Example~\ref{ex:CP^2_genus}, the constraint on the genus of the closed surface can alternatively be deduced from \cite[Theorem 1.6]{Bry98}, instead of applying Lemma~\ref{lem:KMT_double_branch} and \cite[Theorem 1.12]{HLSX22}.
\end{remark}

\begin{remark}\label{rem:philosophy}
    We note that classical sharp bounds, such as Bryan's inequality \cite{Bry98}, rely on the geometric information of the covering involution acting on the $3$-manifold. In contrast, our constraints in Example~\ref{ex:CP^2_genus} are derived without using this specific geometric data. Surprisingly, even after forgetting how the involution acts on the $3$-manifold, our method still provides sharp lower bounds.
\end{remark}

However, there exist infinitely many torus knots that do not satisfy conditions (1) and (2) of Example~\ref{ex:CP^2_genus}. By \cite[Theorem 2.8]{MMRS24}, we have the following:

\begin{proposition}\label{prop:count_ex}
    For any integer $n \ge 19$, there exists a connected, oriented, compact, and properly embedded surface $S \subset \C P^2 \setminus \operatorname{int}D^4$ representing the homology class $[S] = 6 \in \Z = H_2(\C P^2 \setminus \operatorname{int}D^4; \Z)$ such that $\partial S = T(n, n - 1)$ and $g(S) < g_4(T(n, n-1)) + 10$. Moreover, there exists a connected, oriented, compact, and properly embedded surface $S \subset (\C P^2 \mathop{\#} \C P^2) \setminus \operatorname{int}D^4$ representing the homology class $[S] = (6, 6) \in \Z \oplus \Z = H_2((\C P^2 \mathop{\#} \C P^2) \setminus \operatorname{int}D^4; \Z)$ such that $\partial S = T(n, n - 1)$ and $g(S) < g_4(T(n, n - 1)) + 20$.
\end{proposition}

\begin{proof}
    Suppose that $n$ is even. By \cite[Theorem 2.8]{MMRS24}, there exists a connected, oriented, compact, and properly embedded surface $S \subset \C P^2 \setminus \operatorname{int}D^4$ of degree $6$ such that $\partial S = T(n, n - 1)$ and
    \[
    g(S) = \left( \frac{n + 6 - 2}{2} \right)^2 + \left( \frac{n - 6 - 2}{2} \right)^2 = \frac{1}{2}(n^2 - 4n + 40).
    \]
    Recall that $g_4(T(n, n - 1)) + 10 = \frac{1}{2}(n - 1)(n - 2) + 10 = \frac{1}{2}(n^2 - 3n + 22)$. Then, we have
    \[
    g_4(T(n, n - 1)) + 10 - g(S) = \frac{1}{2}\big((n^2 - 3n + 22) - (n^2 - 4n + 40) \big) = \frac{1}{2}(n - 18) > 0,
    \]
    since $n \ge 19$. 
    For the case of $(\C P^2 \mathop{\#} \C P^2) \setminus \operatorname{int}D^4$, by taking the connected sum of $S$ and a smooth closed curve of degree 6 in the second summand, we obtain a connected, oriented, compact, and properly embedded surface $S' \subset (\C P^2 \mathop{\#} \C P^2) \setminus \operatorname{int}D^4$ representing the homology class $[S'] = (6, 6) \in \Z \oplus \Z = H_2((\C P^2 \mathop{\#} \C P^2) \setminus \operatorname{int}D^4; \Z)$ such that $\partial S' = T(n, n - 1)$ and 
    \[
    g(S') = \frac{1}{2}(n^2 - 4n + 40) + 10 = \frac{1}{2}(n^2 - 4n + 60).
    \]
    Since $g_4(T(n, n - 1)) + 20 = \frac{1}{2}(n^2 - 3n + 42)$, we obtain
    \[
    g_4(T(n, n - 1)) + 20 - g(S') = \frac{1}{2}\big((n^2 - 3n + 42) - (n^2 - 4n + 60) \big) = \frac{1}{2}(n - 18) > 0.
    \]

    Now suppose that $n$ is odd. By \cite[Theorem 2.8]{MMRS24}, there exists a connected, oriented, compact, and properly embedded surface $S \subset \C P^2 \setminus \operatorname{int}D^4$ of degree $6$ such that $\partial S = T(n, n - 1)$ and
    \[
    g(S) = \left( \frac{n + 6 - 1}{2} \right) \left( \frac{n + 6 - 3}{2} \right) + \left( \frac{n - 6 - 1}{2} \right) \left( \frac{n - 6 - 3}{2} \right) = \frac{1}{2}(n^2 - 4n + 39).
    \]
    Then, we have
    \[
    g_4(T(n, n - 1)) + 10 - g(S) = \frac{1}{2}\big((n^2 - 3n + 22) - (n^2 - 4n + 39) \big) = \frac{1}{2}(n - 17) > 0,
    \]
    since $n \ge 19$. 
    For the case of $(\C P^2 \mathop{\#} \C P^2) \setminus \operatorname{int}D^4$, a similar connected sum argument yields the desired surface $S'$.
\end{proof}

These results naturally lead to the following question.

\begin{question}\label{q:knot}
    Is there a symplectic geometric characterization of quasi-positive knots that satisfy conditions (1) and (2) of Example~\ref{ex:CP^2_genus}?
\end{question}

\subsection{Splitting $4$-manifolds}

In this subsection, we provide the proof of Theorem~\ref{thm:4-manifold_split}. Recall the decomposition~\eqref{eq:X_decomp} from the introduction. Bauer's strategy for the $11/8$-conjecture involves splitting a hypothetical counterexample $X$ into pieces $X_i$ along integral homology spheres $Y_i$. While Manolescu \cite{Man14} proved the non-existence of decompositions where all $Y_i$ are Floer $K_G$-split, we prove an analogous result for AR homology spheres.

We recall the definition of the Floer $K_G$-split property.

\begin{definition}
    Let $Z$ be a space of type SWF at an even level $l$. Then $Z$ is said to be $K_G$-split if the image of $\K_G(Z) \to \K_G(Z^{S^1}) \cong R(G)$ is generated by $z^k$ for some integer $k$, where the isomorphism is the Bott isomorphism and $R(G) = \Z[z, w] / (w^2 - 2w, zw - 2w)$ is the representation ring of $G$.

    A spin rational homology sphere $(Y, \s)$ is called Floer $K_G$-split if, when we write $\SWF(Y, \s) = (Z, m, n)$, either $Z$ or $\Sigma^{\tilde{\R}}Z$ is $K_G$-split.
\end{definition}

The following proposition characterizes the Floer $K_G$-split property for AR homology spheres.

\begin{proposition}\label{prop:AR_K_G}
    Let $(Y, \s)$ be a spin AR homology sphere. Then $(Y, \s)$ is Floer $K_G$-split if and only if $\delta(Y, \s) = -\bar\mu(Y, \s)$. In particular, there are infinitely many AR homology spheres that are not Floer $K_G$-split.
\end{proposition}

\begin{proof}
    We write $\SWF(Y, \s) = (Z, 0, h)$ as in Construction~\ref{con:lattice_spectrum}, where $Z$ is a space of type SWF at level $0$. Throughout the proof, we set $m = h - \frac{1}{2}\bar\mu(Y, \s) \in \Z$.
    
    First, suppose that $\delta(Y, \s) = -\bar\mu(Y, \s)$. By Theorem~\ref{thm:admissible_map} and Lemma~\ref{lem:can_inclusion}, there exist admissible maps represented by local maps
    \[
    f: Z \to S^{m\mathbb{H}} \quad \text{and} \quad g: S^{m\mathbb{H}} \to Z.
    \]
    Applying $\K_G$-theory and restricting to the $S^1$-fixed-point sets, we obtain the following commutative diagram:
    \[
    \begin{tikzcd}
        \K_{G}(Z) \arrow[r, "g^*"] \arrow[d]
        &\K_{G}(S^{m\mathbb{H}}) \arrow[r, "f^*"] \arrow[d]
        &\K_{G}(Z) \arrow[d]
        \\
        \K_G(S^0) \arrow[r, equal]
        &\K_{G}(S^0) \arrow[r, equal]
        &\K_G(S^0).
    \end{tikzcd}
    \]
    It follows that the image of $\K_G(Z) \to \K_{G}(S^0) \cong R(G)$ coincides with the image of $\K_{G}(S^{m\mathbb{H}}) \to \K_G(S^0)$, which is $(z^m)$. Thus $(Y, \s)$ is Floer $K_G$-split.

    Conversely, suppose that $\delta(Y, \s) > -\bar\mu(Y, \s)$, and set $n = \delta(Y, \s) + \bar\mu(Y, \s) > 0$. By Lemma~\ref{lem:local_map_AR}, there exists an admissible map $\SWF(Y, \s) \to \Sigma^{-\frac{1}{2}\bar\mu(Y, \s)\mathbb{H}} A_n$ represented by a local map
    \[
    f: Z \to \Sigma^{m\mathbb{H}} A_n.
    \]
    By \cite[Lemma 7.8]{DSS26}, the image of $\K_G(A_n) \to \K_G(S^0) \cong R(G)$ is $(w, z^{\lceil n/2 \rceil})$. By Bott periodicity, the image of $\K_G(\Sigma^{m\mathbb{H}} A_n) \to \K_G(S^0) \cong R(G)$ is given by
    \[
    z^m (w, z^{\lceil n/2 \rceil}) = (z^m w, z^{m + \lceil n/2 \rceil}).
    \]
    Applying $\K_G$-theory to $f$ and restricting to the $S^1$-fixed-point sets, we obtain the commutative diagram
    \[
    \begin{tikzcd}
        \K_{G}(\Sigma^{m\mathbb{H}} A_n) \arrow[r, "f^*"] \arrow[d]
        &\K_{G}(Z) \arrow[d]
        \\
        \K_{G}(S^0) \arrow[r, equal]
        &\K_G(S^0).
    \end{tikzcd}
    \]
    Hence the image $I$ of $\K_{G}(Z) \to \K_G(S^0) \cong R(G)$ contains $z^m w = 2^m w$. Suppose for the sake of contradiction that $(Y, \s)$ is Floer $K_G$-split, so that $I = (z^k)$ for some integer $k$. Considering the ring homomorphism $R(G) \to \Z[w]/(w^2 - 2w) \cong \Z \oplus \Z w$ given by $z \mapsto 2$, the image of $(z^k)$ is contained in $2^k(\Z \oplus \Z w)$. Thus $2^m w \in (z^k)$ implies $k \le m$, and hence $z^m \in I$. By definition of $\kappa$, this yields $\kappa(Y, \s) \le -\bar\mu(Y, \s)$, which contradicts \cite[Corollary 1.3]{DSS26}. Therefore, $(Y, \s)$ is not Floer $K_G$-split.

    Finally, there are infinitely many AR homology spheres with $\delta(Y, \s) > -\bar\mu(Y, \s)$, such as the Brieskorn spheres $\Sigma(2, 3, 12n - 5)$ for $n \ge 1$.
\end{proof}

Now, we prove Theorem~\ref{thm:4-manifold_split}.

\begin{proof}[Proof of Theorem~\ref{thm:4-manifold_split}]
    Suppose there exists a closed, smooth, spin $4$-manifold $X$ with a decomposition of the type~\eqref{eq:X_decomp}, such that all the homology spheres $Y_i$ are AR homology spheres. Let $W_i = X_1 \cup_{Y_1} \cdots \cup_{Y_{i - 1}} X_i$. The intersection form of $W_i$ is $2i (-E_8) \oplus 3i \left( \begin{smallmatrix} 0 & 1 \\ 1 & 0 \end{smallmatrix} \right)$, which implies that the Rokhlin invariant $\mu(Y_i)$ is zero. Thus, $\bar{\mu}(Y_i)$ is even.

    Suppose $\bar{\mu}(Y_1) > 0$; since it is even, we have $\bar{\mu}(Y_1) \ge 2$. Note that the intersection form of $X_1$ is $2 (-E_8) \oplus 3 \left( \begin{smallmatrix} 0 & 1 \\ 1 & 0 \end{smallmatrix} \right)$. By Lemma~\ref{lem:existence_of_admissible}, there exists a Furuta--Mahowald class of level-$(1 + \frac{1}{2}\bar{\mu}(Y_1), 3)$ when $\delta(Y_1) = -\bar{\mu}(Y_1)$, and of level-$(1 + \frac{1}{2}\bar{\mu}(Y_1), 4)$ when $\delta(Y_1) > -\bar{\mu}(Y_1)$. Because $\bar{\mu}(Y_1) \ge 2$, we can restrict this map to $S^{2\mathbb{H}} \subset S^{(1 + \frac{1}{2}\bar{\mu}(Y_1))\mathbb{H}}$. By appropriately composing with the inclusion $S^{3\tilde{\R}} \hookrightarrow S^{4\tilde{\R}}$, this yields a Furuta--Mahowald class of level-$(2, 4)$. This contradicts Theorem~\ref{thm:inequality}. Therefore, we have $\bar{\mu}(Y_1) \le 0$.

    Next, we show that $\bar{\mu}(Y_2) \le 0$. By Theorem~\ref{thm:rel_BF}, the relative Bauer--Furuta invariant of $X_2$ gives rise to an admissible morphism:
    \[
    \Phi: \Sigma^{\mathbb{H}} \SWF(Y_1) \to \Sigma^{3\tilde{\R}} \SWF(Y_2).
    \]
    Following Construction~\ref{con:lattice_spectrum}, we can write $\SWF(Y_1) = (Z_1, 0, -h_1)$ and $\SWF(Y_2) = (Z_2, 0, -h_2)$, where $Z_i$ are spaces of type SWF at level $0$ and $h_i \in \Q$. This morphism can be represented by a map
    \[
    \tilde{\Phi}: \Sigma^{(a + h_1 + 1)\mathbb{H}}\Sigma^{b\tilde{\R}}Z_1 \to \Sigma^{(a + h_2)\mathbb{H}}\Sigma^{(3 + b)\tilde{\R}}Z_2
    \]
    for some integers $a, b \in \Z_{\ge 0}$. Restricting $\tilde{\Phi}$ to the subspace $S^{(-\frac{1}{2}\bar{\mu}(Y_1)-h_1)\mathbb{H}} \hookrightarrow Z_1$ induces an admissible morphism:
    \[
    S^{(-\frac{1}{2}\bar{\mu}(Y_1) + 1)\mathbb{H}} \to \Sigma^{3\tilde{\R}}\SWF(Y_2).
    \]
    Since $-\bar{\mu}(Y_1) \ge 0$ and $\bar{\mu}(Y_1)$ is even, a further restriction to $S^{\mathbb{H}} \hookrightarrow S^{(-\frac{1}{2}\bar{\mu}(Y_1) + 1)\mathbb{H}}$ provides the following morphism:
    \[
    S^{\mathbb{H}} \to \Sigma^{3\tilde{\R}}\SWF(Y_2).
    \]
    Applying the same argument used for $Y_1$, we deduce that $\bar{\mu}(Y_2) \le 0$. By induction, it follows that $\bar{\mu}(Y_i) \le 0$ for all $1 \le i \le r - 1$.

    Finally, we focus on $X_r$. The relative Bauer--Furuta invariant of $X_r$ provides a map:
    \[
    \Sigma^{\mathbb{H}}\SWF(Y_{r - 1}) \to \Sigma^{2\tilde{\R}}S^0.
    \]
    By an argument similar to the above, suitably restricting to the representation spheres yields the following admissible morphism:
    \[
    S^{\mathbb{H}} \to S^{2\tilde{\R}}.
    \]
    This contradicts the proof of Furuta's $10/8$ theorem \cite{Fur01}, which completes the proof.
\end{proof}

\begin{remark}
    We note that extending Theorem~\ref{thm:4-manifold_split} to the case where the homology spheres $Y_i$ are connected sums of AR homology spheres remains out of reach with our current tools. For example, we cannot rule out the case where $Y_1 = \Sigma(2, 3, 11) \mathop{\#} \big(\mathop{\#}^{2n} \Sigma(2, 3, 7)\big)$ for large $n$, whose corresponding spectrum is
    \[
    A_1 \wedge \left( \bigwedge^{2n} \Sigma^{-\frac{1}{2}\mathbb{H}} A_1 \right).
    \]
\end{remark}

\end{document}